\documentclass[12pt, reqno]{amsart}
\usepackage[a4paper, margin=30mm]{geometry}
\usepackage{eucal, graphicx, tikz-cd, amssymb}
\usepackage[colorlinks=true,urlcolor=blue,linkcolor=blue, citecolor=blue]{hyperref}

\let\nc\newcommand

\nc{\cmtd}[1]{\color{blue}{{\fbox{D}} #1}\color{black}}
\nc{\cmtk}[1]{\color{red}{{\fbox{K}} #1}\color{black}}

\newcommand{\mb}[1]{{\mathbb{#1}}}
\newcommand{\mc}[1]{{\mathcal{#1}}}

\newcommand{\mf}[1]{{\mathfrak{#1}}}
\newcommand{\mrm}[1]{{\mathrm{#1}}}
\newcommand{\mit}[1]{{\mathit{#1}}}

\newcommand{\mhm}{\mrm{MHM}}
\newcommand{\HC}{\mrm{HC}}
\newcommand{\Mod}{\mrm{Mod}}
\newcommand{\Db}{{\mc{D}\mathit{b}}}
\newcommand{\shom}{{\mc{H}\mit{om}}}

\renewcommand{\hom}{\mrm{Hom}}
\newcommand{\coh}{\mrm{Coh}}

\newcommand{\mon}{{\mathit{mon}}}

\newcommand{\id}{{\mathrm{id}}}

\DeclareMathOperator{\Gr}{Gr}
\DeclareMathOperator{\spec}{Spec}
\DeclareMathOperator{\coker}{coker}

\DeclareMathOperator*{\Res}{Res}
\DeclareMathOperator{\pro}{Pro}

\DeclareMathOperator{\codim}{codim}

\makeatletter
\newcommand*\bigcdot{\mathpalette\bigcdot@{.5}}
\newcommand*\bigcdot@[2]{\mathbin{\vcenter{\hbox{\scalebox{#2}{$\m@th#1\bullet$}}}}}
\makeatother

\theoremstyle{plain}
\newtheorem{thm}{Theorem}[section]
\newtheorem{conj}[thm]{Conjecture}
\newtheorem{prop}[thm]{Proposition}
\newtheorem{cor}[thm]{Corollary}
\newtheorem{lem}[thm]{Lemma}
\theoremstyle{definition}
\newtheorem{defn}[thm]{Definition}
\newtheorem{rmk}[thm]{Remark}
\newtheorem{notation}[thm]{Notation}

\numberwithin{equation}{section}

\nc{\fA}{{\mathfrak A}}
\nc{\fB}{{\mathfrak B}}
\nc{\fC}{{\mathfrak C}}
\nc{\fD}{{\mathfrak D}}
\nc{\fE}{{\mathfrak E}}
\nc{\fF}{{\mathfrak F}}
\nc{\fG}{{\mathfrak G}}
\nc{\fH}{{\mathfrak H}}
\nc{\fI}{{\mathfrak I}}
\nc{\fJ}{{\mathfrak J}}
\nc{\fK}{{\mathfrak K}}
\nc{\fL}{{\mathfrak L}}
\nc{\fM}{{\mathfrak M}}
\nc{\fN}{{\mathfrak N}}
\nc{\fO}{{\mathfrak O}}
\nc{\fP}{{\mathfrak P}}
\nc{\fQ}{{\mathfrak Q}}
\nc{\fR}{{\mathfrak R}}
\nc{\fS}{{\mathfrak S}}
\nc{\fT}{{\mathfrak T}}
\nc{\fU}{{\mathfrak U}}
\nc{\fV}{{\mathfrak V}}
\nc{\fW}{{\mathfrak W}}
\nc{\fZ}{{\mathfrak Z}}
\nc{\fX}{{\mathfrak X}}
\nc{\fY}{{\mathfrak Y}}
\nc{\fa}{{\mathfrak a}}
\nc{\fb}{{\mathfrak b}}
\nc{\fc}{{\mathfrak c}}
\nc{\fd}{{\mathfrak d}}
\nc{\fe}{{\mathfrak e}}
\nc{\ff}{{\mathfrak f}}
\nc{\fg}{{\mathfrak g}}
\nc{\fh}{{\mathfrak h}}
\nc{\fiI}{{\mathfrak i}}  
\nc{\ffi}{{\mathfrak i}}  
\nc{\fj}{{\mathfrak j}}
\nc{\fk}{{\mathfrak k}}
\nc{\fl}{{\mathfrak{l}}}
\nc{\fm}{{\mathfrak m}}
\nc{\fn}{{\mathfrak n}}
\nc{\fo}{{\mathfrak o}}
\nc{\fp}{{\mathfrak p}}
\nc{\fq}{{\mathfrak q}}
\nc{\fr}{{\mathfrak r}}
\nc{\fs}{{\mathfrak s}}
\nc{\ft}{{\mathfrak t}}
\nc{\fu}{{\mathfrak u}}
\nc{\fv}{{\mathfrak v}}
\nc{\fw}{{\mathfrak w}}
\nc{\fz}{{\mathfrak z}}
\nc{\fx}{{\mathfrak x}}
\nc{\fy}{{\mathfrak y}}

\nc{\cB}{{\mathcal B}}
\nc{\cD}{{\mathcal D}}
\nc{\cM}{{\mathcal M}}
\nc{\cN}{{\mathcal N}}
\nc{\cO}{{\mathcal O}}
\nc{\cI}{{\mathcal I}}

\nc{\bC}{{\mathbb C}}
\nc{\bR}{{\mathbb R}}
\nc{\bZ}{{\mathbb Z}}
\nc{\bD}{{\mathbb D}}
\nc{\bL}{{\mathbb L}}

\newcommand{\Hom}{\operatorname{Hom}}

\title[Unitary representations and localization for Hodge modules]{Unitary representations of real groups and localization theory for Hodge modules}
\author{Dougal Davis}
\author{Kari Vilonen}

\dedicatory{Dedicated to Wilfried Schmid on the occasion of his 80th birthday}

\begin{document}

\begin{abstract}
We prove a conjecture of Schmid and the second named author \cite{SV} that the unitarity of a representation of a real reductive Lie group with real infinitesimal character can be read off from a canonical filtration, the Hodge filtration. Our proof rests on three main ingredients. The first is a wall crossing theory for mixed Hodge modules: the key result is that, in certain natural families, the Hodge filtration varies semi-continuously with jumps controlled by extension functors. The second ingredient is a Hodge-theoretic refinement of Beilinson-Bernstein localization: we show that the Hodge filtration of a mixed Hodge module on the flag variety satisfies the usual cohomology vanishing and global generation properties enjoyed by the underlying $\mc{D}$-module. The third ingredient is an explicit calculation of the Hodge filtration on a tempered Hodge module. As byproducts of our work, we obtain a version of Saito's Kodaira vanishing for twisted mixed Hodge modules, a calculation of the Hodge filtration on a certain object in category $\mc{O}$, and a host of new vanishing results for coherent sheaves on flag varieties.
\end{abstract}

\maketitle

\tableofcontents

\section{Introduction}

In \cite{SV}, Schmid and the second named author proposed a conceptual approach to the notoriously difficult problem of computing the unitary dual of a real reductive Lie group. They observed that an irreducible representation with real infinitesimal character carries a canonical Hodge filtration and conjectured that unitarity of the representation can be read off from the Hodge filtration. We prove this conjecture here. The main ingredients in the proof are new algebro-geometric vanishing theorems and deformation arguments for Hodge modules, as well as explicit calculations of the Hodge filtrations on several modules of interest.

The problem of determining the unitary dual of a Lie group, i.e., the set of its irreducible unitary representations, has a long history, going back at least to the 1930s; see, for example, \cite[\S 1]{ALTV} for a brief overview. The solution is known for nilpotent and (reasonable) solvable groups. This leaves the real reductive groups as the principal open case. We will adopt the convention that a real reductive group is a finite cover of an open subgroup of a linear reductive group, the latter being a real form of a connected complex reductive group.

We will not attempt to recall here all that is known about the unitary duals of real groups. Suffice it to say that a number of cases are known by explicit calculation (e.g., for $\mrm{GL}_n(F)$, $F = \mb{R}, \mb{C}, \mb{H}$ \cite{vogan-gln}, complex classical groups \cite{barbasch}, and several others) and that there is an algorithm, developed by Adams, van Leeuwen, Trapa and Vogan \cite{ALTV} (building on several decades of work by many authors) that can in principle compute the lists of unitary representations for linear real groups. This algorithm has been implemented effectively in the {\tt atlas} software, although computational resources become an issue for large examples.

Nevertheless, we are still far from a complete understanding of the unitary dual, even in the linear case. At present, there is not even a precise conjecture as to which representations are unitary in general. In the 1960s a conceptual approach to the problem was proposed, the orbit method, whose main idea is to produce the unitary representations by quantizing co-adjoint orbits.  It has turned out to be difficult to implement, however, and has so far served more as a guiding principle. 

The motivation behind \cite{SV} was to provide a different conceptual geometric approach to the problem of the unitary dual that could potentially overcome the limitations of the existing methods. It was proposed that a representation carries an (infinite dimensional) Hodge structure and that this Hodge structure should be obtained from Morihiko Saito's theory of mixed Hodge modules via Beilinson-Bernstein localization. Once the problem is cast within this framework many tools and techniques, in particular functoriality, can be used which were not available before. For a gentle introduction, see \cite{SV2}, where the theory is worked out for $\mrm{SL}_2(\mb{R})$. 

In this paper, we carry out enough of this program to obtain the desired results about unitarity: the main result in this direction is Theorem \ref{thm:intro unitarity criterion} below, which gives a complete characterization of unitary representations in terms of Hodge theory. Our unitarity criterion works for arbitrary real groups, not necessarily linear. En route to this result, we also prove several other results of both algebro-geometric and representation-theoretic interest: a general theorem on deformations and wall-crossing for mixed Hodge modules (Theorem \ref{thm:intro semi-continuity}), a generalization of Saito's Kodaira vanishing (Theorem \ref{thm:intro twisted kodaira}), a Hodge-theoretic version of Beilinson-Bernstein localization (Theorems \ref{thm:intro filtered exactness} and \ref{thm:intro full faithfulness}) and calculations of the Hodge filtrations on certain representations of particular importance (Theorems \ref{thm:intro xi hodge} and \ref{thm:intro tempered}). We explain these results in more detail below.

Let us begin with our most general results for mixed Hodge modules on arbitrary complex varieties. The general theory of mixed Hodge modules, constructed by Saito \cite{S1, S2}, is a vast generalization of classical Hodge theory. The objects are $\mc{D}$-modules equipped with extra structures, such as a Hodge filtration $F_\bullet$ and a weight filtration $W_\bullet$. These satisfy remarkable strictness properties, as well as all the good functoriality properties of holonomic $\mc{D}$-modules. With an eye on Beilinson-Bernstein localization, our primary interest here is in \emph{twisted} mixed Hodge modules. These can be thought of informally as mixed Hodge modules twisted by an $\mb{R}$-line bundle, or somewhat more formally as twisted $\mc{D}$-modules with Hodge structure; we give a precise definition in \S\ref{subsec:monodromic}. For example, if $\mc{L}$ is a line bundle on a smooth variety $X$, then for every $a \in \mb{R}$ we have an associated category
\[ \mhm_{\mc{L}^a}(X) =\mhm_a(\mc{L}^\times)\]
of $\mc{L}^{a}$-twisted mixed Hodge modules on $X$. When $a \in \mb{Z}$, the twisted theory reduces to
\[ \mhm_{\mc{L}^a}(X) = \{\mc{M} \otimes \mc{L}^a \mid \mc{M} \in \mhm(X)\},\]
where $\mhm(X)$ is the usual (untwisted) category of mixed Hodge modules on $X$.

Our first main theorem is a technical tool for controlling the associated graded of a twisted mixed Hodge module with respect to the Hodge filtration, via deformations and wall crossing. Suppose that $j \colon Q \hookrightarrow X$ is an affinely embedded subvariety and $f \in \mrm{H}^0(\bar{Q}, \mc{L})$ is a global equation for the boundary. Then, for $\mc{M} \in \mhm_{\mc{L}^a}(Q)$, there is an associated $1$-parameter deformation $f^s\mc{M} \in \mhm_{\mc{L}^{a + s}}(Q)$ for $s \in \mb{R}$. Pushing this forward, one obtains $\mc{L}^{a + s}$-twisted mixed Hodge modules $j_!f^s\mc{M}$ and $j_*f^s\mc{M}$ on $X$, such that $j_!f^s\mc{M} = j_*f^s\mc{M}$ for $s$ outside a discrete set. Our first main theorem describes the behavior of the associated gradeds as we vary $s$.

\begin{thm} \label{thm:intro semi-continuity}
The Hodge filtrations on $j_!f^s\mc{M}$ and $j_*f^s\mc{M}$ are semi-continuous in the sense that there are canonical isomorphisms
\[ \Gr^F j_!f^{s - \epsilon} \mc{M} \cong \Gr^F j_!f^s\mc{M} \quad \text{and} \quad \Gr^Fj_*f^{s + \epsilon} \mc{M} \cong \Gr^Fj_* f^s\mc{M} \]
for $0 < \epsilon \ll 1$.
\end{thm}

Theorem \ref{thm:intro semi-continuity} is proved in a slightly more general context in \S\ref{sec:deformations} as Theorem \ref{thm:semi-continuity}. By a series of wall-crossings, it reduces the problem of relating the $\Gr^F j_!f^s\mc{M}$ for different $s$ to the well-studied problem of relating $j_!f^s\mc{M}$ to $j_*f^s\mc{M}$. As well as being a key tool in our study of localization theory and unitary representations, we believe that Theorem \ref{thm:intro semi-continuity} should be of interest far beyond the scope of the present work, such as in singularity theory, where Hodge filtrations on twisted $\mc{D}$-modules of the form $j_*f^s\mc{O}$ are emerging as a powerful invariant (see, e.g., \cite{mustata-popa, schnell-yang}). Theorem \ref{thm:intro semi-continuity} also leads to a short and elegant proof of the following twisted version of Saito's Kodaira vanishing for mixed Hodge modules \cite[Proposition 2.33]{S2}.

\begin{thm} \label{thm:intro twisted kodaira}
Assume that $X$ is projective, $\mc{L}$ is ample, $a > 0$ and $\mc{M} \in \mhm_{\mc{L}^a}(X)$. Then
\[ \mb{H}^i(X, \Gr^F_p\mrm{DR}(\mc{M})) = 0 \quad \text{for $i > 0$ and all $p$},\]
where
\[ \Gr_p^F \mrm{DR}(\mc{M}) = [\Gr_p^F \mc{M} \to \Gr_{p + 1}^F\mc{M} \otimes \Omega^1_X \to \cdots \to \Gr_{p + \dim X}^F\mc{M} \otimes \Omega_X^{\dim X}].\]
\end{thm}

Setting $a = 1$, Theorem \ref{thm:intro twisted kodaira} recovers Saito's result. Theorem \ref{thm:intro twisted kodaira} is a special case of a more general relative statement (Theorem \ref{thm:twisted kodaira}), which we prove in \S\ref{sec:twisted kodaira}. A version of Theorem \ref{thm:intro twisted kodaira} has been proved independently in recent work of Schnell and Yang \cite{schnell-yang}.

Let us now turn to representation theory. We fix from now on a connected complex reductive group $G$ with Lie algebra $\mf{g}$. We write $\mc{B}$ for the flag variety of $G$ and $\mf{h}$ for the Lie algebra of the abstract Cartan. Beilinson-Bernstein localization \cite{beilinson-ICM, BB1, BB2} associates to each $\lambda \in \mf{h}^*$ a sheaf $\mc{D}_\lambda$ of twisted differential operators on $\mc{B}$ and a pair of adjoint functors
\[ \Gamma \colon \Mod(\mc{D}_\lambda) \to \Mod(U(\mf{g}))_{\chi_\lambda}, \quad \Delta \colon \Mod(U(\mf{g}))_{\chi_\lambda} \to \Mod(\mc{D}_\lambda);\]
here $\Mod(U(\mf{g}))_{\chi_\lambda}$ is the category of $U(\mf{g})$-modules on which the center $Z(U(\mf{g}))$ acts via the scalar
\[ \chi_\lambda \colon Z(U(\mf{g})) \cong S(\mf{h})^W \subset S(\mf{h}) \xrightarrow{\lambda} \mb{C},\]
where $W$ is the Weyl group and the isomorphism above is the Harish-Chandra isomorphism. The right adjoint $\Gamma$ is given by taking global sections of the $\mc{D}_\lambda$-module. Beilinson and Bernstein's main results are that $\Gamma$ is exact whenever $\lambda$ is integrally dominant, and an equivalence if, in addition, $\lambda$ is regular.

Assume now that $\lambda \in \mf{h}^*_\mb{R}$ is real. Then we have a category $\mhm(\mc{D}_\lambda)$ of twisted mixed Hodge modules on $\mc{B}$ with a forgetful functor to filtered $\mc{D}_\lambda$-modules. We prove the following refinements of Beilinson and Bernstein's results for this category.

First, we have a refinement of the exactness theorem for dominant $\lambda$:

\begin{thm} \label{thm:intro filtered exactness}
Let $\lambda \in \mf{h}^*_\mb{R}$ be real dominant and let $\mc{M} \in \mhm(\mc{D}_\lambda)$ with Hodge filtration $F_\bullet \mc{M}$. Then
\[ \mrm{H}^i(\mc{B}, F_p\mc{M}) = 0 \quad \text{for $i > 0$ and all $p$}.\]
Hence, the functor
\[ \Gamma \colon \mhm(\mc{D}_\lambda) \to \Mod(U(\mf{g}), F_\bullet)_{\chi_\lambda} \]
is filtered exact.
\end{thm}

Theorem \ref{thm:intro filtered exactness} is restated as Theorem \ref{thm:filtered exactness} and Corollary \ref{cor:filtered exactness} in \S\ref{sec:localization}, and proved in \S\ref{sec:vanishing}. It was originally stated as a theorem in \cite{SV} without proof\footnote{The argument that the authors of \cite{SV} had in mind does not work.} and answers in the affirmative an old question of Tanisaki \cite[\S 5.1]{tanisaki}.

The vanishing in Theorem \ref{thm:intro filtered exactness} is remarkably strong. To illustrate this, we note the following corollary, obtained by applying Theorem \ref{thm:intro filtered exactness} when $\mc{M}$ is the pushforward of a line bundle on a closed subvariety of $\mc{B}$.

\begin{cor} \label{cor:closed subvariety vanishing}
Let $X \subset \mc{B}$ be a smooth closed subvariety with normal bundle $\mc{N}_{X/\mc{B}}$ and canonical bundle $\omega_X$. Then for any ample line bundle $\mc{L}$ on $\mc{B}$, we have
\[ \mrm{H}^i(X, \mrm{Sym}^p(\mc{N}_{X/\mc{B}}) \otimes \omega_X \otimes \mc{L}) = 0 \quad \text{for all $i>0$ and all $p$}.\]
\end{cor}

Corollary \ref{cor:closed subvariety vanishing} has a similar form to the vanishing theorem of Griffiths \cite[Theorem G]{griffiths-vanishing} for positive vector bundles, although neither result implies the other. In very special cases, Corollary \ref{cor:closed subvariety vanishing} recovers some other well-known vanishing theorems in the literature, such as a version of the main result of \cite{broer} (for $X$ the diagonal inside $\mc{B} \times \mc{B}$) and a slight strengthening of \cite[(2.26)]{schmid-discrete-series} (for $X$ the closed orbit attached to a discrete series representation).

There are two main ingredients in the proof of Theorem \ref{thm:intro filtered exactness}: the twisted Kodaira vanishing theorem above and an explicit calculation of the Hodge filtration on a certain mixed Hodge module $\tilde{\Xi}$, which controls the localization theory for Hodge modules in a precise sense (see \S\ref{subsec:filtered exactness outline}).

\begin{thm} \label{thm:intro xi hodge}
Let $\tilde{\Xi}$ denote the pro-projective cover of the irreducible object associated with the closed orbit in the category of $G$-equivariant monodromic mixed Hodge modules on $\mc{B} \times \mc{B}$ with infinitesimal character $0$. Then
\[ \Gr^F \tilde\Xi \cong \mc{O}_{\mrm{St}},\]
where
\[ \mrm{St} = \tilde{\mf{g}}^* \times_{\mf{g}^*} \tilde{\mf{g}}^*\]
is the Steinberg variety. In particular, the Hodge filtration on $\tilde{\Xi}$ is the filtration generated by its unique $G$-invariant section.
\end{thm}

We give a more precise statement (as Theorem \ref{thm:xitilde hodge filtration}) and proof in \S\ref{subsec:monodromic big projective}. We remark that $\tilde{\Xi}$ is a geometric incarnation of the big pro-projective in category $\mc{O}$, which plays a key role in Soergel theory.

We are grateful to Roman Bezrukavnikov for explaining the statement of Theorem \ref{thm:intro xi hodge} to the second named author many years ago. An independent proof will appear in the forthcoming manuscript \cite{BIR}. We understand that the proof in \cite{BIR} deduces Theorem \ref{thm:intro xi hodge} from a similar description of the Hodge filtration on costandard modules; conversely, Theorem \ref{thm:intro xi hodge} can be used to recover this result, see Remark \ref{rmk:standards}.

In addition to the vanishing of Theorem \ref{thm:intro filtered exactness}, the following theorem guarantees abundance of global sections of the Hodge filtration in many situations.

\begin{thm}[=Theorem \ref{thm:hodge generation}]\label{thm:intro hodge generation}
For $\lambda \in \mf{h}^*_\mb{R}$ dominant and $\mc{M} \in \mhm(\mc{D}_\lambda)$, if the $\mc{D}_\lambda$-module $\mc{M}$ is globally generated, then so is its Hodge filtration $F_\bullet \mc{M}$.
\end{thm}

Theorem \ref{thm:intro hodge generation} is proved in \S\ref{sec:hodge generation}. As with Theorem \ref{thm:intro filtered exactness}, the main ingredients in the proof are Theorems \ref{thm:intro twisted kodaira} and \ref{thm:intro xi hodge}.

As a consequence, we obtain a Hodge analog of Beilinson and Bernstein's equivalence of categories for regular dominant $\lambda$. In the statement below, $\mhm^{\mathit{weak}}(U(\mf{g}))_{\chi_\lambda}$ denotes the category of weak mixed Hodge $U(\mf{g})$-modules defined in \S\ref{subsec:mhm globalization}.

\begin{thm}\label{thm:intro full faithfulness}
Let $\lambda \in \mf{h}^*_\mb{R}$ be real dominant and regular. Then the Hodge globalization functor
\[
\Gamma \colon \mhm(\mc{D}_\lambda) \to \mhm^{\mathit{weak}}(U(\mf{g}))_{\chi_\lambda}
\]
is fully faithful and filtered exact.
\end{thm}

We give a more detailed statement, as Theorem \ref{thm:hodge localization} in \S\ref{subsec:mhm globalization}.

Theorems \ref{thm:intro filtered exactness} and \ref{thm:intro hodge generation} give strong control over the algebraic parts of a mixed Hodge module under global sections. Like classical Hodge theory, however, the theory of Hodge modules also incorporates analysis in an essential way. In particular, every pure Hodge module can be equipped with a \emph{polarization}: a Hermitian form satisfying a natural positivity condition with respect to the Hodge structure. By integration, the polarization produces a Hermitian form on the global sections. The main conjecture of \cite{SV} predicts the signature of this integral on each piece of the Hodge filtration. The conjecture would imply that the functor of Theorem \ref{thm:intro full faithfulness} factors through the subcategory of $\mhm^{\mathit{weak}}(U(\mf{g}))_{\chi_\lambda}$ consisting of $U(\mf{g})$-modules with (possibly infinite dimensional) mixed Hodge structure, as well as results about unitary representations of real groups. While we do not prove the full conjecture here, we prove via a slightly different route that the main consequence for unitary representations of real groups still holds.

Let us now explain the connection to the unitary dual of a real group $G_\mb{R}$, given by a finite covering of a real form of $G$. Recall that to each admissible representation of $G_\mb{R}$ is associated a dense subspace, the corresponding Harish-Chandra $(\mf{g}, K)$-module. This is a complex vector space equipped with compatible algebraic actions of the Lie algebra $\mf{g}$ (the complexification of $\mf{g}_\mb{R} = \mrm{Lie}(G_\mb{R})$) and of the complexification $K$ of a maximal compact subgroup $K_\mb{R} \subset G_\mb{R}$. The irreducible Harish-Chandra modules were completely classified by Harish-Chandra and Langlands. Harish-Chandra also showed that the irreducible unitary representations of $G_\mb{R}$ are in bijection with the irreducible unitary Harish-Chandra modules; here we say that a Harish-Chandra module $V$ is \emph{unitary} if it admits a positive definite $(\mf{g}_\mb{R}, K_\mb{R})$-invariant Hermitian form. So the problem of computing the unitary dual is reduced to the question of which irreducible Harish-Chandra modules are unitary. This can be reduced further to the case of real infinitesimal character (see, for example, \cite[Chapter 16]{knapp1986book}), so we will consider only this case from now on.

By Beilinson-Bernstein localization, each irreducible Harish-Chandra module $V$ with real infinitesimal character arises as the global sections of a unique irreducible $K$-equivariant $\mc{D}_\lambda$-module $\mc{M}$ on $\mc{B}$ with $\lambda \in \mf{h}^*_\mb{R}$ real dominant. Moreover, such a $\mc{D}_\lambda$-module has a canonical choice of lift to a twisted Hodge module in $\mhm(\mc{D}_\lambda)$, so $V$ carries a canonical Hodge filtration $F_\bullet V = \Gamma(F_\bullet \mc{M})$.

It is not difficult to determine which irreducible Harish-Chandra modules admit a non-degenerate $(\mf{g}_\mb{R}, K_\mb{R})$-invariant Hermitian form (see, e.g., Remark \ref{rmk:hermitian}). We call such modules \emph{Hermitian}. In the case of real infinitesimal character, these are precisely those Harish-Chandra modules $V$ that admit a Cartan involution $\theta \colon V \to V$, compatible with the Cartan involution on $\theta$ on $G_\mb{R}$ such that $K_\mb{R} = G_\mb{R}^\theta$. Our main theorem is:

\begin{thm} \label{thm:intro unitarity criterion}
Let $V$ be an irreducible Harish-Chandra module with real infinitesimal character as above and suppose that $V$ is Hermitian. Then the Cartan involution $\theta \colon V \to V$ preserves the Hodge filtration $F_\bullet V$ and $V$ is unitary if and only if $\pm \theta$ acts on $\Gr^F_p V$ by $(-1)^p$ for all $p$.
\end{thm}

Theorem \ref{thm:intro unitarity criterion} is an immediate consequence of the more general conjecture of \cite{SV}. We explain the statement in more detail in \S\ref{subsec:intro harish-chandra} (where we also restate it as Theorem \ref{thm:unitarity criterion}) and give the proof in \S\S\ref{sec:tempered}--\ref{sec:pf of unitarity}. For now, let us briefly recall how the statement arises.

One of the principal innovations in \cite{ALTV}, which gave rise to the {\tt atlas} algorithm, was to replace $\mf{g}_\mb{R}$-invariant forms with $\mf{u}_\mb{R}$-invariant forms, where $\mf{u}_\mb{R}$ is the Lie algebra of a \emph{compact} real form of $G$. In the setting of real infinitesimal characters, non-degenerate $\mf{u}_\mb{R}$-invariant forms always exist, their signs can be normalized (this makes them amenable to recursive calculations) and they are related to the $\mf{g}_\mb{R}$-invariant forms of interest via the Cartan involution $\theta$. In terms of Hodge theory, $\mf{u}_\mb{R}$-invariant forms arise by integrating the polarizations on pure Hodge modules: thus, their signs are fixed automatically and the main conjecture of \cite{SV} predicts their signatures. Theorem \ref{thm:intro unitarity criterion} arises from the corresponding prediction of the signature of the $\mf{g}_\mb{R}$-invariant form.

The proof of Theorem \ref{thm:intro unitarity criterion} roughly follows a sketch of Adams, Trapa and Vogan \cite{ATV}, who observed that it follows rather formally from the methods of \cite{ALTV} assuming a number of conjectures about Hodge modules. The main idea is that, by imitating the {\tt atlas} algorithm, one can prove a weak form of the main conjecture of \cite{SV}:

\begin{thm}[Theorems \ref{thm:hodge and signature K} and \ref{thm:hodge and signature K'}] \label{thm:intro characters}
Let $\lambda \in \mf{h}^*_\mb{R}$ be dominant and $\mc{M} \in \mhm(\mc{D}_\lambda, K)$ be pure of weight $w$. Then
\[ \chi^{\mathit{sig}}(\mc{M}, \zeta) = \zeta^c \chi^H(\mc{M}, \zeta) \mod \zeta^2 - 1\]
where $c = \dim \mc{B} - w$ and $\chi^H$ and $\chi^{\mathit{sig}}$ are the Hodge and signature $K$-characters of $\Gamma(\mc{M})$. If $\mc{M}$ is equivariant with respect to the extended group $K \times \{1, \theta\}$, then the same is true for the extended Hodge and signature characters.
\end{thm}

The characters $\chi^H$ and $\chi^{\mathit{sig}}$ above record the multiplicities of the $K$-types (and, in the ``extended group'' setting, the action of the Cartan involution $\theta$) as well as their Hodge degrees and signatures with respect to the polarization respectively; so Theorem \ref{thm:intro characters} says that the Hodge filtration determines the signature. We give the proof in \S\ref{sec:pf of unitarity}. Theorem \ref{thm:intro characters} is weaker than the main conjecture of \cite{SV}, since it concerns only Harish-Chandra modules and does not assert non-degeneracy of the polarization restricted to the Hodge filtration, but is nevertheless strong enough to imply Theorem \ref{thm:intro unitarity criterion}.

Very roughly, Theorem \ref{thm:intro characters} is proved as follows. First, we check the statement by hand when the $(\mf{g}, K)$-module $\Gamma(\mc{M})$ is tempered. Since tempered representations are known to be unitary on general grounds, this boils down to a computation of the Hodge filtration and its associated graded in this case. For example, we have:

\begin{thm}\label{thm:intro tempered}
Let $G_\mb{R}$ be a split group and $\mc{M} \in \mhm(\mc{D}_0, K)$ the irreducible Hodge module such that $V = \Gamma(\mc{M})$ is the tempered spherical principal series representation with real infinitesimal character. Then
\[ \Gr^F V \cong \mc{O}_{\mc{N}^*_K} \quad \text{and} \quad \Gr^F\mc{M} \cong \mc{O}_{\tilde{\mc{N}}^*_K},\]
where $\mc{N}^*_K = \mc{N}^* \cap (\mf{g}/\mf{k})^*$ is the $K$-nilpotent cone and $\tilde{\mc{N}}^*_K$ is its scheme-theoretic pre-image under the Springer resolution.
\end{thm}

In fact, we prove a somewhat more general version of Theorem \ref{thm:intro tempered} (Theorem \ref{thm:split tempered}), which also includes certain tempered principal series for non-linear groups. The key input in the proof is Theorem \ref{thm:intro hodge generation}, which ensures that the lowest piece of the Hodge filtration has non-zero global sections. Theorem \ref{thm:intro tempered} implies that the Hodge filtration of a general tempered Harish-Chandra sheaf is always generated by its lowest piece (Theorem \ref{thm:tempered hodge}).

Finally, Theorem \ref{thm:intro characters} is proved for a general irreducible Harish-Chandra module $V$ by deforming it to a tempered one using the techniques of Theorem \ref{thm:intro semi-continuity}. By that theorem, the Hodge $K$-character $\chi^H$ is piecewise constant, with jumps at a discrete set of reducibility walls. By continuity of eigenvalues, the same is true for the signature $K$-character $\chi^{\mathit{sig}}$. Thus, we need only check that the jumps are consistent with Theorem \ref{thm:intro characters}. In \cite{ATV}, it was suggested that this could be done by computing these jumps explicitly in terms of Kazhdan-Lusztig-Vogan polynomials, as in \cite{ALTV}. This can be carried out for linear groups (cf., \cite{DV1}), but the combinatorics for non-linear groups are not so explicitly understood. Instead, we give a more conceptual geometric argument using Theorem \ref{thm:intro semi-continuity} and a mild generalization of \cite[Theorem 1.2]{DV1}.

\subsection{Conventions}

Throughout this paper, all varieties are assumed quasi-projective and defined over the complex numbers. When working with $\mc{D}$-modules, we will always take these to be left algebraic $\mc{D}$-modules unless otherwise specified. When working with mixed Hodge modules, we will always mean the complex Hodge modules of, say, \cite{SS}.

If $X$ is a smooth variety, we will write $\Mod(\mc{D}_X)$ for the category of all quasi-coherent $\mc{D}_X$-modules on $X$ and $\Mod(\mc{D}_X)_{rh}$ for the full subcategory of regular holonomic modules. See also Notation \ref{notation:functors} for our notation for the various pullback and pushforward operations, and \S\ref{subsec:flag variety} for our conventions for reductive groups and their flag varieties.

\subsection{Outline of the paper}

The paper is divided broadly into three parts.

The first part (\S\S\ref{sec:mhm}--\ref{sec:twisted kodaira}) deals with general results about mixed Hodge modules. We begin in \S\ref{sec:mhm} with a brief recollection of the basics of mixed Hodge module theory and the definition of twisted (or more generally, monodromic) mixed Hodge modules. In \S\ref{sec:deformations}, we develop our theory of deformations and wall-crossing for mixed Hodge modules, including the proof of Theorem \ref{thm:intro semi-continuity} and a generalization of \cite[Theorem 1.2]{DV1}. In \S\ref{sec:twisted kodaira}, we prove both the twisted Kodaira vanishing theorem (Theorem \ref{thm:intro twisted kodaira}) and a variant for monodromic mixed Hodge modules.

The second part (\S\S\ref{sec:localization}--\ref{sec:hodge generation}) concerns localization theory for mixed Hodge modules on the flag variety. We state our main results and their consequences in \S\ref{sec:localization} and give the proofs in \S\S\ref{sec:vanishing}--\ref{sec:hodge generation}. The proof of Theorem \ref{thm:intro xi hodge} can also be found in \S\ref{sec:vanishing}.

The third and final part (\S\S\ref{sec:harish-chandra}--\ref{sec:pf of unitarity}) concerns unitary representations of real groups. We explain our unitarity criterion in more detail in \S\ref{sec:harish-chandra} and give an illustrative example of it in action by reproving Vogan's theorem on unitarity of cohomological induction \cite[Theorem 1.3]{vogan-annals}. In \S\ref{sec:tempered}, we study the Hodge theory of tempered representations, proving Theorem \ref{thm:intro tempered} and that tempered representations satisfy the main conjecture of \cite{SV}. Finally, in \S\ref{sec:pf of unitarity}, we run the deformation argument outlined above to deduce Theorem \ref{thm:intro characters} and hence Theorem \ref{thm:intro unitarity criterion}.

\subsection*{Acknowledgements}

We would like to thank Jeff Adams, Sung Gi Park, Wilfried Schmid, Peter Trapa, David Vogan, Ruijie Yang, Shilin Yu and Xinwen Zhu for helpful conversations, and Roman Bezrukavnikov, who explained to the second named author that Theorem \ref{thm:intro xi hodge} should be true many years ago. We would also like to thank Toshiyuki Tanisaki and Qixian Zhao for bringing to our attention some precedents for our vanishing results in the literature.

\section{Twisted and monodromic mixed Hodge modules} \label{sec:mhm}

In this section, we briefly recall the basics of mixed Hodge modules (\S\ref{subsec:mhm}) and their twisted and monodromic variants (\S\ref{subsec:monodromic}).

\subsection{Mixed Hodge modules} \label{subsec:mhm}

In this subsection, we recall some aspects of the theory of mixed Hodge modules, following the treatment in \cite{SS}. For our purposes, the starting point is the notion of polarized Hodge structure:

\begin{defn} \label{defn:polarized hodge str}
Let $w \in \mb{Z}$. A \emph{pure Hodge structure of weight $w$} is a finite dimensional complex vector space $V$ equipped with a \emph{Hodge decomposition}
\[ V = \bigoplus_{p + q = w} V^{p, q}.\]
A \emph{polarized Hodge structure of weight $w$} is a pure Hodge structure $V$ equipped with a Hermitian form $S \colon V \otimes \overline{V} \to \mb{C}$ (the \emph{polarization}) such that
\begin{enumerate}
\item the Hodge decomposition is orthogonal with respect to $S$, and
\item the form $S|_{V^{p, q}}$ is $(-1)^q$-definite.
\end{enumerate}
\end{defn}

Definition \ref{defn:polarized hodge str} gives a robust way to normalize the signs of certain indefinite signature Hermitian forms in the presence of a Hodge structure. This is crucial both in general Hodge theory and in our applications to unitary representation theory.

For a pure Hodge structure $V$, we can recover the Hodge decomposition from the (increasing) \emph{Hodge filtration} and \emph{conjugate Hodge filtration}
\[ F_p V := \bigoplus_{p' \geq -p} V^{p', w - p'} \quad \text{and} \quad \bar{F}_q V := \bigoplus_{q' \geq -q} V^{w - q', q'}\]
via the formula
\[ V^{p, q} = F_{-p} V \cap \bar{F}_{-q} V.\]
If the polarization $S$ and the weight $w$ are fixed, then we have, moreover,
\[\bar{F}_q V = (F_{-w - q - 1} V)^\perp;\]
thus, a polarized Hodge structure is captured by the data $(V, F_\bullet, S, w)$.

Now let $X$ be a smooth variety over $\mb{C}$. A \emph{polarized Hodge module of weight $w$} consists of a triple $(\mc{M}, F_\bullet\mc{M}, S)$, where $\mc{M}$ is a regular holonomic $\mc{D}_X$-module, $F_\bullet \mc{M}$ is a good filtration (the \emph{Hodge filtration}) compatible with the order filtration $F_\bullet \mc{D}_X$ and
\[ S \colon \mc{M} \otimes \overline{\mc{M}} \to \Db_X \]
is a $\mc{D}_X \otimes \overline{\mc{D}}_X$-linear Hermitian form (the \emph{polarization}), valued in the sheaf $\Db_X$ of complex-valued distributions on $X$ regarded as a smooth manifold. Here we view distributions as continuous duals of smooth compactly supported top-dimensional differential forms. Such pairings to distributions were first considered by Kashiwara \cite{kashiwara}.

In order for the data $(\mc{M}, F_\bullet\mc{M}, S)$ to define a polarized Hodge module, it must satisfy a further condition, which amounts to requiring that any specialization to a point produces a polarized Hodge structure. For example, if $\mc{M}$ is a vector bundle, equipped with a flat connection by the $\mc{D}$-module structure, then the polarization $S$ is equivalent to a flat Hermitian form on $\mc{M}$ in the usual sense; in this case, one requires that $(\mc{M}_x, F_\bullet \mc{M}_x, S_x)$ is a polarized Hodge structure of weight $w - \dim X$ for all $x \in X$, i.e., that $(\mc{M}, F_\bullet \mc{M}, S)$ is a polarized variation of Hodge structure. When $\mc{M}$ is singular, the key idea of Saito is to interpret ``specialization to a point'' in terms of nearby and vanishing cycles. We direct the reader to, for example, \cite[Chapter 14]{SS} for the precise definition.

We note here that, for $(\mc{M}, F_\bullet \mc{M}, S)$ a polarized Hodge module, the polarization is always perfect, i.e., it induces an isomorphism
\[ \mc{M} \cong \mc{M}^h \]
at the level of $\mc{D}$-modules, where $\mc{M}^h$ (the \emph{Hermitian dual} of $\mc{M}$) is the regular holonomic $\mc{D}$-module characterized by
\[ (\mc{M}^h)^{\mathit{an}} = \shom_{\overline{\mc{D}}_X}(\overline{\mc{M}}, \Db_X);\]
see also \cite[\S 3.1]{DV1}.

In our study, the polarized Hodge modules are in some sense the main objects of interest: roughly speaking, they are analogous to (although much more general than) unitary representations. Like unitary representations, however, they do not form a very good category on their own. The main idea in mixed Hodge theory is to embed the polarized objects into a much larger abelian category with better formal properties.

First, one relaxes the structure slightly to obtain an abelian category of pure objects (aka polarizable, in this context). There are a few ways to set this up; we will use the following formulation, which is equivalent to the language of ``triples'' of \cite{SS}. In place of the polarization $S$, we remember the filtration
\[ F_\bullet \mc{M}^h := F_{\bullet - w} \mc{M} \]
on the Hermitian dual induced by the isomorphism $S \colon \mc{M} \cong \mc{M}^h$. Over a point, where $\mc{M} = V$ is a Hodge structure, $F_\bullet \mc{M}^h$ is nothing but the orthogonal complement of $\bar{F}_\bullet V$. In general, we say that the data $(\mc{M}, F_\bullet\mc{M}, F_\bullet \mc{M}^h)$ defines a \emph{polarizable Hodge module of weight $w$} if it arises from a polarized Hodge module $(\mc{M}, F_\bullet\mc{M}, S, w)$ in this way.

Note that if $\mc{M}$ is irreducible, then the polarization $S$ is unique up to a positive scalar if it exists; in particular, the sign is fixed by the Hodge structure.

The abelian category of mixed Hodge modules is now built by allowing certain extensions of pure objects of different weights. Over a point, one obtains the category of \emph{mixed Hodge structures}; that is, tuples $(V, F_\bullet, \bar{F}_\bullet, W_\bullet)$, where $V$ is a finite dimensional vector space equipped with finite increasing filtrations $F_\bullet$, $\bar{F}_\bullet$, $W_\bullet$ such that $(\Gr^W_w V, F_\bullet, \bar{F}_\bullet)$ is a pure (=polarizable) Hodge structure of weight $w$ for all $w$. Over a general smooth variety, a \emph{mixed Hodge module} is specified by a tuple
\[ (\mc{M}, F_\bullet \mc{M}, F_\bullet \mc{M}^h, W_\bullet \mc{M}),\]
where $F_\bullet \mc{M}$ and $F_\bullet \mc{M}^h$ are good filtrations as before and $W_\bullet \mc{M}$ is a finite filtration by $\mc{D}$-submodules such that
\[ (\Gr^W_w\mc{M}, F_\bullet \Gr^W_w\mc{M}, F_\bullet (\Gr^W_w\mc{M})^h) \]
is a polarizable Hodge module of weight $w$ for all $w \in \mb{Z}$. The filtration $W_\bullet$ is called the \emph{weight filtration}. Over a point this is the full definition, but on a general $X$ one needs to impose further conditions on the weight filtration near singularities (again relating to nearby and vanishing cycles). For example, when $\mc{M}$ is a vector bundle, a tuple $(\mc{M}, F_\bullet\mc{M}, F_\bullet\mc{M}^h, W_\bullet\mc{M})$ as above is the same thing as a graded polarizable variation of mixed Hodge structure, and the condition that this be a mixed Hodge module is that this variation of mixed Hodge structure be admissible at infinity \cite[Theorem 3.27]{S2} \cite{kashiwara2}.

As an illustrative example, the conventions are such that $(\mc{O}_X, S)$ is a polarized Hodge module of weight $\dim X$, where
\[ F_p \mc{O}_X = \begin{cases} 0, & \mbox{if $p < 0$}, \\ \mc{O}_X, &\mbox{otherwise}\end{cases} \quad \text{and} \quad S(f, \bar{g}) = f\bar{g}.\]
Hence, the triple $(\mc{O}_X, F_\bullet\mc{O}_X, F_\bullet\mc{O}_X^h = F_{\bullet - \dim X} \mc{O}_X)$ is a pure Hodge module of weight $\dim X$.

The abelian category $\mhm(X)$ of mixed Hodge modules on a variety $X$ has many remarkable properties, which we will not attempt to recall here. A fundamental feature is that the standard six functors $(\mb{D}, \boxtimes, f_*, f^*, f_!, f^!)$ for regular holonomic $\mc{D}$-modules (or equivalently, for perverse sheaves) lift to mixed Hodge modules. For example, if $f \colon X \to Y$ is a morphism, then we have functors
\[ f_*, f_! \colon \mrm{D}^b\mhm(X) \to \mrm{D}^b \mhm(Y) \]
and
\[ f^*, f^! \colon \mrm{D}^b\mhm(Y) \to \mrm{D}^b \mhm(X) \]
such that $f^*$ is left adjoint to $f_*$ and $f^!$ is right adjoint to $f_!$. The main upshot is that any $\mc{D}$-module constructed functorially from a polarizable variation of Hodge structure is endowed with a canonical mixed Hodge module structure.

\begin{notation} \label{notation:functors}
In this paper, we will reserve the notation $f_*, f_!, f^*, f^!$ for the above functors coming from the six functor formalism for mixed Hodge modules or regular holonomic $\mc{D}$-modules. For the usual pushforward of sheaves (or $\mc{O}$-modules) along a morphism $f \colon X \to Y$, we will write $f_{\bigcdot}$. We write $f^{-1}$ for the pullback of sheaves and $f^{\bigcdot}$ for the pullback of $\mc{O}$-modules. Finally, when the morphism $f$ is smooth of relative dimension $d$, we write
\[ f^\circ := f^*[d] = f^![-d] \colon \Mod(\mc{D}_Y)_{rh} \to \Mod(\mc{D}_X)_{rh} \]
and
\[ f^\circ := f^*[d] = f^!(-d)[-d] \colon \mhm(Y) \to \mhm(X).\]
where $(-d)$ denotes a Tate twist. Here, for $n \in \mb{Z}$, the Tate twist $(n)$ is defined by $\mc{M}(n) = \mc{M} \otimes \mb{C}_{n, n}$, where
\[ (\mc{M}, F_\bullet \mc{M}, F_\bullet \mc{M}^h, W_\bullet \mc{M})\otimes \mb{C}_{p, q} := (\mc{M}, F_{\bullet - p} \mc{M}, F_{\bullet + q}\mc{M}^h, W_{\bullet + p + q}\mc{M}) \]
is the $(p, q)$-Hodge twist on complex mixed Hodge modules. The conventions are such that $f^\circ \mc{M} = f^{\bigcdot}\mc{M}$ as $\mc{O}_X$-modules, with Hodge and weight filtrations given by
\[ F_pf^\circ \mc{M} = f^{\bigcdot}F_p\mc{M} \quad \mbox{and} \quad W_wf^\circ \mc{M} = f^{\bigcdot}W_{w - d}\mc{M}.\]
For example, we have $f^\circ \mc{O}_Y = \mc{O}_X$ as mixed Hodge modules.
\end{notation}

\subsection{The monodromic setting} \label{subsec:monodromic}

For applications to representation theory, we need to work not only with the sheaf $\mc{D}_X$ of differential operators on a variety $X$, but also with certain sheaves of twisted differential operators and their modules. In this subsection, we explain how the theory of mixed Hodge modules extends to this setting.

In order to transfer standard $\mc{D}$-module (and mixed Hodge module) theory to the twisted setting, it is convenient to introduce an auxiliary space $\tilde{X} \to X$ and define twisted objects in terms of untwisted $\mc{D}$-modules on $\tilde{X}$ as in \cite[\S 2.5]{BB2}. This can be regarded as a formal device, however: we explain below how to view these as modules over a sheaf of rings on the original space $X$. We will often adopt this latter point of view.

Let $\tilde X$ be a smooth variety equipped with an action of an algebraic torus $H$. Recall that a \emph{weakly $H$-equivariant $\mc{D}$-module} on $\tilde X$ is a $\mc{D}_{\tilde X}$-module equipped with an action of $H$ as an $\mc{O}_X$-module such that the map $\mc{D}_{\tilde X} \otimes \mc{M} \to \mc{M}$ is $H$-equivariant. Every such $\mc{M}$ comes equipped with a functorial action of the Lie algebra $\mf{h} = \mrm{Lie}(H)$, commuting with the action of $\mc{D}_{\tilde X}$, given by
\begin{equation} \label{eq:weak difference}
hm := i(h) m - h \cdot m \quad \text{for $h \in \mf{h}$ and $m \in \mc{M}$},
\end{equation}
where $i \colon \mf{h} \to \mc{D}_{\tilde X}$ is the derivative of the $H$-action on $\tilde X$ and $h \cdot m$ denotes the derivative of the $H$-action on $\mc{M}$.

\begin{defn}
Let $\mc{M}$ be a weakly $H$-equivariant $\mc{D}$-module on $\tilde{X}$ and fix $\lambda \in \mf{h}^*$. We say that $\mc{M}$ is \emph{$\lambda$-twisted} (resp., \emph{$\lambda$-monodromic}) if the $\mf{h}$-action \eqref{eq:weak difference} is such that $h - \lambda(h)$ acts by zero (resp., nilpotently) on $\mc{M}$ for all $h \in \mf{h}$. We say that $\mc{M}$ is \emph{monodromic} if it is a direct sum of its generalized eigenspaces under \eqref{eq:weak difference}. We write
\[ \Mod_\lambda(\mc{D}_{\tilde{X}}) \subset \Mod_{\widetilde{\lambda}}(\mc{D}_{\tilde{X}}) \subset \Mod_{\mon}(\mc{D}_{\tilde{X}}) \subset \Mod_H^{\mathit{weak}}(\mc{D}_{\tilde X}) \]
for the categories of $\lambda$-twisted, $\lambda$-monodromic, monodromic and weakly $H$-equivariant $\mc{D}$-modules respectively.
\end{defn}

For fixed $\lambda \in \mf{h}^*$, $\lambda$-monodromicity is in fact a property of the underlying $\mc{D}_{\tilde X}$-module, not extra structure. Indeed, the weak $H$-action, if it exists, can be deduced from the decomposition of local sections of $\mc{M}$ into generalized eigenspaces under $i(\mf{h}) \subset \mc{D}_{\tilde X}$. More generally, we have the following.

\begin{prop}[{\cite[Lemma 2.5.4]{BB2}}] \label{prop:bb monodromic}
The forgetful functor
\[ \mrm{D}^b\Mod_{\widetilde{\lambda}}(\mc{D}_{\tilde{X}}) \to \mrm{D}^b \Mod(\mc{D}_{\tilde{X}}) \]
is fully faithful.
\end{prop}

Note that the analog of Proposition \ref{prop:bb monodromic} for twisted $\mc{D}$-modules is false. For this reason (and others) it is often technically more convenient to work in the larger category of monodromic $\mc{D}$-modules. We will work in this generality throughout most of this paper.

We extend the theory of mixed Hodge modules to the monodromic setting as follows.

\begin{defn}
A mixed Hodge module $\mc{M}$ on $\tilde X$ is \emph{$\lambda$-monodromic} or \emph{$\lambda$-twisted} if the underlying $\mc{D}$-module is so. A \emph{monodromic mixed Hodge module} is a mixed Hodge module on $\tilde X$ equipped with a decomposition
\[ \mc{M} = \bigoplus_{\lambda \in \mf{h}^*} \mc{M}_\lambda,\]
where each $\mc{M}_\lambda$ is a $\lambda$-monodromic mixed Hodge module. We write
\[ \mhm_\lambda(\tilde{X}) \subset \mhm_{\widetilde{\lambda}}(\tilde{X}) \subset \mhm_\mon(\tilde{X})\]
for the categories of $\lambda$-twisted, $\lambda$-monodromic, and monodromic mixed Hodge modules respectively.
\end{defn}

\begin{prop} \label{prop:real twisting}
We have $\mhm_{\widetilde{\lambda}}(\tilde{X}) = 0$ unless $\lambda \in \mf{h}^*_\mb{R} := \mb{X}^*(H) \otimes \mb{R}$.
\end{prop}
\begin{proof}
If $\mhm_{\widetilde{\lambda}}(\tilde{X}) \neq 0$, then there exists a nonzero polarized Hodge module $(\mc{M}, F_\bullet, S)$ such that the underlying $\mc{D}$-module $\mc{M}$ is $\lambda$-monodromic. In particular, $\mc{M}$ is Hermitian self-dual. But the Hermitian dual of a $\lambda$-monodromic $\mc{D}$-module is $\bar{\lambda}$-monodromic, so we must have $\lambda = \bar{\lambda}$.
\end{proof}

In light of Proposition \ref{prop:real twisting}, we will always assume $\lambda \in \mf{h}^*_\mb{R}$ when working with monodromic mixed Hodge modules.

Now suppose that $\tilde{X}$ is an $H$-torsor over another variety $X$ and write $\pi \colon \tilde{X} \to X$ for the quotient map. Consider the sheaf
\[ \tilde{\mc{D}} = \pi_{\bigcdot}(\mc{D}_{\tilde{X}})^H \]
of rings on $X$, a central extension of $\mc{D}_X$ by $S(\mf{h})$. We have an equivalence of categories
\begin{equation} \label{eq:monodromic descent}
\begin{aligned}
\Mod_H^{\mathit{weak}}(\mc{D}_{\tilde X}) &\cong \Mod(\tilde{\mc{D}}) \\
\mc{M} &\mapsto \pi_{\bigcdot}(\mc{M})^H \\
\pi^{\bigcdot} \mc{N} &\mathrel{\reflectbox{$\mapsto$}} \mc{N}
\end{aligned}
\end{equation}
such that action \eqref{eq:weak difference} on the left corresponds to the action of the central subalgebra $S(\mf{h}) \subset \tilde{\mc{D}}$ on the right. We will often identify a monodromic $\mc{D}_{\tilde{X}}$-module with its corresponding $\tilde{\mc{D}}$-module on $X$. Note that \eqref{eq:monodromic descent} identifies $\Mod_\lambda(\mc{D}_{\tilde X})$ with the category of modules over the sheaf of twisted differential operators
\[ \tilde{\mc{D}} \otimes_{S(\mf{h}), \lambda} \mb{C}.\]

Similarly, to any monodromic mixed Hodge module on $\tilde{X}$ we can associate an underlying filtered $\tilde{\mc{D}}$-module on $X$. Let us first consider the case where $H = \mb{C}^\times$ and $\tilde{X} = \mb{C}^\times \times X$ is the trivial $\mb{C}^\times$-torsor over $X$. In this case, we have the functor of nearby cycles
\begin{equation} \label{eq:monodromic nearby cycles}
\psi_{t} \colon \mhm_{\widetilde{\lambda}}(\tilde X) \to \mhm(X, \mrm{N}) := \left\{(\mc{M}, \mrm{N}) \,\left|\, \begin{matrix} \mc{M} \in \mhm(X) \\ \mrm{N} \colon \mc{M}(1) \to \mc{M}\end{matrix}\right.\right\}
\end{equation}
with respect to the coordinate $t$ on $\mb{C}^\times$. At the level of underlying $\mc{D}$-modules, $\psi_{t}$ is precisely the equivalence \eqref{eq:monodromic descent}, where we identify $\tilde{\mc{D}}$ with $\mc{D}_X \otimes S(\mf{h}) = \mc{D}_X[\mrm{N}]$ for $\mrm{N} = \lambda - t \partial_t$.

\begin{lem}[cf., \cite{S3}]\label{lem:monodromic nearby cycles}
The functor \eqref{eq:monodromic nearby cycles} is an equivalence of categories.
\end{lem}
\begin{proof}
Let us write down an explicit inverse. For any $n > 0$, consider the rank $n$ admissible variation of mixed Hodge structure
\[ \frac{t^{\lambda + s}\mc{O}_{\mb{C}^\times}[s]}{(s^n)} \]
on $\mb{C}^\times$, where $t$ denotes the coordinate on $\mb{C}^\times$, $s$ is an additional variable and $t^{\lambda + s}$ is a formal symbol indicating that the $\mc{D}_{\mb{C}^\times} = \mb{C}\langle t, \partial_t \rangle$-module structure is given by the formula
\[ \partial_t t^{\lambda + s} g = t^{\lambda + s}\left(\partial_t g + \frac{\lambda + s}{t} g\right)\]
for $g \in \mc{O}_{\mb{C}^\times}[s]/(s^n)$. The Hodge and weight filtrations are given by
\[ W_w \frac{t^{\lambda + s}\mc{O}_{\mb{C}^\times}[s]}{(s^n)} = \sum_{2k \geq -w} s^k t^{\lambda + s}\mc{O}_{\mb{C}^\times} \quad \text{and} \quad F_p \frac{t^{\lambda + s}\mc{O}_{\mb{C}^\times}[s]}{(s^n)} = \sum_{k \leq p} s^k t^{\lambda + s}\mc{O}_{\mb{C}^\times}.\]
The conjugate Hodge filtration $\bar{F}_\bullet$ is determined uniquely by requiring that the fiber over $1$ be the split mixed Hodge structure $\bigoplus_{k = 0}^{n - 1} \mb{C}(k)$. Explicitly, $\bar{F}_\bullet$ is given on the fiber over $z \in \mb{C}^\times$ by
\[ \bar{F}_q \left.\frac{t^{\lambda + s} \mc{O}_{\mb{C}^\times}[s]}{(s^n)}\right|_z = |z|^{-2s} \sum_{k \leq q} \mb{C}s^k t^{\lambda + s}.\]
The inverse to \eqref{eq:monodromic nearby cycles} is
\begin{equation} \label{eq:monodromic inverse}
(\mc{M}, \mrm{N}) \mapsto \varprojlim_n \coker\left(s + \mrm{N} \colon \frac{t^{\lambda + s}\mc{O}_{\mb{C}^\times}[s]}{(s^n)} \boxtimes \mc{M}(1) \to \frac{t^{\lambda + s}\mc{O}_{\mb{C}^\times}[s]}{(s^n)} \boxtimes \mc{M}\right).
\end{equation}
Here we note that, since $\mrm{N}$ is necessarily nilpotent, the limit on the right hand side of \eqref{eq:monodromic inverse} stabilizes for $n \gg 0$.

Now, one easily checks that $\psi_{t}$ composed with \eqref{eq:monodromic inverse} gives the identity. So $\psi_{t}$ is essentially surjective; it remains to show that $\psi_{t}$ is fully faithful. This follows from the functoriality of mixed Hodge modules as follows. Since the corresponding functor for $\mc{D}$-modules is an equivalence, the map
\begin{equation} \label{eq:monodromic nearby cycles 1}
\Hom_{\Mod(\mc{D}_{\mb{C}^\times \times X})}(\mc{M}, \mc{N}) \to \Hom_{\Mod(\mc{D}_X, \mrm{N})} (\psi_t\mc{M}, \psi_{t} \mc{N})
\end{equation}
is an isomorphism of vector spaces for $\mc{M}, \mc{N} \in \mhm_{\widetilde{\lambda}}(\mb{C}^\times \times X)$. Now, the six functor formalism for mixed Hodge modules endows both sides of \eqref{eq:monodromic nearby cycles 1} with mixed Hodge structures so that \eqref{eq:monodromic nearby cycles 1} is a morphism, and hence an isomorphism, of mixed Hodge structures. Applying $\Hom_{\mrm{MHS}}(\mb{C}, -)$, we obtain the morphism
\[\Hom_{\mhm(\mb{C}^\times \times X)}(\mc{M}, \mc{N}) \to \Hom_{\mhm(X, \mrm{N})} (\psi_{t}\mc{M}, \psi_{t} \mc{N}),\]
so this must be an isomorphism as well.
\end{proof}

We deduce the following useful observation for a general torus $H$:

\begin{lem} \label{lem:monodromic local model}
Restriction to $\{1\} \times X$ defines an equivalence of categories
\[ \mhm_{\widetilde{\lambda}}(H \times X) \overset{\sim}\to \mhm(X, S(\mf{h}(1))),\]
where $\mhm(X, S(\mf{h}(1)))$ is the category of mixed Hodge modules on $X$ equipped with a module structure over $S(\mf{h}(1))$.
\end{lem}
\begin{proof}
In the case $H = \mb{C}^\times$, the proof of Lemma \ref{lem:monodromic nearby cycles} shows that restriction to $\{1\} \times X$ agrees with the nearby cycles functor, since both are left inverses to \eqref{eq:monodromic inverse}, and is hence an equivalence. Iterating, we deduce the lemma for a general $H = (\mb{C}^\times)^n$.
\end{proof}

Returning to a general $H$-torsor $\tilde{X} \to X$, we have:

\begin{prop} \label{prop:monodromic mhm}
Let $\mc{M} \in \mhm_{\widetilde{\lambda}}(\tilde X)$. Then
\begin{enumerate}
\item \label{itm:monodromic mhm 1} The Hodge and weight filtrations $F_\bullet\mc{M}$ and $W_\bullet \mc{M}$ are equivariant with respect to the weak $H$-action.
\item \label{itm:monodromic mhm 2} The nilpotent part of the $\mf{h}$-action \eqref{eq:weak difference} defines a morphism of mixed Hodge modules
\[ \mf{h}(1) \otimes \mc{M} \to \mc{M}.\]
\end{enumerate}
\end{prop}
\begin{proof}
By induction on $\dim H$, it suffices to consider the case where $H = \mb{C}^\times$. Since both statements are local on $X$, we may furthermore assume that $\tilde X = \mb{C}^\times \times X$. In this case \eqref{itm:monodromic mhm 2} is immediate from Lemma \ref{lem:monodromic nearby cycles}. Moreover, \eqref{itm:monodromic mhm 1} is true by inspection for any mixed Hodge module in the image of the equivalence \eqref{eq:monodromic inverse}, and hence for any $\mc{M} \in \mhm_{\widetilde{\lambda}}(\tilde X)$.
\end{proof}

If $\mc{M} \in \mhm_\mon(\tilde{X})$ is a monodromic mixed Hodge module, then by Proposition \ref{prop:monodromic mhm}, we may endow the associated $\tilde{\mc{D}}$-module $\pi_{\bigcdot}(\mc{M})^H$ with a Hodge filtration (compatible with the natural filtration on $\tilde{\mc{D}}$) and a weight filtration (by $\tilde{\mc{D}}$-submodules) given by
\begin{equation} \label{eq:dtilde filtrations}
F_p\pi_{\bigcdot}(\mc{M})^H = \pi_{\bigcdot}(F_p\mc{M})^H \quad \mbox{and} \quad W_w\pi_{\bigcdot}(\mc{M})^H = \pi_{\bigcdot}(W_{w + \dim H}\mc{M})^H.
\end{equation}
The shift in the weight filtration above is chosen so that the natural equivalence
\[ \pi^\circ \colon \mhm(X) \overset{\sim}\to \mhm_0(\tilde{X}) \]
is compatible with forgetful functor to bi-filtered $\tilde{\mc{D}}$-modules, cf., Notation \ref{notation:functors}. When speaking of the weight of a monodromic mixed Hodge module, we will generally mean the weight with respect to the filtration \eqref{eq:dtilde filtrations} on $X$, rather than the weight of the underlying mixed Hodge module on $\tilde X$.

We often use the same notation for $\mc{M} \in \mhm_{\mon}(\tilde X)$ and for its underlying $\tilde{\mc{D}}$-module on $X$. Similarly, we will generally write $(\mc{M}, F_\bullet)$ for the underlying filtered $\tilde{\mc{D}}$-module instead of the more cumbersome $(\pi_{\bigcdot}(\mc{M})^H, \pi_{\bigcdot}(F_\bullet \mc{M})^H)$.

\subsection{The equivariant setting}

In our applications to real groups, we are interested in monodromic mixed Hodge modules that are strongly equivariant under another group $K$ acting on $\tilde{X}$. If $K$ is an algebraic group acting compatibly on both $\tilde{X}$ and $H$, then a \emph{$K$-equivariant monodromic mixed Hodge module} is a mixed monodromic mixed Hodge module $\mc{M}$ on $\tilde X$ equipped with a (strong) $K$-action such that for all $k \in K$ and all $\lambda \in \mf{h}^*$, the action map
\[ k^*\mc{M} \to \mc{M} \]
sends $k^*\mc{M}_{k\lambda}$ to $\mc{M}_\lambda$. Equivariant monodromic $\mc{D}$-modules are defined similarly. For $\lambda \in (\mf{h}^*)^K$ we can restrict to the categories of $\lambda$-twisted and $\lambda$-monodromic objects. We write
\[ \Mod^K_\lambda(\mc{D}_{\tilde{X}}) \subset \Mod^K_{\widetilde{\lambda}}(\mc{D}_{\tilde{X}}) \subset \Mod^K_{\mon}(\mc{D}_{\tilde{X}}) \]
and
\[ \mhm^K_\lambda(\tilde{X}) \subset \mhm^K_{\widetilde{\lambda}}(\tilde{X}) \subset \mhm^K_{\mon}(\tilde{X}) \]
for the categories of $K$-equivariant objects.

\begin{rmk}
Most of the time, we can restrict to the setting where $K$ acts trivially on $H$, i.e., the actions of $K$ and $H$ commute. However, we will need the more general setting to apply the trick of passing to the extended group when dealing with Hermitian representations; see \S\ref{subsec:hodge and signature}.
\end{rmk}

\section{Deformations and wall crossing for mixed Hodge modules} \label{sec:deformations}

In this section, we prove our key technical results on the behavior of mixed Hodge modules under certain natural deformations. We work in the general geometric setting of an affine locally closed immersion $j \colon Q \to X$ (i.e., for which the boundary of $Q$ is a divisor) and a mixed Hodge module $\mc{M}$ on $Q$. To these data, we associate a real vector space $\Gamma_\mb{R}(Q)$ and a deformation $f\mc{M}$ indexed by $f \in \Gamma_\mb{R}(Q)$. Our main results are generalizations of Theorem \ref{thm:intro semi-continuity} and \cite[Theorem 1.2]{DV1} where the boundary equation $f$ is replaced by an element of a natural positive cone $\Gamma_\mb{R}(Q)_+ \subset \Gamma_\mb{R}(Q)$. These techniques form the backbone of our proof of the unitarity criterion for Harish-Chandra modules. We believe this perspective is also of independent interest from a purely algebro-geometric point of view: for example, our techniques yield a quick and easy proof of the Kodaira vanishing theorem for twisted Hodge modules (see \S\ref{sec:twisted kodaira}).

The outline of the section is as follows. In \S\ref{subsec:deformation construction}, we write down the general setting for our deformations and discuss the cone of positive deformations. We briefly recall some background on Hodge filtrations and pushforward in \S\ref{subsec:hodge pushforward} before stating and proving our first main theorem on wall crossing for Hodge filtrations in \S\ref{subsec:semi-continuity}. In \S\ref{subsec:jantzen} we discuss Jantzen filtrations and prove a wall crossing theorem for polarizations. Finally, in \S\ref{subsec:monodromic deformation}, we discuss how our results apply in the setting of equivariant and monodromic mixed Hodge modules.

\subsection{Deformations and positivity} \label{subsec:deformation construction}

Let $X$ be a smooth quasi-projective variety and $j \colon Q \hookrightarrow X$ the inclusion of a smooth affinely embedded locally closed subvariety. Consider the set $\Gamma(Q, \mc{O}_Q^\times)$ of non-vanishing regular functions on $Q$, regarded as an abelian group under multiplication. We will consider deformations of mixed Hodge modules parametrized by the $\mb{R}$-vector space
\[ \Gamma_\mb{R}(Q) := \frac{\Gamma(Q, \mc{O}_Q^\times)}{\mb{C}^\times}\otimes \mb{R}.\]
To be consistent with the notation for $\Gamma(Q, \mc{O}_Q^\times)$, we will write the operations in this vector space multiplicatively. So, for example, a general element $f \in \Gamma_\mb{R}(Q)$ is of the form $f = f_1^{s_1} \cdots f_n^{s_n}$ for some $f_i \in \Gamma(Q, \mc{O}_Q^\times)$ and $s_i \in \mb{R}$.

If $\mc{M}$ is a mixed Hodge module on $Q$, we define a family of deformed mixed Hodge modules $f\mc{M}$ parametrized by $f \in \Gamma_\mb{R}(Q)$ as follows. To each $f  = f_1^{s_1}\cdots f_n^{s_n}$, we associate the local system $f\mc{O}_Q$ given by $\mc{O}_Q$ with the $\mc{D}$-module structure
\[ \partial (fv) = f \partial v + f \left(s_1\frac{\partial f_1}{f_1} + \cdots + s_n \frac{\partial f_n}{f_n}\right) v\]
for vector fields $\partial$ and $v \in \mc{O}_Q$. Note that the positive definite Hermitian form $S(fu, \overline{fv}) = |f|^2 u\bar{v}$ (well-defined up to a positive real scalar) is flat with respect to this $\mc{D}$-module structure, so the local system $f\mc{O}_Q$ is unitary and therefore underlies a pure Hodge module with trivial Hodge structure. Hence, if $\mc{M}$ is any mixed Hodge module on $Q$ then
\[ f\mc{M} := f\mc{O}_Q \otimes \mc{M} \]
is also a mixed Hodge module with the Hodge and weight filtrations induced from $\mc{M}$.

We are interested in the behavior of the extensions
\[ j_!f\mc{M} \twoheadrightarrow j_{!*}f\mc{M} \hookrightarrow j_* f\mc{M} \]
as we vary $f$. Our main results concern this behavior when $f$ varies along a ray inside the following positive cone.

Choose any normal variety $Q'$ with a proper birational map $Q' \to \bar{Q} \subset X$ that is an isomorphism over $Q$; for example, we could take $Q'$ to be the normalization of $\bar{Q}$, but it will be convenient to allow other choices also. Note that, since $Q$ is assumed affinely embedded in $X$, the complement of $Q$ in $\bar{Q}$ is a divisor, and hence so is the complement of $Q$ in $Q'$. If $D \subset Q' - Q$ is an irreducible component, we therefore have a linear map
\[ \operatorname{ord}_D \colon \Gamma_\mb{R}(Q) \to \mb{R}\]
taking the order of vanishing along $D$.

\begin{defn}
We say that $f \in \Gamma_\mb{R}(Q)$ is \emph{positive} if
\[ \operatorname{ord}_D f > 0 \]
for all irreducible components $D \subset Q' - Q$. We write
\[ \Gamma_\mb{R}(Q)_+ \subset \Gamma_\mb{R}(Q) \]
for the set of positive elements.
\end{defn}

In other words, $f$ is positive if the associated $\mb{R}$-divisor on $Q'$ is effective and has support equal to the entire boundary of $Q$. We will sometimes also refer to positive elements in the larger space $\Gamma(Q, \mc{O}^\times) \otimes \mb{R}$; by this we mean elements whose image in the quotient $\Gamma_\mb{R}(Q)$ is positive.

The following proposition ensures that the notion of positivity is independent of the choice of $Q'$. For the statement, we say that $f \in \Gamma(Q, \mc{O}_Q^\times)$ is a \emph{boundary equation} if $f$ extends to a regular function on $\bar{Q}$ such that $Q = f^{-1}(\mb{C}^\times)$.

\begin{prop} \label{prop:positivity}
Let $f \in \Gamma_\mb{R}(Q)$.
\begin{enumerate}
\item \label{itm:positivity 1} If
\[ f \in \frac{\Gamma(Q, \mc{O}_Q^\times)}{\mb{C}^\times} \otimes \mb{Q} \subset \Gamma_\mb{R}(Q) \]
then $f$ is positive if and only if $f^n$ is the image of a boundary equation for some $n \in \mb{Z}_{>0}$.
\item \label{itm:positivity 2} In general, $f \in \Gamma_\mb{R}(Q)$ is positive if and only if $f = \prod_i f_i^{a_i}$ with $f_i$ boundary equations and $a_i > 0$.
\end{enumerate}
\end{prop}
\begin{proof}
To prove \eqref{itm:positivity 1}, it is clear that if $f^n$ is a boundary equation, then $f$ is positive. For the converse, by replacing $f$ with some positive integer power, we may assume without loss of generality that $f \in \Gamma(Q, \mc{O}^\times)$. The condition that $f$ be positive implies that $f$ extends to a regular function on $Q'$ (since $Q'$ is normal and $f$ extends outside codimension $2$) such that $Q = f^{-1}(\mb{C}^\times)$. It therefore remains to show that some power of $f$ descends to a regular function on $\bar{Q}$.

First, it is clear that $f$ extends to a regular function on the normalization $Q^{\mathit{norm}}$ of $\bar{Q}$, since the map $Q' \to Q^{\mathit{norm}}$ must be an isomorphism outside a set of codimension $2$ in $Q^{\mathit{norm}}$. Then, locally on $\bar{Q}$, we may write $\bar{Q} = \spec R$ and $Q^{\mathit{norm}} = \spec S$. Since $Q$ is affinely embedded in $\bar{Q}$, we may choose $g \in R$ a boundary equation. Since $f$ and $g$ have the same vanishing locus on $Q^{\mathit{norm}}$, we have $f^p \in gS$ for some $p \in \mb{Z}_{> 0}$ by the Nullstellensatz. Moreover, the quotient $S/R$ is a finitely generated $R$-module such that the localization $(S/R)[g^{-1}] = 0$. Hence there exists $q \in \mb{Z}_{>0}$ such that $g^q(S/R) = 0$. So $g^q S \subset R$. Hence, setting $n = pq$, we have
\[ f^n \in (gS)^q \subset g^q S \subset R.\]
So $f^n$ extends to a boundary equation for $Q$ as claimed.

To prove \eqref{itm:positivity 2}, write
\[ f = g_1^{s_1} \cdots g_n^{s_n}\]
for $g_i \in \Gamma(Q, \mc{O}^\times)$ and $s_i \in \mb{R}$, and consider the map $\phi \colon \mb{R}^n \to \Gamma_\mb{R}(Q)$ sending $(t_1, \ldots, t_n)$ to $g_1^{t_1} \cdots g_n^{t_n}$. Note that since the positivity condition is defined by the positivity of finitely many linear functionals, the set of $t \in \mb{R}^n$ such that $\phi(t)$ is positive is open. Moreover, for any $\epsilon > 0$, we may choose vectors $v_0, v_1, \ldots, v_n \in \mb{Q}^n$ within a ball of radius $\epsilon$ around $s = (s_1, \ldots, s_n)$ such that $s$ lies in the interior of the convex hull of the $v_i$. So
\[ f = \prod_{i  = 0}^{n + 1} \phi(v_i)^{b_i} \]
for some $b_i > 0$ with $\sum_i b_i = 1$. But for small enough $\epsilon$, $\phi(v_i) \in \Gamma_\mb{Q}(Q)$ will be positive for all $i$, and hence $\phi(v_i) = f_i^{1/n_i}$ for some boundary equations $f_i$ and $n_i > 0$ by \eqref{itm:positivity 1}. So
\[ f = \prod_i f_i^{a_i} \quad \text{with} \quad a_i = b_i/n_i > 0 \]
as claimed.
\end{proof}

\subsection{The Hodge filtration on a pushforward} \label{subsec:hodge pushforward}

To analyze the Hodge filtration on $j_!f\mc{M}$ and $j_*f\mc{M}$, we need to recall how this filtration is constructed in the theory.

We first recall the naive pushforward operation for filtered $\mc{D}$-modules (see, e.g., \cite{laumon}). Let $g \colon X \to Y$ be a morphism of smooth varieties. If $\mc{M}$ is any $\mc{D}$-module on $X$, then the $\mc{D}$-module pushforward is
\[g_+\mc{M} := \mrm{R}g_{\bigcdot}(\mc{D}_{Y \leftarrow X} \overset{\mrm{L}}\otimes_{\mc{D}_X}\mc{M}),\]
where
\[ \mc{D}_{Y \leftarrow X} = g^{-1}\mc{D}_Y \otimes_{g^{-1}\mc{O}_Y} \omega_{X/Y},\]
with right $\mc{D}_X$-module structure defined so that the action map
\[ g^{-1}\omega_Y \otimes \mc{D}_{Y \leftarrow X} \to \omega_X \]
is a map of right $\mc{D}_X$-modules. We endow $\mc{D}_{Y \leftarrow X}$ with the filtration by order of differential operator, shifted to start in degree $\dim Y - \dim X$. If $(\mc{M}, F_\bullet)$ is a filtered $\mc{D}_X$-module, then there is an induced (naive) filtration on the complex $g_+\mc{M}$ given by
\[ F_p g_+\mc{M} = \mrm{R}g_{\bigcdot}F_p((\mc{D}_{Y \leftarrow X}, F_\bullet) \overset{\mrm{L}}\otimes_{(\mc{D}_X, F_\bullet)} (\mc{M}, F_\bullet)).\]

We note that if $\mc{M}$ is regular holonomic, then $g_+\mc{M} = g_*\mc{M}$ as objects in the derived category of $\mc{D}$-modules. If $g$ is proper, then the naive pushforward is also the correct operation at the level of Hodge filtrations: we have, for $\mc{M} \in \mhm(X)$,
\[ g_*\mc{M} = g_!\mc{M} = g_+\mc{M} \]
as objects in the filtered derived category of $\mc{D}_Y$-modules. When $g$ is not proper, the filtration needs to be adjusted. By choosing an appropriate factorization of $g$, the behavior is fixed once we give the formula for the extension across a smooth divisor.

So, suppose now that $U \subset X$ is the complement of a smooth divisor $D$ and write $j \colon U \to X$ for the inclusion. Then, for $\mc{M} \in \mhm(U)$, Saito's theory endows the extensions $j_!\mc{M}$ and $j_*\mc{M}$ with a Hodge filtration coming from the ``order of pole'' along $D$.\footnote{This is a very old idea. It plays a central role in Deligne's work on mixed Hodge structures \cite{deligne} and traces its origins via Griffiths \cite{G} as far back as Poincar\'e \cite{P}.} To make sense of ``order of pole'' for a general $\mc{D}$-module, one uses the $V$-filtration of Kashiwara and Malgrange (although this may not always be the most obvious notion of pole order in every situation, see e.g., \cite[(0.8)]{S4}). Recall that the $V$-filtration on the sheaf of differential operators is the decreasing $\mb{Z}$-indexed filtration
\[ V^n \mc{D}_X = \{ \partial \in \mc{D}_X \mid \partial \mc{I}^m \subset \mc{I}^{m + n} \text{ for all } m \geq 0\},\]
where $\mc{I}$ is the ideal sheaf of $D$. The $V$-filtration on $j_+\mc{M}$ is the unique decreasing $\mb{R}$-indexed filtration $V^\bullet j_+\mc{M}$ such that $V^\bullet j_+ \mc{M}$ is locally finitely generated over $V^\bullet \mc{D}_X$ and for all $\alpha \in \mb{R}$ the operator
\[ \alpha - t\partial_t \colon \Gr_V^\alpha j_+\mc{M} \to \Gr_V^\alpha j_+\mc{M} \]
is nilpotent. Now, one easily checks that
\begin{equation} \label{eq:V-filtration extensions}
j_!\mc{M} = \mc{D}_X \otimes_{V^0\mc{D}_X} V^{>-1}j_+\mc{M} \quad \text{and} \quad  j_*\mc{M} = \mc{D}_X \otimes_{V^0\mc{D}_X} V^{\geq -1}j_+\mc{M}
\end{equation}
as $\mc{D}$-modules; the Hodge filtrations are then given by the formulae
\begin{equation} \label{eq:shriek hodge}
(j_!\mc{M}, F_\bullet) = (\mc{D}_X, F_\bullet) \otimes_{(V^0\mc{D}_X, F_\bullet)} V^{> -1}j_+(\mc{M}, F_\bullet)
\end{equation}
and
\begin{equation} \label{eq:star hodge}
(j_*\mc{M}, F_\bullet) = (\mc{D}_X, F_\bullet) \otimes_{(V^0\mc{D}_X, F_\bullet)} V^{\geq -1}j_+(\mc{M}, F_\bullet).
\end{equation}
See, e.g., \cite[Definitions 11.3.1 and 11.4.1]{SS}.

\subsection{Semi-continuity of the Hodge filtration} \label{subsec:semi-continuity}

Let us now fix $f \in \Gamma_\mb{R}(Q)_+$ positive and $\mc{M} \in \mhm(Q)$. In this subsection, we describe how the Hodge filtration on $j_{!*}f^s\mc{M}$ varies with $s$.

We first make the following observation.

\begin{prop} \label{prop:hyperplanes}
The $\mc{D}$-module $f^s\mc{M}$ extends cleanly to $X$ (i.e., the canonical morphism $j_!f^s\mc{M} \to j_*f^s\mc{M}$ is an isomorphism) for $s \in \mb{R}$ outside a discrete subset.
\end{prop}
\begin{proof}
We may assume without loss of generality that $\mc{M} = j'_*\mc{V}$ for some local system $\mc{V}$ on a smooth locally closed subset $Q' \subset Q$, where $j' \colon Q' \to Q$ is the inclusion, since any $\mc{M}$ may be resolved by a complex of $\mc{D}$-modules of this form. Passing to a resolution of singularities $\tilde{Q}' \to \bar{Q}'$ of the closure of $Q'$ in $X$ if necessary, we may reduce further to the case where $Q' \subset Q \subset X$ are open subsets with normal crossings boundary.

Write $D_1, \ldots, D_m$ for the irreducible components of $X - Q'$, $V^\bullet_{i}$ for the $V$-filtration along $D_i$ and $I_0 = \{i \in I \mid D_i \subset X - Q\}$. For $I \subset \{1, \ldots, m\}$, write $D_I^\circ = \bigcap_{i \in I} D_i - \bigcup_{j \not\in I} D_j$. A regular holonomic $\mc{D}$-module $\mc{N}$ is said to be of \emph{normal crossings type} if it is constructible with respect to the Whitney stratification defined by the $D_I^\circ$. For $\mc{D}$-modules of normal crossings type, the collection of $V$-filtrations $V^\bullet_{i}$ are distributive, i.e., the functors $\Gr_{V_{i}}^\alpha$ commute (see, e.g., \cite[Proposition 15.7.15]{SS}). For $I = \{i_1, \ldots, i_k\} \subset \{1, \ldots, n\}$ and $\boldsymbol{\alpha} \in \mb{R}^I$, write
\[ \Gr_{V_I}^{\boldsymbol{\alpha}} \mc{N} = \Gr_{V_{i_1}}^{\alpha_{i_1}} \cdots \Gr_{V_{i_k}}^{\alpha_{i_k}} \mc{N}|_{D_I^\circ}\]
and
\[ \phi_I \mc{N} = \Gr_{V_I}^{-1, \ldots, -1}\mc{N}.\]
It is an elementary and well-known fact that, for $\mc{N}$ of normal crossings type, $\mc{N} = 0$ if and only if $\phi_I \mc{N} = 0$ for all $I$. Now, choosing local coordinates $t_i$ such that $D_i = \{t_i = 0\}$, the morphism
\[ \phi_I(j_!f^s\mc{M}) \to \phi_I(j_*f^s\mc{M}) \]
may be identified with the morphism
\begin{equation} \label{eq:hyperplanes 1}
\prod_{i \in I \cap I_0} t_i \partial_{t_i} \colon \Gr_{V_I}^{0, \ldots, 0} j_*f^s\mc{M} \to \Gr_{V_I}^{0, \ldots, 0} j_*f^s\mc{M}.
\end{equation}
Moreover, setting $\alpha_i = -\mrm{ord}_{D_i}(f)s$, we have 
\[ \Gr_{V_I}^{0, \ldots, 0}j_*f^s\mc{M} \cong \Gr_{V_I}^{\boldsymbol{\alpha}}j_*\mc{M},\]
so \eqref{eq:hyperplanes 1} is an isomorphism unless $-\mrm{ord}_{D_i}(f) s$ lies in the discrete set of jump points of $V_i^\bullet j_*\mc{M}$ for some $i \in I_0$. Since $f$ is positive, $\mrm{ord}_{D_i}(f) > 0$ for all $i \in I_0$, so this happens only for $s$ in a discrete subset of $\mb{R}$.
\end{proof}

The main result of this subsection is:

\begin{thm} \label{thm:semi-continuity}
The Hodge filtrations on $j_!f^s\mc{M}$ and $j_*f^s\mc{M}$ are semi-continuous in the sense that
\[ \Gr^F j_!f^{s - \epsilon}\mc{M} \cong \Gr^F j_!f^s\mc{M} \quad \text{and} \quad \Gr^F j_*f^{s + \epsilon} \mc{M} \cong \Gr^Fj_*f^s\mc{M}\]
as coherent sheaves on $T^* X$ for all $0 < \epsilon \ll 1$.
\end{thm}

In particular, Theorem \ref{thm:semi-continuity} implies that the isomorphism class of $\Gr^F j_{!*}f^s\mc{M}$ is constant on each interval on which $j_!f^s\mc{M} = j_*f^s\mc{M}$.

The basic idea is behind the proof of Theorem \ref{thm:semi-continuity} is that the inequalities in \eqref{eq:shriek hodge} and \eqref{eq:star hodge} lead to semi-continuity properties of the associated gradeds. In order to exploit this idea to produce canonical isomorphisms, we will use the following constructions.

Fix $\mc{M} \in \mhm(Q)$ and $f \in \Gamma_\mb{R}(Q)$ positive. Consider the (infinitely generated) $\mc{D}_Q$-module
\[ f^s\mc{M}[s] \]
given by $f^s\mc{M}[s] = \mc{M}\otimes \mb{C}[s]$ as $\mc{O}_Q$-modules, with differential operators acting by
\[ \partial f^s m(s) = f^s \partial m(s) + s f^s\frac{\partial f}{f} m(s) \]
for vector fields $\partial$ on $Q$, where, for $f = f_1^{s_1} \cdots f_k^{s_k}$, $f_i \in \Gamma(Q, \mc{O}^\times)$ we write
\[ \frac{\partial f}{f} := s_1 \frac{\partial f_1}{f_1} + \cdots + s_k \frac{\partial f_k}{f_k}.\]
We endow $f^s\mc{M}[s]$ with a filtration by setting
\[ F_pf^s\mc{M}[s] = \sum_{k + q \leq p} s^kf^sF_q\mc{M},\]
i.e., the filtration comes from the Hodge filtration on $\mc{M}$ and the polynomial order in $s$. Via the naive filtered pushforward, we obtain a filtered $\mc{D}$-module $j_+f^s\mc{M}[s]$ on $X$.

Now fix $s_0 \in \mb{R}$. Our aim is to realize the filtered $\mc{D}$-modules $j_!f^{s_0}\mc{M}$ and $j_*f^{s_0}\mc{M}$ as subquotients of $j_+f^s\mc{M}[s]$. To do so, observe that, for any $n > 0$, the quotient module
\[ \frac{f^s\mc{M}[s]}{(s - s_0)^n}, \]
equipped with its inherited filtration, underlies a mixed Hodge module on $Q$, given by tensoring $\mc{M}$ with a rank $n$ admissible variation of mixed Hodge structure on $Q$ (cf., the proof of Lemma \ref{lem:monodromic nearby cycles}). Moreover, for any $p$, there exists $n$ such that
\[ F_{q} \frac{f^s\mc{M}[s]}{(s - s_0)^{m + 1}} \to F_{q} \frac{f^s\mc{M}[s]}{(s - s_0)^{m}} \]
is an isomorphism for all $q \leq p$ and all $m \geq n$. It follows that for each $p$, the sequences
\[ F_p j_!\left(\frac{f^s\mc{M}[s]}{(s - s_0)^n}\right) \quad \text{and} \quad F_p j_*\left(\frac{f^s\mc{M}[s]}{(s - s_0)^n}\right) \]
stabilize for $n \gg 0$. We define (infinitely generated) filtered $\mc{D}$-modules $j_!^{(s_0)}f^s\mc{M}[s]$ and $j_*^{(s_0)}f^s\mc{M}[s]$ on $X$ by setting
\[ F_p j_!^{(s_0)}f^s\mc{M}[s] := \varprojlim_n F_p j_! \left(\frac{f^s\mc{M}[s]}{(s - s_0)^n}\right), \quad j_!^{(s_0)}f^s\mc{M}[s] = \bigcup_p F_pj_!^{(s_0)}f^s\mc{M}[s]\]
and
\[ F_p j_*^{(s_0)}f^s\mc{M}[s] := \varprojlim_n F_p j_*\left(\frac{f^s\mc{M}[s]}{(s - s_0)^n}\right), \quad j_*^{(s_0)}f^s\mc{M}[s] = \bigcup_p F_pj_*^{(s_0)}f^s\mc{M}[s]. \]
By construction, we have strict short exact sequences
\[ 0 \to j_!^{(s_0)}f^s\mc{M}[s]\{1\} \xrightarrow{s - s_0} j_!^{(s_0)}f^s\mc{M}[s] \to j_!f^{s_0}\mc{M} \to 0\]
and
\[ 0 \to j_*^{(s_0)}f^s\mc{M}[s]\{1\} \xrightarrow{s - s_0} j_*^{(s_0)}f^s\mc{M}[s] \to j_*f^{s_0}\mc{M} \to 0\]
and canonical morphisms of filtered $\mc{D}$-modules
\begin{equation} \label{eq:semi-continuity inclusion 1}
 j_!^{(s_0)} f^s\mc{M}[s] \to j_*^{(s_0)} f^s\mc{M}[s] \to j_+f^s\mc{M}[s]
\end{equation}
restricting to the identity on $Q$. Here $\{n\}$ denotes the filtration shift $F_p(\mc{N}\{n\}) = F_{p - n}\mc{N}$.

The main step in the proof of Theorem \ref{thm:semi-continuity} is:

\begin{lem} \label{lem:semi-continuity inclusion}
The morphisms \eqref{eq:semi-continuity inclusion 1} are injective and strict with respect to the filtrations. Moreover, for $0 < \epsilon \ll 1$, we have
\begin{equation} \label{eq:semi-continuity inclusion shriek}
 j_!^{(s_0)} f^s\mc{M}[s] = j_!^{(s_0 - \epsilon)}f^s\mc{M}[s] = j_*^{(s_0 - \epsilon)}f^s\mc{M}[s]
\end{equation}
and
\begin{equation} \label{eq:semi-continuity inclusion star}
j_*^{(s_0)} f^s\mc{M}[s] = j_*^{(s_0 + \epsilon)}f^s\mc{M}[s] = j_!^{(s_0 + \epsilon)}f^s\mc{M}[s]
\end{equation}
as subsheaves of $j_+f^s\mc{M}[s]$, and
 \[ j_+f^s\mc{M}[s] = \bigcup_{s_0 \in \mb{R}} j_!^{(s_0)}f^s\mc{M}[s] = \bigcup_{s_0 \in \mb{R}} j_*^{(s_0)} f^s\mc{M}[s]. \]
\end{lem}

Let us first prove the strict injectivity for the morphism
\[ j_!^{(s_0)} f^s\mc{M}[s] \to j_*f^{(s_0)} f^s\mc{M}[s].\]
By construction, this is obtained from the morphism of pro-mixed Hodge modules
\begin{equation} \label{eq:completed inclusion 1}
j_!f^s\mc{M}[[s - s_0]] := \varprojlim_n j_! \frac{f^s\mc{M}[s]}{(s - s_0)^n} \to \varprojlim_n j_* \frac{f^s\mc{M}[s]}{(s - s_0)^n} =: j_*f^s\mc{M}[[s - s_0]]
\end{equation}
by applying the functor $\bigcup_p F_p$.

\begin{lem} \label{lem:completed inclusion}
The morphism \eqref{eq:completed inclusion 1} is injective and $s - s_0$ acts nilpotently on the cokernel. Hence, the morphism
\begin{equation} \label{eq:completed inclusion 5}
j_!^{(s_0)} f^s\mc{M}[s] \to j_*^{(s_0)} f^s\mc{M}[s]
\end{equation}
is strict injective.
\end{lem}
\begin{proof}
Let us assume for notational simplicity that $s_0 = 0$. We claim that there exists $N$ such that the kernel and cokernel of
\[ j_!\frac{f^s\mc{M}[s]}{(s^n)} \to j_*\frac{f^s\mc{M}[s]}{(s^n)} \]
are annihilated by $s^N$ for all $n$. The asserted property of the cokernel of \eqref{eq:completed inclusion 5} is then clear, and injectivity follows since the morphism
\[ s^N \colon j_!\frac{f^s\mc{M}[s]}{(s^n)} \to j_!\frac{f^s\mc{M}[s]}{(s^{n + N})}(-N)\]
is injective. Moreover, \eqref{eq:completed inclusion 5} is strict since it is the limit of a tower of morphisms of mixed Hodge modules.

To prove the claim, we may reduce as in the proof of Proposition \ref{prop:hyperplanes} to the case where $\mc{M}$ is the intermediate extension of a local system on $Q' $, where $Q' \subset Q \subset X$ are open subsets with normal crossings boundaries. Using the notation from the proof of Proposition \ref{prop:hyperplanes}, it is enough to prove that for all $I \subset \{1, \ldots, m\}$, the kernel and cokernel of
\begin{equation} \label{eq:completed inclusion 2}
\phi_I j_!\frac{f^s\mc{M}[s]}{(s^n)} \to \phi_I j_*\frac{f^s\mc{M}[s]}{(s^n)}
\end{equation}
are annihilated by $s^N$ for some $N$. But \eqref{eq:completed inclusion 2} may be identified with the morphism
\begin{equation} \label{eq:completed inclusion 3}
\prod_{i \in I \cap I_0} (t_i\partial_{t_i} + \mrm{ord}_{D_i}(f) s) \colon \Gr_{V_I}^{0, \ldots, 0}j_*\mc{M} \otimes \frac{\mb{C}[s]}{(s^n)} \to \Gr_{V_I}^{0, \ldots, 0}j_*\mc{M} \otimes \frac{\mb{C}[s]}{(s^n)}.
\end{equation}
Since the operators $t_i\partial_{t_i}$ are nilpotent and $\mrm{ord}_{D_i}(f) > 0$, the kernel and cokernel of \eqref{eq:completed inclusion 3} are indeed annihilated by $s^N$ for some $N$ independent of $n$.
\end{proof}

\begin{proof}[Proof of Lemma \ref{lem:semi-continuity inclusion}]
Let us first consider the case where $f$ is a boundary equation for $Q$. Since the claim is local on $X$, we may assume without loss of generality that $f$ lifts to a regular function $f \colon X \to \mb{C}$. Using the standard trick of embedding via the graph of $f$, we may reduce to the case where $X = X' \times \mb{C}$, $Q = X' \times \mb{C}^\times$ and $f = t$ is the coordinate on $\mb{C}$ for some smooth variety $X'$.

In this setting, the formulas \eqref{eq:shriek hodge} and \eqref{eq:star hodge} give
\[ j_!^{(s_0)} f^s\mc{M}[s] = \mc{D}_X \otimes_{V^0\mc{D}_X} f^{s - s_0}(V^{>-1}j_+f^{s_0}\mc{M})[s] = \mc{D}_X \otimes_{V^0\mc{D}_X}f^s(V^{>-1 - s_0}j_+\mc{M})[s] \]
and
\[ j_*^{(s_0)} f^s\mc{M}[s] = \mc{D}_X \otimes_{V^0\mc{D}_X}f^s(V^{\geq -1 - s_0}j_+\mc{M})[s] \]
as filtered $\mc{D}$-modules. The semi-continuity properties \eqref{eq:semi-continuity inclusion shriek} and \eqref{eq:semi-continuity inclusion star} are now clear, and also
\[ j_+ f^s\mc{M}[s] = \varinjlim_{s_0 \in \mb{R}} j_!^{(s_0)}f^s\mc{M}[s] = \varinjlim_{s_0 \in \mb{R}} j_*^{(s_0)}f^s\mc{M}[s].\]
It therefore remains to prove that the maps
\begin{equation} \label{eq:semi-continuity inclusion 4}
j_!^{(s_0)} f^s\mc{M}[s] \to j_!^{(s_1)} f^s\mc{M}[s]
\end{equation}
induced by the inclusions of the $V$-filtrations for $s_0 < s_1$ are strict and injective. But by the semi-continuity properties noted above, these are all finite compositions of filtered isomorphisms and maps of the form
\begin{equation} \label{eq:semi-continuity inclusion 2}
j_!^{(s_0)}f^s\mc{M}[s] \to j_*^{(s_0)}f^s\mc{M}[s].
\end{equation}
Since \eqref{eq:semi-continuity inclusion 2} is strict injective by Lemma \ref{lem:completed inclusion}, so is \eqref{eq:semi-continuity inclusion 4}, and the lemma is proved in the case where $f$ is a genuine boundary equation.

Let us now consider the general case. Using Proposition \ref{prop:positivity}, we write $f$ as a product of positive real powers of boundary equations. By induction on the number of boundary equations required, it suffices to prove the lemma for $f = f_1 f_2$ under the assumption that it holds for $f_1$ and $f_2$ separately. For $a_1, a_2 \in \mb{R}$, we may form the two-variable version of our construction, to define,
\[ j_!^{(a_1, a_2)} f_1^{s_1}f_2^{s_2} \mc{M}[s_1, s_2] \quad \text{and} \quad j_*^{(a_1, a_2)} f_1^{s_1}f_2^{s_2}\mc{M}[s_1, s_2]\]
with
\[ F_p j_!^{(a_1, a_2)} f_1^{s_1}f_2^{s_2} \mc{M}[s_1, s_2] := \varprojlim_{m, n} F_p j_!\left(\frac{f^{s_1}f^{s_2}\mc{M}[s_1, s_2]}{((s_1 - a_1)^m, (s_2 - a_2)^n)}\right)\]
and similarly for $j_*$. Applying the induction hypotheses to $f_1$ and $f_2$, we deduce that the maps
\[ j_!^{(a_1, a_2)} f_1^{s_1}f_2^{s_2} \mc{M}[s_1, s_2] \to j_*^{(a_1, a_2)} f_1^{s_1}f_2^{s_2} \mc{M}[s_1, s_2] \to j_+ f_1^{s_1}f_2^{s_2}\mc{M}[s_1, s_2]\]
are injective and strict, and that, for any fixed $a \in \mb{R}$, we have
\[ j_+ f_1^{s_1}f_2^{s_2}\mc{M}[s_1, s_2] = \bigcup_{a_1 \in \mb{R}} j_!^{(a_1, a)} f_1^{s_1}f_2^{s_2} \mc{M}[s_1, s_2] = \bigcup_{a_2 \in \mb{R}} j_!^{(a, a_2)} f_1^{s_1}f_2^{s_2}\mc{M}[s_1, s_2].\]
Taking the quotient setting $s_1 = s_2 = s$, we deduce that
\begin{equation} \label{eq:semi-continuity inclusion 3}
 j_+f^s\mc{M}[s] = \varinjlim_{a \in \mb{R}} j_!^{(a)}f^s\mc{M}[s] = \varinjlim_{a \in \mb{R}} j_*^{(a)}f^s\mc{M}[s];
\end{equation}
here, for example, the morphisms
\[
 j_!^{(a)}f^s\mc{M}[s] \to j_!^{(b)}f^s\mc{M}[s]
\]
for $a < b$ in the first colimit come from taking the quotient $s_1 = s_2 = s$ of the inclusion
\[ j_!^{(a, a)}f_1^{s_1}f_2^{s_2}\mc{M}[s_1, s_2] \subset j_!^{(b, b)}f_1^{s_1}f_2^{s_2} \mc{M}[s_1, s_2].\]
Now, semi-continuity for $f_1$ and $f_2$ imply that \eqref{eq:semi-continuity inclusion shriek} and \eqref{eq:semi-continuity inclusion star} hold also for $j_!^{(s_0)}f^s\mc{M}[s]$ and $j_*^{(s_0)}f^s\mc{M}[s]$ with respect to the maps constructed in this way; together with \eqref{eq:semi-continuity inclusion 3}, this implies that the inclusions \eqref{eq:semi-continuity inclusion 1} are injective and strict as argued above, so we are done.
\end{proof}

\begin{proof}[Proof of Theorem \ref{thm:semi-continuity}]
Let us prove the semi-continuity statement for $j_!$; the proof for $j_*$ is identical. We have
\[ j_!f^{s_0}\mc{M} = \coker((s - s_0) \colon j_!^{(s_0)}f^s\mc{M}[s] \{1\} \to j_!^{(s_0)}f^s\mc{M}[s])\]
as filtered $\mc{D}$-modules, and hence
\[ \Gr^F j_!f^{s_0}\mc{M} = \coker(s \colon \Gr^F j_!^{(s_0)}f^s\mc{M}[s]\{1\} \to \Gr^Fj_!^{(s_0)}f^s\mc{M}[s])\]
as coherent sheaves on $T^*X$. But now, $j_!^{(s_0 - \epsilon)} f^s\mc{M}[s] = j_!^{(s_0)} f^s\mc{M}[s]$ by Lemma \ref{lem:semi-continuity inclusion}, so we deduce
\[  \Gr^Fj_!f^{s_0 - \epsilon}\mc{M} \cong \Gr^F j_!f^{s_0}\mc{M} \quad \text{for $0 < \epsilon \ll 1$}\]
as claimed.
\end{proof}

\subsection{Wall crossing for polarizations} \label{subsec:jantzen}

Theorem \ref{thm:semi-continuity} tells us that, for $f \in \Gamma_\mb{R}(Q)_+$ positive, the associated graded $\Gr^F j_{!*}f^s\mc{M}$ does not vary with $s$ as long as $j_!f^s\mc{M} = j_*f^s\mc{M}$, and that the wall-crossing behavior is controlled by the morphism $j_!f^s\mc{M} \to j_*f^s\mc{M}$. To prove our unitarity criterion, we will also need to understand the analogous wall-crossing behavior for polarizations. This is accomplished using the standard yoga of Jantzen filtrations as in, for example, \cite[\S 3.3]{DV1}. We recall how this works below.

Let $f \in \Gamma(Q, \mc{O}^\times) \otimes \mb{R}$ be positive and fix $\mc{M} \in \mhm(Q)$. As in \cite[\S 3.3]{DV1}, we assume $\mc{M}$ pure of weight $w$ and let
\[ S \colon \mc{M} \otimes \overline{\mc{M}} \to \Db_Q\]
be a polarization. We have an induced $\mb{C}[s]$-bilinear pairing
\[ S \colon f^s\mc{M}[s] \otimes \overline{f^s\mc{M}[s]} \to \Db_Q[s]^{\mathit{hol}}\]
given by
\[ S(f^sm, \overline{f^sm'}) = |f|^{2s} S(m, \overline{m'}).\]
Here $\Db_Q[s]^{\mathit{hol}}$ denotes the sheaf of distributions on $Q$ depending holomorphically on $s$. This extends uniquely to a pairing
\begin{equation} \label{eq:jantzen 2}
 S \colon j_+f^s\mc{M}[s] \otimes \overline{j_+f^s\mc{M}[s]} \to \Db_X(s)^{\mathit{mer}},
\end{equation}
where $\Db_X(s)^{\mathit{mer}}$ denotes the sheaf of distributions on $X$ depending meromorphically on $s$ (cf., e.g., \cite[Proposition 12.5.4]{SS}). If $s_0 \in \mb{R}$ is such that $j_!f^{s_0}\mc{M} \to j_*f^{s_0}\mc{M}$ is an isomorphism, then \eqref{eq:jantzen 2} has no poles at $s = s_0$, and evaluating at $s=s_0$ gives a polarization on $j_{!*}f^{s_0}\mc{M} = j_*f^{s_0}\mc{M} = j_+f^{s_0}\mc{M}$.

To analyze the behavior of \eqref{eq:jantzen 2} near a wall, which we may suppose for convenience to be $s = 0$, we proceed as follows. Taking the formal completion, we obtain a pairing
\[ j_*f^s\mc{M}[[s]] \otimes \overline{j_*f^s\mc{M}[[s]]} \to \Db_X((s)),\]
which, upon restriction to $j_!f^s\mc{M}[[s]] \subset j_*f^s\mc{M}[[s]]$, gives a pairing
\begin{equation} \label{eq:jantzen 3}
j_!f^s\mc{M}[[s]] \otimes \overline{j_*f^s\mc{M}[[s]]} \to \Db_X[[s]]
\end{equation}
with no poles. Setting $s = 0$, \eqref{eq:jantzen 3} becomes the unique (perfect) pairing between $j_!\mc{M}$ and $j_*\mc{M}$ extending the polarization $S$. Now, by Lemma \ref{lem:completed inclusion}, the morphism
\[ j_! f^s\mc{M}[[s]] \to j_*f^s\mc{M}[[s]] \]
is an isomorphism after inverting $s$. So by standard linear algebra, we obtain (finite) Jantzen filtrations on $j_!\mc{M}$ and $j_*\mc{M}$ given by
\[ J_{-n} j_!\mc{M} = (j_!f^s\mc{M}[[s]] \cap s^{n}j_*f^s\mc{M}[[s]])/(s) \]
and
\[ J_n j_*\mc{M} = (s^{-n}j_!f^s\mc{M}[[s]] \cap j_*f^s\mc{M}[[s]])/(s)\]
and isomorphisms
\[ s^n \colon \Gr^J_n j_*\mc{M} \overset{\sim}\to \Gr^J_{-n} j_!\mc{M}(-n).\]
The Jantzen filtrations on $j_!\mc{M}$ and $j_*\mc{M}$ are dual under $S$, so we obtain non-degenerate pairings (``Jantzen forms'')
\[ s^{-n} \Gr^J_{-n}(S) \colon \Gr^J_{-n} j_!\mc{M} \otimes \overline{\Gr^J_{-n} j_!\mc{M}} \to \Db_X \]
for all $n$.

\begin{thm} \label{thm:jantzen}
For all $n$, $\Gr^J_{-n}j_!\mc{M}$ is a pure Hodge module of weight $w - n$, and the Jantzen form $s^{-n} \Gr^J_{-n}(S)$ is a polarization. In particular, $J_n j_!\mc{M} = W_{w + n} j_!\mc{M}$.
\end{thm}

Theorem \ref{thm:jantzen} is a version of Beilinson and Bernstein's theorem on Jantzen filtrations and weight filtrations \cite{BB2} in the context of polarized Hodge modules. In the case where $f$ is a boundary equation for $Q$, Theorem \ref{thm:jantzen} is \cite[Theorem 3.2]{DV1}. We briefly outline here the proof given in \cite[\S 7]{DV1} and explain which modifications are needed for general $f \in \Gamma_\mb{R}(Q)_+$.

First, in \cite[\S 7.1]{DV1}, we wrote down the Beilinson functors
\[ \pi_f^a(\mc{M}) := \coker(j_!f^s\mc{M}[[s]](a) \xrightarrow{s^a} j_*f^s\mc{M}[[s]]) \in \mhm(X)\]
for $\mc{M} \in \mhm(Q)$ and $a \geq 0$, and their nilpotent endomorphisms
\[ \mrm{N} : = s \colon \pi_f^a(\mc{M}) \to \pi_f^a(\mc{M})(-1).\]
For $(\mc{M}, S)$ polarized of weight $w$, $\pi_f^a(\mc{M})$ is also equipped with a non-degenerate Hermitian form
\[ \pi_f^a(S) \colon \pi_f^a(\mc{M}) \otimes \overline{\pi_f^a(\mc{M})} \to \Db_X\]
given by
\[ \pi_f^a(S)(f^sm, \overline{f^sm'}) = \Res_{s = 0}s^{-a}S(f^sm, \overline{f^sm'}).\]
These constructions all make sense for general $f \in \Gamma(Q, \mc{O}^\times) \otimes \mb{R}$ positive.

Next, in \cite[\S 7.3]{DV1}, we adapted the arguments of \cite{BB2} to reduce Theorem \ref{thm:jantzen} to the following statement about $\pi_f^0(\mc{M})$ (cf., \cite[Corollary 7.3]{DV1}):

\begin{lem} \label{lem:beilinson polarization}
For all $n$, the morphism
\[ \mrm{N}^n \colon \Gr^W_{w + 1 + n} \pi_f^0(\mc{M}) \to \Gr^W_{w + 1 - n} \pi_f^0(\mc{M})(-n) \]
is an isomorphism, and the form
\[ (-1)^{n - 1} \pi_f^a(S)(-, \overline{\mrm{N}^n(-)}) \]
is a polarization on the Lefschetz primitive part
\[ \mrm{P}_n\pi_f^0(\mc{M}) := \ker (\mrm{N}^{n + 1} \colon \Gr^W_{w + 1 + n} \pi_f^0(\mc{M}) \to \Gr^W_{w - 1 - n} \pi_f^0(\mc{M})(-n - 1)).\]
\end{lem}

This reduction works word for word the same in our setting, so it remains to prove Lemma \ref{lem:beilinson polarization}.

In \cite[\S 7.2]{DV1}, we proved Lemma \ref{lem:beilinson polarization} by relating the Beilinson functor $\pi_f^0(\mc{M})$ to the unipotent nearby cycles functor $\psi_f^{\mathit{un}} \mc{M}$, for which the analogue of Lemma \ref{lem:beilinson polarization} is one of the axioms for polarized Hodge modules. For arbitrary positive $f \in \Gamma(Q, \mc{O}^\times) \otimes \mb{R}$, we instead argue as follows.

First, by \cite[Theorem 3.21]{S2}, we can write $(\mc{M}, S) = (j'_{!*}\mc{M}', S')$ for some locally closed embedding $j' \colon Q' \hookrightarrow Q$, with $Q'$ smooth, and some polarized variation of Hodge structure $(\mc{M}', S')$ on $Q'$. Next, take a resolution $g \colon \tilde{Q}' \to \bar{Q}'$ for the closure of $Q'$ in $X$ such that $Q' \cong g^{-1}(Q') \subset \tilde{Q}'$ is the complement of a divisor with simple normal crossings. By \cite[Proposition 1]{S1}, it is enough to prove the lemma with $(Q', \mc{M}', \tilde{Q}')$ in place of $(Q, \mc{M}, X)$. So we may as well assume that $Q \subset X$ is the complement of a divisor with simple normal crossings and that $\mc{M}$ is a polarized variation of Hodge structure on $Q$.

In this setting, the analysis of \cite[\S 3]{S2} applies, in particular the proof of \cite[Theorem 3.20]{S2}, which uses the nilpotent orbit theorem and Kashiwara's result \cite[Proposition 3.19]{S2} to prove Lemma \ref{lem:beilinson polarization} in the case where $f$ is a product of positive integer powers of the monomials defining the components of the normal crossings divisor $X - Q$. In our generality, $f$ may be a product of positive real powers of these monomials; all we need to generalize Saito's argument to this setting is a version of \cite[Proposition 3.19]{S2} in which the coefficients are positive reals instead of natural numbers. But the proof of this result works word for word in this generality, so we deduce that Lemma \ref{lem:beilinson polarization} does indeed hold. This completes the proof of Theorem \ref{thm:jantzen}.

\begin{rmk}
We will use Theorem \ref{thm:jantzen} in the proof of Theorem \ref{thm:intro unitarity criterion} on unitary representations of real groups. For this purpose it is possible, if desired, to make do with \cite[Theorem 3.2]{DV1}: using standard continuity arguments, one can reduce Theorem \ref{thm:intro unitarity criterion} to the case of \emph{rational} infinitesimal character $\lambda \in \mf{h}^*_\mb{Q}$. In this setting, the deformations of \S\ref{sec:pf of unitarity} are all obtained from genuine boundary equations by Proposition \ref{prop:positivity}, so \cite[Theorem 3.2]{DV1} applies.
\end{rmk}

\subsection{The monodromic and equivariant settings} \label{subsec:monodromic deformation}

In this subsection, we explain how the results above extend to the setting of twisted and monodromic mixed Hodge modules. The main observations are that our deformations change the twisting parameter in a simple way and that the isomorphisms of Theorem \ref{thm:semi-continuity} descend to isomorphisms between the associated gradeds of the corresponding $\tilde{\mc{D}}$-modules.

Let us first define the appropriate deformation spaces. Suppose we are given a Cartesian diagram
\[
\begin{tikzcd}
\tilde{Q} \ar[r] \ar[d, "\pi_Q"] & \tilde{X} \ar[d, "\pi_X"] \\
Q \ar[r] & X,
\end{tikzcd}
\]
where the vertical arrows are principal $H$-bundles for some torus $H$ and the top arrow is $H$-equivariant. Define
\[ \Gamma(\tilde Q, \mc{O}^\times)^\mon = \coprod_{\mu \in \mb{X}^*(H)} \Gamma(\tilde{Q}, \mc{O}^\times)_\mu \]
to be the group of non-vanishing semi-invariant functions on $\tilde Q$, where $\Gamma(\tilde{Q}, \mc{O}^\times)_\mu$ denotes the set of non-vanishing functions lying in the $\mu$-weight space $\Gamma(\tilde{Q}, \mc{O})_\mu$. Note that
\[ \Gamma(\tilde Q, \mc{O})_\mu = \Gamma(Q, \mc{O}_Q(\mu)) \]
is the space of sections of a line bundle $\mc{O}_Q(\mu) = \pi_{Q\bigcdot}(\mc{O}_{\tilde Q} \otimes \mb{C}_{-\mu})^H$ on $Q$. We set
\[ \Gamma_\mb{R}(\tilde Q)^\mon = \frac{\Gamma(\tilde Q, \mc{O}_{\tilde Q}^\times)^\mon}{\mb{C}^\times} \otimes \mb{R} \subset \Gamma_\mb{R}(\tilde{Q}).\]
We have a canonical linear map
\[\varphi \colon \Gamma_\mb{R}(\tilde Q)^\mon \to \mf{h}^*_\mb{R}\]
sending $f = f_1^{s_1} \cdots f_n^{s_n}$ to
\[ \varphi(f) = \sum_i s_i \mu_i \quad \text{for $f_i \in \Gamma(Q, \mc{O}_Q(\mu_i)^\times) = \Gamma(\tilde Q, \mc{O}^\times)_{\mu_i}$}.\] 

One easily checks that if $\mc{M} \in \mhm(\tilde{Q})$ is $\lambda$-monodromic then $f\mc{M}$ is naturally $(\lambda + \varphi(f))$-monodromic. Moreover, the following is clear from the proof of Theorem \ref{thm:semi-continuity}.

\begin{prop} \label{prop:monodromic semi-continuity}
If $\mc{M} \in \mhm_{\mon}(\tilde{Q})$ and $f \in \Gamma_\mb{R}(\tilde Q)^{\mon}$ is positive, then the isomorphisms of Theorem \ref{thm:semi-continuity} are equivariant under the weak $H$-actions. In particular, they descend to isomorphisms
\[ \Gr^F j_!f^{s - \epsilon} \mc{M} \cong \Gr^F j_!f^s\mc{M} \quad \text{and} \quad \Gr^F j_*f^{s + \epsilon}\mc{M} \cong \Gr^F j_*f^s\mc{M} \]
of $\Gr^F\tilde{\mc{D}}$-modules on $X$.
\end{prop}

We conclude by noting that there is also an equivariant version of this theory. Suppose that an algebraic group $K$ acts compatibly on $\tilde{Q}$, $\tilde{X}$, $Q$, $X$ and $H$. We set
\[ \Gamma_\mb{R}^K(\tilde Q) = \frac{\Gamma(\tilde Q, \mc{O}_{\tilde Q}^\times)^K}{\mb{C}^\times} \otimes \mb{R} \subset \Gamma_\mb{R}(\tilde Q) \]
and
\[ \Gamma_\mb{R}^K(\tilde Q)^\mon = \Gamma_\mb{R}^K(\tilde{Q}) \cap \Gamma_\mb{R}(\tilde Q)^\mon.\]
Elements $f \in \Gamma_\mb{R}^K(\tilde Q)^\mon$ have the property that $\varphi(f) \in (\mf{h}^*_\mb{R})^K$, and if $\mc{M}$ is a $K$-equivariant monodromic mixed Hodge module, then so is $f\mc{M}$. In this setting, the isomorphisms of Proposition \ref{prop:monodromic semi-continuity} respect the action of $K$ by construction.

\section{The twisted and monodromic Kodaira vanishing theorems}
\label{sec:twisted kodaira}

In this section, we state and prove versions of Saito's Kodaira vanishing theorem for twisted and monodromic mixed Hodge modules.

The basic idea is that for a proper morphism $g \colon X \to Y$, we can push forward twisted or monodromic mixed Hodge modules from $X$ to $Y$ as long as the twisting parameter $\lambda$ is trivial (or, more generally, pulled back from $Y$), and this pushforward is given by a simple formula. When the twisting parameter is instead relatively ample, the simple formula still makes sense (at the graded level in the twisted case and the filtered level in the monodromic case) and satisfies a vanishing theorem. We treat the twisted setting in \S\ref{subsec:twisted kodaira} and the monodromic setting in \S\ref{subsec:monodromic kodaira}.

\subsection{Graded pushforwards and twisted Kodaira vanishing} \label{subsec:twisted kodaira}

We first recall (cf., e.g., \cite{laumon}) the interaction between proper pushforwards and associated gradeds for mixed Hodge modules. Let $g \colon X \to Y$ be a proper morphism of smooth varieties and let $\mc{M} \in \mhm(X)$. Recall from \S\ref{subsec:hodge pushforward} that the filtered complex underlying $g_*\mc{M}$ is given by
\begin{equation} \label{eq:hodge proper pushforward}
(g_*\mc{M}, F_\bullet) = \mrm{R}g_{\bigcdot}((\mc{D}_{Y\leftarrow X}, F_\bullet) \overset{\mrm{L}}\otimes_{(\mc{D}_X, F_\bullet)} (\mc{M}, F_\bullet)),
\end{equation}
where $\mc{D}_{Y \leftarrow X}$ is equipped with the filtration by order of differential operator, shifted to start in degree $\dim Y - \dim X$. Identifying $\Gr^F \mc{D}_X \cong \mc{O}_{T^*X}$ and $\Gr^F \mc{D}_Y \cong \mc{O}_{T^*Y}$, we have, canonically,
\[ \Gr^F \mc{D}_{Y \leftarrow X} \cong \mc{O}_{T^*Y \times_Y X}  \otimes \omega_{X/Y}\{\dim Y - \dim X\} \]
as a bimodule over $g^{-1}\mc{O}_{T^*Y}$ and $\mc{O}_{T^*X}$, where $\{\cdot\}$ denotes grading shift. We deduce the following.

\begin{prop} \label{prop:associated graded pushforward}
Let $g \colon X \to Y$ be a proper morphism between smooth varieties as above. Define
\[ \Gr(g)_* \colon \mrm{D}^b\coh^{\mb{G}_m}(T^*X) \to \mrm{D}^b\coh^{\mb{G}_m}(T^*Y) \]
by
\[ \Gr(g)_*\mc{F} = \mrm{R} p_{g\bigcdot} \mrm{L} q_g^{\bigcdot}(\mc{F} \otimes \omega_{X/Y})\{\dim Y - \dim X\},\]
where $q_g$ and $p_g$ are the morphisms
\[ T^*X \xleftarrow{q_g} T^*Y \times_Y X \xrightarrow{p_g} T^*Y.\]
Then we have a canonical isomorphism
\[ \Gr^F(g_*\mc{M}) \cong \Gr(g)_*\Gr^F\mc{M}  \]
for $\mc{M} \in \mhm(X)$.
\end{prop}

Now suppose that $\pi \colon \tilde{X} \to X$ is an $H$-torsor, for some torus $H$, and $g \colon X \to Y$ is a projective morphism. For $\lambda \in \mf{h}^*_\mb{R}$ and $\mc{M} \in \mhm_\lambda(\tilde X)$, recall that we have an underlying filtered module $(\mc{M}, F_\bullet)$ over the sheaf of twisted differential operators
\[ \tilde{\mc{D}} \otimes_{S(\mf{h}), \lambda} \mb{C}\]
on $X$. Now,
\[ \Gr^F(\tilde{\mc{D}} \otimes_{S(\mf{h}), \lambda} \mb{C}) \cong \Gr^F\mc{D}_X \cong \mc{O}_{T^*X}.\]
So we have an associated coherent sheaf $\Gr^F\mc{M} \in \coh^{\mb{G}_m}(T^*X)$, and hence a pushforward $\Gr(g)_*\Gr^F\mc{M} \in \mrm{D}^b\coh^{\mb{G}_m}(T^*Y)$.

We have the following relative Kodaira vanishing result for twisted mixed Hodge modules. Define an element $\lambda \in \mf{h}^*_\mb{R}$ to be \emph{$g$-ample} if it is an $\mb{R}_{>0}$-linear combination of elements $\mu \in \mb{X}^*(H)$ such that the line bundle $\mc{O}_X(\mu) = \pi_{\bigcdot}(\mc{O}_{\tilde{X}} \otimes \mb{C}_{-\mu})^H$ is $g$-ample.

\begin{thm} \label{thm:twisted kodaira}
Let $\lambda \in \mf{h}^*_\mb{R}$ and $\mc{M} \in \mhm_\lambda(\tilde X)$. Suppose that $\lambda$ is $g$-ample. Then
\[ \mc{H}^i\Gr(g)_*\Gr^F\mc{M} = 0 \quad \mbox{for $i > 0$}.\]
\end{thm}

In particular, taking $g$ to be the map from $X$ to a point, we recover the absolute statement (Theorem \ref{thm:intro twisted kodaira}) given in the introduction. 

\begin{proof}
We prove the desired vanishing by induction on the dimension of the support of $\mc{M}$.

By tensoring with a line bundle on $Y$ if necessary, we may assume without loss of generality that
\[ \lambda = \sum_i a_i \mu_i, \quad\mu_i \in \mb{X}^*(H),  a_i > 0,\]
such that each $\mc{O}_X(\mu_i)$ is the restriction of a very ample line bundle on a projective completion of $X$. In this case, we may choose $f_i \in \mrm{H}^0(X, \mc{O}_X(\mu_i))$ with divisors $D_i \subset X$ such that $\dim (D_i \cap \operatorname{Supp}\mc{M}) < \dim \operatorname{Supp}\mc{M}$. Setting $D = \bigcup_i D_i$ and $f = \prod_i f_i^{a_i}$, we have that $U := X - D$ is affine over $Y$ and $f \in \Gamma_\mb{R}(\tilde U)^\mon$ satisfies $\varphi(f) = \lambda$, where $\tilde U = \pi^{-1}(U)$.

Consider the family $f^sj^*\mc{M}$, where $j \colon U \to X$ is the inclusion. By Proposition \ref{prop:hyperplanes}, there exist $-1 < s_1 < \cdots < s_n < 0$ such that the extension $j_*f^sj^*\mc{M}$ is clean for $s \in (-1, 0) - \{s_1, \ldots, s_n\}$. We prove the vanishing by induction on $n$.

We first claim that it suffices to prove the desired vanishing for the Hodge module $j_!j^*\mc{M}$. To see this, observe that the tautological map $j_!j^*\mc{M} \to \mc{M}$ gives exact sequences
\[ 0 \to \mc{N} \to j_!j^*\mc{M} \to \mc{P} \to 0 \]
and
\[ 0 \to \mc{P} \to \mc{M} \to \mc{Q} \to 0\]
such that $\mc{N}$ and $\mc{Q}$ are supported in $D \cap \operatorname{Supp}\mc{M}$. So by induction, we have the vanishing for $\mc{N}$ and $\mc{Q}$. The associated long exact sequences for $\mc{H}^i \Gr(g)_*\Gr^F$ show that the vanishing for $j_!j^*\mc{M}$ implies the vanishing for $\mc{P}$, which implies the vanishing for $\mc{M}$ as claimed.

We now proceed with the induction. First, if $n = 0$, then
\[ \Gr^Fj_!j^*\mc{M} \cong \Gr^F j_*f^{-1}j^*\mc{M} \]
by Theorem \ref{thm:semi-continuity} and Proposition \ref{prop:monodromic semi-continuity}. But now $f^{-1}j^*\mc{M}$ is an untwisted mixed Hodge module on $U$, so
\[ \Gr(g)_*\Gr^Fj_*f^{-1}j^*\mc{M} = \Gr^F g_*j_*f^{-1}\mc{M}\]
by Proposition \ref{prop:associated graded pushforward}. But $\mc{H}^i g_*j_*f^{-1}\mc{M} = 0$ for $i > 0$ since $g \circ j$ is affine, so we are done in this case.

If $n > 0$, on the other hand, then we have
\[ \Gr^F j_!j^*\mc{M} \cong \Gr^F j_*f^{s_n}j^*\mc{M}\]
by Theorem \ref{thm:semi-continuity} and Proposition \ref{prop:monodromic semi-continuity} again. But now the $(1 + s_n)\lambda$-twisted Hodge module $j_*f^{s_n}j^*\mc{M}$ satisfies our hypotheses, but has only $n - 1$ reducibility points in the corresponding $f^s$ family. So by induction on $n$, the vanishing holds for $j_*f^{s_n}j^*\mc{M}$, so we are done in this case as well. This completes the proof.
\end{proof}

\subsection{Filtered pushforwards and monodromic Kodaira vanishing} \label{subsec:monodromic kodaira}

Let us now consider pushforward for monodromic mixed Hodge modules. We start with the following observation.

\begin{prop} \label{prop:mhm fully faithful}
The functor
\[ \mrm{D}^b\mhm_{\widetilde{\lambda}}(\tilde X) \to \mrm{D}^b\mhm(\tilde X) \]
induced by the inclusion $\mhm_{\widetilde{\lambda}}(\tilde X) \subset \mhm(\tilde X)$ is fully faithful.
\end{prop}
\begin{proof}
Observe that the triangulated category $\mrm{D}^b\mhm_{\widetilde{\lambda}}(\tilde X)$ is generated by objects of the form $j_!\mc{M}$ where $j \colon \pi^{-1}(U) \to \tilde X$ is the inclusion of the pre-image of an open set $U \subset X$ over which the $H$-torsor is trivial and $\mc{M} \in \mhm_{\widetilde{\lambda}}(\pi^{-1}(U))$. So we may reduce to the case where $\tilde{X} = H \times X$; in this case, full faithfulness follows easily from the explicit description of $\mhm_{\widetilde{\lambda}}(H \times X)$ given in Lemma \ref{lem:monodromic local model}.
\end{proof}

Consider now a commutative diagram
\[
\begin{tikzcd}
\tilde{X} \ar[r, "\tilde{g}"] \ar[d, "\pi"] & \tilde{Y} \ar[d, "\pi"] \\
X \ar[r, "g"] & Y,
\end{tikzcd}
\]
where $g$ is proper, $\tilde{X} \to X$ is an $H$-torsor for some torus $H$, $\tilde{Y} \to Y$ is an $H'$-torsor for another torus $H'$, and $\tilde{g}$ is $H$-equivariant for some surjective morphism $H \to H'$. We will assume for the sake of convenience that $\ker(H \to H')$ is connected. Then for $\lambda \in (\mf{h}')^*_\mb{R} \subset \mf{h}^*_\mb{R}$, we have well-defined pushforward functors
\[ \tilde{g}_!, \tilde{g}_* \colon \mrm{D}^b\mhm_{\widetilde{\lambda}}(\tilde X) \to \mrm{D}^b\mhm_{\widetilde{\lambda}}(\tilde Y) \]
given by restricting the corresponding functors on $\mrm{D}^b\mhm(\tilde X)$.

\begin{prop} \label{prop:monodromic ! vs *}
For $\mc{M} \in \mrm{D}^b\mhm_{\widetilde{\lambda}}(\tilde X)$, we have
\[ \tilde{g}_* \mc{M} \cong \tilde{g}_!\mc{M} \otimes \det(\mf{h}/\mf{h}')^*[\dim \mf{h}/\mf{h}'].\]
\end{prop}
\begin{proof}
Observe that we may factor $\tilde{g}$ as
\[ \tilde{X} \xrightarrow{p} H' \times^H \tilde{X} \xrightarrow{g'} \tilde{Y}.\]
The claim is clearly true for $\tilde{g} = g'$, since $g'$ is proper, so it suffices to treat the case $\tilde{g} = p$. By Lemma \ref{lem:monodromic push/pull} below, we have a natural transformation
\begin{align*}
p^*p_! \mc{M} &= p^!p_!\mc{M} (-\dim \mf{h}/\mf{h}')[-2\dim \mf{h}/\mf{h}'] \\
&= \mc{M} \overset{\mrm{L}}\otimes_{S(\mf{h}(1))} S(\mf{h}'(1))(-\dim \mf{h}/\mf{h}')[-2\dim \mf{h}/\mf{h}'] \to \mc{M} \otimes \det (\mf{h}/\mf{h}')^*[-\dim \mf{h}/\mf{h}'],
\end{align*}
where $S(\mf{h}(1))$ acts on $\mc{M}$ via Proposition \ref{prop:monodromic mhm} \eqref{itm:monodromic mhm 2}, and hence a natural transformation
\begin{equation} \label{eq:monodromic ! vs *}
p_!\mc{M} \to p_*\mc{M} \otimes \det(\mf{h}/\mf{h}')^*[-\dim \mf{h}/\mf{h}'].
\end{equation}
Using the local model of Lemma \ref{lem:monodromic local model}, \eqref{eq:monodromic ! vs *} is an isomorphism when $\tilde{X} = H \times X$; since the statement is local on $X$, \eqref{eq:monodromic ! vs *} is therefore an isomorphism always.
\end{proof}

\begin{lem} \label{lem:monodromic push/pull}
Assume $\lambda \in (\mf{h}')_\mb{R}^*$ and write $p \colon \tilde{X} \to H' \times^H \tilde{X}$ for the natural quotient morphism. Then we have
\[ p^!p_!\mc{M} \cong \mc{M} \overset{\mrm{L}}\otimes_{S(\mf{h}(1))} S(\mf{h}'(1))\]
and 
\[p^*p_*\mc{M} \cong \mc{M} \overset{\mrm{L}}\otimes_{S(\mf{h}(1))} S(\mf{h}'(1)) \otimes \det (\mf{h}/\mf{h}')^*(-\dim \mf{h}/\mf{h}')[-\dim \mf{h}/\mf{h}'] \]
for $\mc{M} \in \mrm{D}^b \mhm_{\widetilde{\lambda}}(\tilde X)$.
\end{lem}
\begin{proof}
Let us prove the claim for $p^!p_!$; the proof for $p^*p_*$ is similar. Since $\mf{h}(1)$ acts via $\mf{h}'(1)$ on $p^!p_!\mc{M}$, the unit $\mc{M} \to p^!p_!\mc{M}$ induces a natural transformation
\begin{equation} \label{eq:monodromic push/pull}
\mc{M} \overset{\mrm{L}}\otimes_{S(\mf{h}(1))} S(\mf{h}'(1)) \to p^!p_!\mc{M}.
\end{equation}
Using the local model of Lemma \ref{lem:monodromic local model}, \eqref{eq:monodromic push/pull} is an isomorphism when $\tilde{X} = H \times X$ and hence in general.
\end{proof}

At the level of filtered $\tilde{\mc{D}}$-modules, the pushforwards $\tilde{g}_!$ and $\tilde{g}_*$ behave as follows. For $(\mc{M}, F_\bullet)$ a coherent filtered module over $\tilde{\mc{D}}_X = \pi_{\bigcdot}(\mc{D}_{\tilde{X}})^H$, let us set
\[ \tilde{g}_\dagger (\mc{M}, F_\bullet) = \mrm{R}g_{\bigcdot}((\tilde{\mc{D}}_{\tilde{g}}, F_\bullet) \overset{\mrm{L}}\otimes_{(\tilde{\mc{D}}_X, F_\bullet)} (\mc{M}, F_\bullet)) \in \mrm{D}^b\coh(\tilde{\mc{D}}_Y, F_\bullet),\]
where
\[ \tilde{\mc{D}}_{\tilde{g}} := \pi_{\bigcdot}(\mc{D}_{\tilde{Y} \leftarrow \tilde{X}})^H, \quad \tilde{\mc{D}}_Y := \pi_{\bigcdot}(\mc{D}_{\tilde Y})^{H'}\]
and $\mrm{D}^b\coh(\tilde{\mc{D}}_Y, F_\bullet)$ denotes the filtered derived category of coherent $\tilde{\mc{D}}_Y$-modules with a good filtration (i.e., the derived category of coherent graded modules over the Rees algebra of $\tilde{\mc{D}}_Y$).

\begin{prop} \label{prop:monodromic filtered pushforward}
For $\mc{M} \in \mrm{D}^b\mhm_{\widetilde{\lambda}}(\tilde X)$ with $\lambda \in (\mf{h}')^*_\mb{R}$ we have
\[ \tilde{g}_\dagger(\mc{M}, F_\bullet) \cong (\tilde{g}_*\mc{M}, F_\bullet) = (\tilde{g}_!\mc{M}, F_\bullet) \otimes \det (\mf{h}/\mf{h}')^*[\dim \mf{h}/\mf{h}'] \]
in the filtered derived category of coherent $\tilde{\mc{D}}_Y$-modules.
\end{prop}
\begin{proof}
Let us factor $\tilde{g}$ as
\[ \tilde{X} \xrightarrow{p} H' \times^H \tilde{X} \xrightarrow{g'} \tilde{Y}.\]
Clearly $\tilde{g}_\dagger = g'_\dagger \circ p_\dagger$, so it suffices to prove the proposition for $\tilde{g} = g'$ and $\tilde{g} = p$. Now, since $g'$ is proper, the result for $\tilde{g} = g'$ is clear from \eqref{eq:hodge proper pushforward}. For $\tilde{g} = p$, we note that
\[ p_\dagger \mc{M} = \mc{M} \overset{\mrm{L}}\otimes_{S(\mf{h})} S(\mf{h}') \otimes \det(\mf{h}/\mf{h}')^*\{-\dim \mf{h}/\mf{h}'\} \]
is the right adjoint to the functor
\[ p^\dagger \colon \mrm{D}^b\coh(\tilde{\mc{D}}_X \otimes_{S(\mf{h})} S(\mf{h}'), F_\bullet) \to \mrm{D}^b\coh(\tilde{\mc{D}}_X, F_\bullet) \]
given by restriction of scalars and a cohomological shift by $\dim \mf{h}/\mf{h}'$. Clearly
\[ p^\dagger(\mc{N}, F_\bullet) = (p^*\mc{N}, F_\bullet),\]
so, passing to adjoints, we have a natural transformation
\[ (p_*\mc{M}, F_\bullet) \to p_\dagger (\mc{M}, F_\bullet).\]
Arguing locally using Lemma \ref{lem:monodromic local model}, this is an isomorphism.
\end{proof}

Note that if $\mc{M}$ is $\lambda$-monodromic with $\lambda \not\in (\mf{h}')^*$ then $g_\dagger (\mc{M}, F_\bullet) = 0$ in $\mrm{D}^b\coh(\tilde{\mc{D}}_Y)$ but not necessarily in $\mrm{D}^b\coh(\tilde{\mc{D}}_Y, F_\bullet)$; i.e., we might have $\Gr^Fg_\dagger (\mc{M}, F_\bullet) \neq 0$. In particular, $\tilde{g}_\dagger(\mc{M}, F_\bullet)$ does not underlie a complex of mixed Hodge modules. However, we do have the following monodromic version of the Kodaira vanishing theorem.

\begin{thm} \label{thm:monodromic kodaira}
Assume $\mc{M} \in \mhm_{\widetilde{\lambda}}(\tilde X)$ with $\lambda \in \mf{h}^*_\mb{R}$ $g$-ample. Then
\[ \mc{H}^i \Gr^F\tilde{g}_\dagger(\mc{M}, F_\bullet) = 0 \quad \text{for $i > 0$}.\]
\end{thm}
\begin{proof}
We may assume without loss of generality that $\mc{M} \in \mhm_\lambda(\tilde{X})$ and that $H' = \{1\}$. Then, since $\mf{h}$ acts trivially on $\Gr^F\mc{M}$, we have
\begin{align*}
\Gr^F \tilde{g}_\dagger (\mc{M}, F_\bullet) &= \Gr(g)_*(\mb{C} \overset{\mrm{L}}\otimes_{S(\mf{h})} \Gr^F\mc{M})\{-\dim \mf{h}\} \\
& = \bigoplus_{n \geq 0} \wedge^n \mf{h}\{n - \dim \mf{h}\}[n] \otimes \Gr(g)_*(\Gr^F\mc{M}).
\end{align*}
The claim is now immediate from Theorem \ref{thm:twisted kodaira}.
\end{proof}

\section{Localization theory for Hodge modules} \label{sec:localization}

In this section, we consider the interaction between Hodge theory and Beilinson-Bernstein localization \cite{BB1, BB2}.

Recall that Beilinson-Bernstein localization sets up equivalences of categories between twisted (or monodromic) $\mc{D}$-modules on a flag variety and certain categories of representations of a Lie algebra, with the functor from $\mc{D}$-modules to representations given by taking global sections. The theory of mixed Hodge modules, on the other hand, endows certain $\mc{D}$-modules with extra structures: either a Hodge filtration and a polarization (for polarized Hodge modules) or a weight filtration, a Hodge filtration, and a conjugate Hodge filtration (for mixed Hodge modules). For each of these structures, one can ask first whether it makes sense to take ``global sections'', and second, whether the resulting functor has reasonable properties.

For the weight filtration, the answer to both questions is obvious: since it is a filtration by sub-objects, Beilinson and Bernstein's results already imply that the weight filtration behaves well under global sections. Similarly, it is clear what ``global sections'' should mean for the Hodge filtration, but since this is \emph{not} a filtration by sub-objects, new results are required to show that this operation is well-behaved. We formulate these new results and their consequences in \S\S\ref{subsec:hodge localization}--\ref{subsec:hodge localization functors}, deferring the proofs until \S\S\ref{sec:vanishing}--\ref{sec:hodge generation}. Finally, for the polarization or conjugate Hodge filtration, we explain how to take global sections in \S\ref{subsec:polarization} using a construction of Schmid and Vilonen \cite{SV}. Putting everything together, we obtain in \S\ref{subsec:mhm globalization} a functor from mixed Hodge modules to a category of Lie algebra representations with extra structure, which we show is fully faithful under appropriate hypotheses.

\begin{rmk}
We work here with localization for monodromic $\mc{D}$-modules as in \cite[\S 3.3]{BB2} rather than the version for twisted $\mc{D}$-modules in \cite{BB1}. This is both more general and, in many places, technically more convenient: for example, the monodromic localization functor always has bounded cohomological dimension, which is not the case for the twisted localization functor at singular infinitesimal characters.
\end{rmk}

\subsection{Notation and conventions} \label{subsec:flag variety}

From now on, we fix a connected reductive group $G$ over $\mb{C}$. We write $\mc{B}$ for the flag variety of $G$, $H$ for the universal Cartan of $G$ and $\pi \colon \tilde{\mc{B}} \to \mc{B}$ for the base affine space, the canonical $H$-torsor over $\mc{B}$. We will write $\Phi \subset \mb{X}^*(H)$ and $\check \Phi \subset \mb{X}_*(H)$ for the set of roots and coroots respectively.

As in \cite[\S 2.2]{DV1}, our conventions for positive and negative roots are as follows. If we fix a maximal torus and Borel subgroup $T \subset B \subset G$ then $T$ is identified naturally with $H = B/N$, where $N$ is the unipotent radical of $B$, $\mc{B}$ with $G/B$ and $\tilde{\mc{B}}$ with $G/N$. We define the set $\Phi_- \subset \mb{X}^*(H)$ of \emph{negative} roots to be the characters of $T = H$ acting on $\mf{n} = \mrm{Lie}(N)$. With this convention, a character $\lambda \in \mb{X}^*(H)$ is dominant if and only if the associated line bundle $\mc{O}(\lambda) = \pi_{\bigcdot}(\mc{O}_{\tilde{\mc{B}}} \otimes \mb{C}_{-\lambda})^H$ on $\mc{B}$ is semi-ample.

We are interested in monodromic $\mc{D}$-modules and mixed Hodge modules on $\tilde{\mc{B}}$. In this setting, it is convenient to introduce a $\rho$-shift into the indexing for the monodromicity $\lambda$. To distinguish this from our convention for a general variety, we will introduce the following special notation.
We write
\[ \mc{D}_\lambda = \mc{D}_{\mc{B}, \lambda} := \tilde{\mc{D}}_{\mc{B}} \otimes_{S(\mf{h}), \lambda - \rho} \mb{C}\]
for $\lambda \in \mf{h}^*$, where $\rho$ is half the sum of the positive roots of $G$. Note that we have
\[ \Mod(\mc{D}_\lambda) = \Mod_{\lambda - \rho}(\mc{D}_{\tilde{\mc{B}}}).\]
We will also write
\[ \Mod(\mc{D}_{\widetilde{\lambda}}) := \Mod_{\widetilde{\lambda - \rho}}(\mc{D}_{\tilde{\mc{B}}}) \subset \Mod(\tilde{\mc{D}}) = \Mod_\mon(\mc{D}_{\tilde{\mc{B}}}).\]
For mixed Hodge modules, we write
\[ \mhm(\mc{D}_\lambda) = \mhm_{\lambda - \rho}(\tilde{\mc{B}}), \quad \mhm(\mc{D}_{\widetilde{\lambda}}) = \mhm_{\widetilde{\lambda - \rho}}(\tilde{\mc{B}})\]
\[ \mbox{and} \quad \mhm(\tilde{\mc{D}}) = \mhm_\mon(\tilde{\mc{B}}).\]
Finally, if an algebraic group $K$ acts compatibly on $\tilde{\mc{B}}$ and $H$ and $\lambda \in (\mf{h}^*)^K$, we write
\[ \Mod(\mc{D}_\lambda, K) \subset \Mod(\mc{D}_{\widetilde{\lambda}}, K) \subset \Mod(\tilde{\mc{D}}, K) := \Mod^K_\mon(\mc{D}_{\tilde{\mc{B}}})\]
and
\[ \mhm(\mc{D}_\lambda, K) \subset \mhm(\mc{D}_{\widetilde{\lambda}}, K) \subset \mhm(\tilde{\mc{D}}, K) := \mhm^K_\mon(\tilde{\mc{B}}) \]
for the corresponding $K$-equivariant categories.

\subsection{Global properties of the Hodge filtration} \label{subsec:hodge localization}

The fundamental fact underlying Beilinson and Bernstein's original localization theorem is that, when $\lambda \in \mf{h}^*$ is regular and integrally dominant, the flag variety $\mc{B}$ is \emph{$\mc{D}_\lambda$-affine}. That is, if $\mc{M} \in \Mod(\mc{D}_\lambda)$ is a $\mc{D}_\lambda$-module (or more generally, if $\mc{M} \in \Mod(\mc{D}_{\widetilde{\lambda}})$) then
\begin{enumerate}
\item $\mrm{H}^i(\mc{B}, \mc{M}) = 0$ for $i > 0$, and
\item $\mc{M}$ is globally generated, i.e., $\mc{M} = \tilde{\mc{D}}\cdot \Gamma(\mc{M})$.
\end{enumerate}
The main point of this section is that the Hodge filtration on a mixed Hodge module has similarly nice global properties. In the statement below, we say that $\lambda \in \mf{h}^*_\mb{R}$ is \emph{dominant} if $\langle \lambda, \check\alpha\rangle \geq 0$ for all positive coroots $\check\alpha \in \check\Phi_+$.

\begin{thm} \label{thm:filtered exactness}
Assume that $\lambda \in \mf{h}^*_\mb{R}$ is dominant. Then for any $\mc{M} \in \mhm(\mc{D}_{\widetilde{\lambda}})$, we have
\[ \mrm{H}^i(\mc{B}, F_p \mc{M}) = 0 \quad \mbox{for all $i > 0$ and all $p$}.\]
\end{thm}

We give the proof of Theorem \ref{thm:filtered exactness} in \S\ref{sec:vanishing}. We also have the following global generation result.

\begin{thm} \label{thm:hodge generation}
Assume $\lambda \in \mf{h}^*_\mb{R}$ is dominant and $\mc{M} \in \mhm(\mc{D}_{\widetilde{\lambda}})$. If $\mc{M}$ is globally generated as a $\tilde{\mc{D}}$-module, then so is the Hodge filtration $F_\bullet \mc{M}$. That is,
\[ F_p\mc{M} = \sum_{p_1 + p_2 \leq p} F_{p_1} \tilde{\mc{D}} \cdot \Gamma(F_{p_2}\mc{M}) \quad \text{for all $p$}.\]
\end{thm}

We give the proof of Theorem \ref{thm:hodge generation} in \S\ref{sec:hodge generation}.

\subsection{Localization functors and the Hodge filtration} \label{subsec:hodge localization functors}

In Beilinson and Bernstein's work, cohomology vanishing and global generation are used to establish good properties of the global sections functor $\Gamma$ and its left adjoint $\Delta$. We explain here the consequences of Theorems \ref{thm:filtered exactness} and \ref{thm:hodge generation} for the interaction of these functors with the Hodge filtration.

Recall that taking global sections defines a functor, called \emph{globalization},
\[ \Gamma \colon \Mod(\mc{D}_{\widetilde{\lambda}}) \to \Mod(U(\mf{g}))_{\widetilde{\chi_\lambda}},\]
where $\chi_\lambda \colon Z(U(\mf{g})) \to \mb{C}$ is the character corresponding to $\lambda$ under the Harish-Chandra isomorphism, and $\Mod(U(\mf{g}))_{\widetilde{\chi_\lambda}}$ is the category of $U(\mf{g})$-modules on which $Z(U(\mf{g}))$ acts with generalized character $\chi_\lambda$. The functor $\Gamma$ has a left adjoint
\[ \Delta = \Delta_{\widetilde{\lambda}} \colon \Mod(U(\mf{g}))_{\widetilde{\chi_\lambda}} \to  \Mod(\mc{D}_{\widetilde{\lambda}}),\]
called \emph{localization}, given as follows. For $V \in \Mod(U(\mf{g}))_{\widetilde{\chi_\lambda}}$, consider the $\tilde{\mc{D}}$-module $\tilde{\mc{D}} \otimes_{U(\mf{g})} V$. Under the action of the center $S(\mf{h}) \subset \tilde{\mc{D}}$, this decomposes into generalized eigenspaces as
\[ \tilde{\mc{D}} \otimes_{U(\mf{g})} V = \bigoplus_{\mu \in W\lambda} (\tilde{\mc{D}} \otimes_{U(\mf{g})} V)_{\widetilde{\mu}},\]
where, $(-)_{\tilde{\mu}}$ denotes the $(\mu - \rho)$-generalized eigenspace under the action of $S(\mf{h})$. The localization functor is given by extracting the appropriate summand:
\[\Delta(V) = (\tilde{\mc{D}} \otimes_{U(\mf{g})} V)_{\widetilde{\lambda}}.\]
The main result of Beilinson and Bernstein is that, when $\lambda$ is regular and integrally dominant, the adjoint functors $\Gamma$ and $\Delta$ form inverse equivalences of categories.

Consider now the filtered version of this story. Let $F_\bullet U(\mf{g})$ be the PBW filtration on $U(\mf{g})$, $\Mod(U(\mf{g}), F_\bullet)_{\widetilde{\chi_\lambda}}$ denote the category of filtered $U(\mf{g})$-modules $V$ with generalized infinitesimal character $\chi_\lambda$, and $\Mod(\mc{D}_{\widetilde{\lambda}}, F_\bullet)$ denote the category of filtered $\tilde{\mc{D}}$-modules such that the underlying $\tilde{\mc{D}}$-module lies in $\Mod(\mc{D}_{\widetilde{\lambda}})$. The globalization functor lifts to
\[ \Gamma \colon \Mod(\mc{D}_{\widetilde{\lambda}}, F_\bullet) \to \Mod(U(\mf{g}), F_\bullet)_{\widetilde{\chi_\lambda}} \]
by setting $F_p \Gamma(\mc{M}) = \Gamma(F_p \mc{M})$. This has a left adjoint sending $(V, F_\bullet)$ to the $\tilde{\mc{D}}$-module $\Delta(V)$ with filtration given by the image
\[ F_p \Delta(V) = \operatorname{image}\left(\sum_{p_1 + p_2 \leq p} F_{p_1} \tilde{\mc{D}} \otimes F_{p_2} V \to \tilde{\mc{D}} \otimes_{U(\mf{g})} V \to \Delta(V)\right) \]
of the natural filtration on $\tilde{\mc{D}} \otimes_{U(\mf{g})} V$ under the projection to the $(\lambda - \rho)$-generalized eigenspace. Note that, while $\Delta(V)$ is a summand of $\tilde{\mc{D}} \otimes_{U(\mf{g})} V$, it is \emph{not} a filtered summand in general.

Now, for $\lambda \in \mf{h}^*_\mb{R}$, let us write $\mc{C}_{\widetilde{\lambda}} \subset \Mod(\mc{D}_{\widetilde{\lambda}}, F_\bullet)$ for the essential image of the forgetful functor
\[ \mhm(\mc{D}_{\widetilde{\lambda}}) \to \Mod(\mc{D}_{\widetilde{\lambda}}, F_\bullet).\]
We have the following corollary of Theorem \ref{thm:filtered exactness}.

\begin{cor} \label{cor:filtered exactness}
Assume $\lambda \in \mf{h}^*_\mb{R}$ is dominant. Then the restriction of $\Gamma$ to $\mc{C}_{\widetilde{\lambda}}$ is filtered exact. That is, if
\[ 0 \to (\mc{M}, F_\bullet) \to (\mc{N}, F_\bullet) \to (\mc{P}, F_\bullet) \to 0 \]
is a strict exact sequence in $\mc{C}_{\widetilde{\lambda}}$, then the sequence in $\Mod(U(\mf{g}))_{\widetilde{\chi_\lambda}}$
\[ 0 \to \Gamma(\mc{M}, F_\bullet) \to \Gamma(\mc{N}, F_\bullet) \to \Gamma(\mc{P}, F_\bullet)  \to 0 \]
is also strict exact.
\end{cor}

Similarly, Theorem \ref{thm:hodge generation} and Beilinson and Bernstein's results imply the following.

\begin{cor} \label{cor:hodge full faithful}
Assume $\lambda \in \mf{h}^*_\mb{R}$ is regular dominant. Then, for any $(\mc{M}, F_\bullet) \in \mc{C}_{\widetilde{\lambda}}$, the natural morphism
\[ \Delta \Gamma(\mc{M}, F_\bullet) \to (\mc{M}, F_\bullet) \]
is a filtered isomorphism. Hence, the functor
\[ \Gamma \colon \mc{C}_{\widetilde{\lambda}} \to \Mod(U(\mf{g}), F_\bullet)_{\widetilde{\chi}_\lambda} \]
is fully faithful.
\end{cor}

More generally, we have the following.

\begin{thm} \label{thm:localization hodge module}
Let $\mc{M} \in \mhm(\mc{D}_{\widetilde{\lambda}})$ with $\lambda \in \mf{h}^*_{\mb{R}}$ dominant. Then the filtered $\tilde{\mc{D}}$-module $\Delta \Gamma(\mc{M}, F_\bullet)$ underlies an object in $\mhm(\mc{D}_{\widetilde{\lambda}})$ and the morphism $\Delta \Gamma(\mc{M}, F_\bullet) \to (\mc{M}, F_\bullet)$ underlies a morphism of mixed Hodge modules.
\end{thm}

We give the proof of Theorem \ref{thm:localization hodge module}, which implies Theorem \ref{thm:hodge generation}, in \S\ref{sec:hodge generation}.

\subsection{Integrating the polarization and the Schmid-Vilonen conjecture} \label{subsec:polarization}

We now consider how to take global sections of polarized Hodge modules and recall the main conjecture of \cite{SV} concerning this process.

The core construction (due to \cite{SV}) is as follows. Fix once and for all a maximal compact subgroup $U_\mb{R} \subset G$ and a base point in $\tilde{\mc{B}}$ defining an equivariant embedding $U_\mb{R} \subset \tilde{\mc{B}}$ (see \S\ref{subsec:choices} for a discussion of the dependence of the construction on these choices). Suppose that $\lambda \in \mf{h}^*_\mb{R}$, $\mc{M}, \mc{M}' \in \Mod(\mc{D}_{\widetilde{\lambda}})$ and
\[ \mf{s} \colon \mc{M}_{\tilde{\mc{B}}} \otimes \overline{\mc{M}'_{\tilde{\mc{B}}}} \to \Db_{\tilde{\mc{B}}}\]
is a $\mc{D}_{\tilde{\mc{B}}} \otimes \overline{\mc{D}_{\tilde{\mc{B}}}}$-linear distribution-valued sesquilinear pairing. Here we write $\mc{M}_{\tilde{\mc{B}}}$, $\mc{M}'_{\tilde{\mc{B}}}$ for the monodromic $\mc{D}$-modules on $\tilde{\mc{B}}$ corresponding to $\mc{M}$ and $\mc{M}'$. For
\[ m \in \Gamma(\mc{M}) = \Gamma(\tilde{\mc{B}}, \mc{M}_{\tilde{\mc{B}}})^H \quad \text{and} \quad m' \in \Gamma(\mc{M}') = \Gamma(\tilde{\mc{B}}, \mc{M}'_{\tilde{\mc{B}}})^H,\]
the distribution $\mf{s}(m, \overline{m}')$ is $(\lambda - \rho)$-monodromic, i.e.,
\[ (i(h) - (\lambda - \rho))^n \mf{s}(m, \overline{m}') = (\overline{i(h)} - (\lambda - \rho)(\bar{h}))^n \mf{s}(m, \overline{m}') = 0,\]
for all $h \in \mf{h}$ and some $n > 0$, where $i \colon S(\mf{h}) \to \mc{D}_{\tilde{\mc{B}}}$ denotes the inclusion coming from the action of $H$. It therefore makes sense to restrict $\mf{s}(m, \overline{m}')$ to $U_\mb{R} \subset \tilde{\mc{B}}$; we set
\[ \Gamma(\mf{s})(m, \overline{m}') = \int_{U_\mb{R}} \mf{s}(m, \overline{m}')|_{U_\mb{R}},\]
where we take the integral with respect to the unique invariant volume form on $U_\mb{R}$ of volume $1$. This defines a $\mf{u}_\mb{R}:= \mrm{Lie}(U_\mb{R})$-invariant pairing
\[ \Gamma(\mf{s}) \colon \Gamma(\mc{M}) \otimes \overline{\Gamma(\mc{M}')} \to \mb{C}.\]

Now suppose that $\mc{M} \in \mhm(\mc{D}_{\widetilde{\lambda}})$ is a polarized Hodge module of weight $w$ with polarization $S$. Recall that the polarization is defined as a distribution-valued Hermitian form on $\mc{M}_{\tilde{\mc{B}}}$. The \emph{globalization} of $\mc{M}$ is the triple
\[ (\Gamma(\mc{M}), \Gamma(F_\bullet \mc{M}), (-1)^{\dim \mc{B}} \Gamma(S)),\]
a $U(\mf{g})$-module equipped with a good filtration and a $\mf{u}_\mb{R}$-invariant Hermitian form $\Gamma(S)$. The main conjecture of \cite{SV} is:

\begin{conj}[{\cite[Conjecture 5.12]{SV}}] \label{conj:schmid-vilonen}
Assume $\lambda \in \mf{h}^*_\mb{R}$ is dominant. Then the tuple $(\Gamma(\mc{M}), \Gamma(F_\bullet \mc{M}), (-1)^{\dim \mc{B}}\Gamma(S))$ is a polarized Hodge structure of weight $w$. That is, $\Gamma(S)$ is $(-1)^{p + w - \dim \mc{B}}$-definite on the subspace
\[ \Gamma(F_p\mc{M}) \cap \Gamma(F_{p - 1}\mc{M})^\perp \subset \Gamma(\mc{M})\]
for all $p$.
\end{conj}

We do not know how to prove Conjecture \ref{conj:schmid-vilonen} in general, although it has been explicitly checked in some cases (see e.g., \cite{SV2} and \cite{chaves}). However, we do have the following much weaker statement.

\begin{prop} \label{prop:polarization non-degenerate}
In the setting above, assume that $F_{p_0} \mc{M} \neq 0$ and $F_{p_0 - 1} \mc{M} = 0$. Then $\Gamma(S)$ is non-degenerate on $\Gamma(\mc{M})$ and $(-1)^{p_0 + w - \dim \mc{B}}$-definite on $\Gamma(F_{p_0}\mc{M})$.
\end{prop}
\begin{proof}
Since polarized Hodge modules are semi-simple, we can assume without loss of generality that $\mc{M}$ is irreducible. That $\Gamma(S)$ is definite of the correct sign on the lowest piece $\Gamma(F_{p_0}\mc{M})$ of the Hodge filtration follows from the argument of \cite[Proposition 4.7]{DV1}, cf., also \cite[Theorem A]{SY1}. To deduce non-degeneracy, note that either $\Gamma(\mc{M}) = 0$, in which case the claim is vacuous, or else $\mc{M}$ is globally generated and $\Gamma(\mc{M})$ is an irreducible $U(\mf{g})$-module. In the latter case, $\Gamma(F_{p_0}\mc{M}) \neq 0$ by Theorem \ref{thm:hodge generation}, so $\Gamma(S)$ is non-zero and hence non-degenerate as claimed.
\end{proof}

\begin{rmk}
In fact, the non-degeneracy asserted by Proposition \ref{prop:polarization non-degenerate} is true in a much more general context. Suppose that $\lambda \in \mf{h}^*$ is integrally dominant and $\mf{s}$ is a non-degenerate sesquilinear pairing between a globally generated $\mc{D}_\lambda$-module $\mc{M}$ and a globally generated $\mc{D}_{\bar{\lambda}}$-module $\mc{M}'$. Let $\Db_\lambda$ (resp., $\Omega(\lambda)$) be the space of distributions (resp., smooth top forms) on $\tilde{\mc{B}}$ annihilated on the left (resp., right) by the vector fields $i(h) - (\lambda - \rho)(h)$ and $\overline{i(h)} - (\lambda - \rho)(\bar{h})$ for $h \in \mf{h}$. Then $\Db_\lambda = \Omega(\lambda)^*$ is the continuous dual of $\Omega(\lambda)$ and the integral over $U_\mb{R}$ is given by evaluation at the unique (up to a scalar) $U_\mb{R}$-invariant vector $\eta_0 \in \Omega(\lambda)$. For $\lambda$ integrally dominant, the space $\eta_0 U(\mf{g})$ is dense in $\Omega(\lambda)$ and hence the pairing $\Gamma(\mf{s})$ between $\Gamma(\mc{M})$ and $\Gamma(\mc{M}')$ is non-degenerate.
\end{rmk}

\subsection{Globalizing mixed Hodge modules} \label{subsec:mhm globalization}

Let us now consider how to globalize a general mixed object $\mc{M} \in \mhm(\mc{D}_{\widetilde{\lambda}})$. Let $\mhm^{\mathit{weak}}(U(\mf{g}))_{\widetilde{\chi_\lambda}}$ denote the category of tuples $(V, W_\bullet V, F_\bullet V, \bar{F}_\bullet V)$, where $V$ is a $U(\mf{g})$-module with generalized infinitesimal character $\chi_\lambda$, $W_\bullet V$ is a finite increasing filtration by submodules, $F_\bullet V$ is a good filtration and $\bar{F}_\bullet V$ is an increasing filtration satisfying $\bar{F}_p V = V$ for $p \gg 0$ and $F_p U(\mf{g}) \cdot \bar{F}_q V \subset \bar{F}_{p + q} V$. We define a functor
\begin{equation} \label{eq:hodge global sections}
\Gamma \colon \mhm(\mc{D}_{\widetilde{\lambda}}) \to \mhm^{\mathit{weak}}(U(\mf{g}))_{\widetilde{\chi_\lambda}}
\end{equation}
as follows.

First, for $\mc{M} \in \mhm(\mc{D}_{\widetilde{\lambda}})$, the $U(\mf{g})$-module underlying $\Gamma(\mc{M})$ is, of course, the space of global sections of the $\tilde{\mc{D}}$-module underlying $\mc{M}$. The Hodge and weight filtrations are defined by
\[ F_p \Gamma(\mc{M}) := \Gamma(F_p \mc{M}) \quad \text{and} \quad W_w\Gamma(\mc{M}) = \Gamma(W_w\mc{M}).\]
Finally, recalling from \S\ref{subsec:monodromic} the filtration $F_\bullet \mc{M}^h$ on the Hermitian dual of $\mc{M}$, we define the filtration $\bar{F}_\bullet \Gamma(\mc{M})$ by
\[ \bar{F}_q\Gamma(\mc{M}) := \Gamma(F_{- q - 1}\mc{M}^h)^\perp,\]
where $(-)^\perp$ denotes orthogonal complement with respect to the integral $\Gamma(\mf{s})$ of the tautological pairing
\[ \mf{s} \colon \mc{M}_{\tilde{\mc{B}}} \otimes \overline{\mc{M}^h_{\tilde{\mc{B}}}} \to \Db_{\tilde{\mc{B}}}.\]

\begin{thm} \label{thm:hodge localization}
We have the following.
\begin{enumerate}
\item \label{itm:hodge localization 1} If $\lambda \in \mf{h}^*_\mb{R}$ is dominant, then the functor \eqref{eq:hodge global sections} is filtered exact with respect to all three filtrations, i.e., the functors $W_w\Gamma$, $F_p\Gamma$ and $\bar{F}_q\Gamma$ are exact for $w, p, q \in \mb{Z}$.
\item \label{itm:hodge localization 2} If $\lambda \in \mf{h}^*_\mb{R}$ is regular dominant, then the functor \eqref{eq:hodge global sections} is fully faithful.
\end{enumerate}
\end{thm}
\begin{proof}
To prove \eqref{itm:hodge localization 1}, note that filtered exactness for $F_\bullet$ (resp., $W_\bullet$) follows from the fact that every exact sequence of mixed Hodge modules is strict together with Corollary \ref{cor:filtered exactness} (resp., Beilinson-Bernstein's result that $\Gamma$ is exact on the category $\Mod(\mc{D}_{\widetilde{\lambda}})$). For $\bar{F}_\bullet$, by Corollary \ref{cor:filtered exactness} we have a short exact sequence
\[ 0 \to \Gamma(F_{-q - 1}\mc{P}^h) \to \Gamma(F_{-q - 1}\mc{N}^h) \to \Gamma(F_{-q-1}\mc{M}^h) \to 0\]
of finite dimensional vector spaces. Passing to duals, we have an exact sequence
\[ 0 \to \Gamma(F_{- q - 1}\mc{M}^h)^* \to \Gamma(F_{- q - 1}\mc{N}^h)^* \to \Gamma(F_{-q-1}\mc{P}^h)^* \to 0.\]
But by Lemma \ref{lem:pairing non-degenerate}, we have
\[ \Gamma(F_{- q - 1}\mc{M}^h)^* \cong \frac{\Gamma(\mc{M})}{\bar{F}_q\Gamma(\mc{M})}, \quad \text{etc},\]
so the desired exactness for $\bar{F}_\bullet$ follows.

To prove \eqref{itm:hodge localization 2}, suppose that $\lambda$ is regular dominant and we are given $\mc{M}, \mc{N} \in \mhm(\mc{D}_{\widetilde{\lambda}})$ and a morphism
\[ f \colon (\Gamma(\mc{M}), W_\bullet, F_\bullet, \bar{F}_\bullet) \to (\Gamma(\mc{N}), W_\bullet, F_\bullet, \bar{F}_\bullet).\]
Now, $f$ is the image of the morphism of $\tilde{\mc{D}}$-modules
\[ \Delta(f) \colon \Delta\Gamma(\mc{M}) = \mc{M} \to \Delta\Gamma(\mc{N}) = \mc{N},\]
which respects the weight filtrations (obviously) and the Hodge filtrations (by Corollary \ref{cor:hodge full faithful}). Since $\Delta(f)$ is clearly the only such morphism, it remains to show that it lifts to a morphism of mixed Hodge modules, i.e., that the Hermitian dual morphism
\[ \Delta(f)^h \colon \mc{N}^h \to \mc{M}^h \]
respects the Hodge filtrations. But using Lemma \ref{lem:pairing non-degenerate} again to write
\[ \Gamma(F_{-q-1} \mc{M}^h)= (\bar{F}_q\Gamma(\mc{M}))^\perp,\]
we see that $\Gamma(\Delta(f)^h)$ respects the Hodge filtrations on $\Gamma(\mc{M}^h)$ and $\Gamma(\mc{N}^h)$, and hence $\Delta(f)^h$ itself respects the Hodge filtrations on $\mc{M}^h$ and $\mc{N}^h$ by Corollary \ref{cor:hodge full faithful}.
\end{proof}

\begin{lem} \label{lem:pairing non-degenerate}
Let $\mc{M} \in \mhm(\mc{D}_{\widetilde{\lambda}})$ with $\lambda \in \mf{h}^*_\mb{R}$ dominant. Then the pairing $\Gamma(\mf{s})$ between $\Gamma(\mc{M})$ and $\Gamma(\mc{M}^h)$ is non-degenerate.
\end{lem}
\begin{proof}
We may assume without loss of generality that $\mc{M}$ is pure, hence polarizable. In this case, we may identify $\mc{M}^h$ with $\mc{M}$ and $\mf{s}$ with a polarization $S$, so the lemma holds by Proposition \ref{prop:polarization non-degenerate}.
\end{proof}

Theorem \ref{thm:hodge localization} implies that, for $\lambda$ regular dominant, the essential image $\Gamma(\mhm(\mc{D}_{\widetilde{\lambda}}))$ of \eqref{eq:hodge global sections} is an abelian category, in which every morphism is strict with respect to all three filtrations. This suggests the following conjecture.

\begin{conj} \label{conj:mhs}
Assume $\lambda \in \mf{h}^*_\mb{R}$ is dominant. Then the functor \eqref{eq:hodge global sections} factors through the full subcategory in $\mhm^{\mathit{weak}}(U(\mf{g}))_{\widetilde{\chi_\lambda}}$ consisting of (infinite-dimensional) mixed Hodge structures.
\end{conj}

Conjecture \ref{conj:mhs} is an easy consequence of Conjecture \ref{conj:schmid-vilonen}.

\subsection{Dependence on choices} \label{subsec:choices}

To construct the integral pairings $\Gamma(\mf{s})$, and hence the globalization functor \eqref{eq:hodge global sections} for mixed Hodge modules, we had to choose both a maximal compact subgroup $U_\mb{R} \subset G$ and a $U_\mb{R}$-orbit in $\tilde{\mc{B}}$. Let us conclude this section by commenting on the dependence on these choices.

The dependence on the compact form $U_\mb{R}$ is very serious: for example, if one knows Conjecture \ref{conj:schmid-vilonen} for a given Hodge module and compact form, it is not at all clear that the conjecture also holds for the same Hodge module and other compact forms.

The dependence on the $U_\mb{R}$-orbit in $\tilde{\mc{B}}$ is mild, but non-trivial. For twisted Hodge modules, the integral pairings depend on the choice of orbit only up to a positive real scalar. So, for instance, the form $\Gamma(S)$ on the global sections of a polarized Hodge module and the globalization functor on $\mhm(\mc{D}_\lambda)$ are essentially independent of this choice. For genuinely monodromic Hodge modules, however, there is a subtle dependence of the globalization functor on the $U_\mb{R}$-orbit. This comes from a familiar phenomenon in Hodge theory: if we restrict to a fiber of $\tilde{\mc{B}} \to \mc{B}$, an object in $\mhm(\mc{D}_{\widetilde{\lambda}})$ becomes a unipotent variation of mixed Hodge structure (see, e.g., \cite{hain-zucker}) tensored with a rank $1$ local system on $H$. Such variations of mixed Hodge structure are trivial on their pure subquotients, but the extensions generally vary from point to point. As one varies the $U_\mb{R}$-orbit, these unipotent variations are reflected in a variation of the filtration $\bar{F}_\bullet$ on global sections.

\section{Cohomology vanishing} \label{sec:vanishing}

In this section, we give the proof of Theorem \ref{thm:filtered exactness}.

\subsection{Outline of the proof} \label{subsec:filtered exactness outline}

Let us first outline the strategy for proving Theorem \ref{thm:filtered exactness}. The central idea is to study the functor
\begin{equation} \label{eq:localize globalize}
\tilde{\mc{D}} \overset{\mrm{L}}\otimes_{U(\mf{g})} \mrm{R}\Gamma(\cdot) \colon \mrm{D}^b\coh(\tilde{\mc{D}}, F_\bullet) \to \mrm{D}^b\coh(\tilde{\mc{D}}, F_\bullet)
\end{equation}
rather than the functor of global sections itself. We begin with the following observation.

\begin{lem} \label{lem:localize globalize vanishing}
Let $\mc{M} \in \coh(\tilde{\mc{D}}, F_\bullet)$ be a coherent filtered $\tilde{\mc{D}}$-module (i.e., a coherent $\tilde{\mc{D}}$-module equipped with a good filtration). Then
\begin{equation} \label{eq:localize globalize vanishing 1}
 \mrm{H}^i(\mc{B}, F_p\mc{M}) = 0 \quad \text{for $i > 0$ and all $p$}
\end{equation}
if and only if
\begin{equation} \label{eq:localize globalize vanishing 2}
 \mc{H}^i\Gr^F(\tilde{\mc{D}} \overset{\mrm{L}}\otimes_{U(\mf{g})} \mrm{R}\Gamma(\mc{M})) = 0 \quad \text{for $i > 0$}.
 \end{equation}
\end{lem}
\begin{proof}
Observe that $\Gr^F U(\mf{g}) = \mc{O}_{\mf{g}^*}$ and $\Gr^F \tilde{\mc{D}} = \mc{O}_{\tilde{\mf{g}}^*}$, where $\tilde{\mf{g}}^* = T^*\tilde{\mc{B}}/H$ is the Grothendieck-Springer resolution. So $\Gr^F\mc{M} \in \coh(\tilde{\mf{g}}^*)$ and
\[ \Gr^F(\tilde{\mc{D}} \overset{\mrm{L}} \otimes_{U(\mf{g})} \mrm{R}\Gamma(\mc{M})) = \mrm{L}\tilde{\mu}^{\bigcdot}\mrm{R}\tilde{\mu}_{\bigcdot} \Gr^F\mc{M},\]
where $\tilde{\mu} \colon \tilde{\mf{g}}^* \to \mf{g}^*$ is the Grothendieck-Springer map. Since $\tilde{\mu}$ is surjective, $\mrm{L}\tilde{\mu}^{\bigcdot}$ preserves the top degree of non-zero cohomology of any bounded complex of coherent sheaves. So \eqref{eq:localize globalize vanishing 2} holds if and only if
\[ \mrm{R}^i \tilde{\mu}_{\bigcdot} \Gr^F\mc{M} = 0 \quad \text{for $i > 0$},\]
which holds if and only if
\begin{equation} \label{eq:localize globalize vanishing 3}
\mrm{H}^i(\mc{B}, \Gr^F\mc{M}) = 0 \quad \text{for $i > 0$}
\end{equation}
since $\mf{g}^*$ is an affine variety. Since the filtration $F_\bullet \mc{M}$ is good, it is bounded below, so \eqref{eq:localize globalize vanishing 3} is equivalent to \eqref{eq:localize globalize vanishing 1}.
\end{proof}

The next step is to rewrite the functor \eqref{eq:localize globalize} as convolution with a sheaf of filtered $\tilde{\mc{D}} \boxtimes \tilde{\mc{D}}$-modules on $\mc{B} \times \mc{B}$. Let us write
\[ \mrm{St} = (\tilde{\mu}, \tilde{\mu})^{-1}(0) \subset \tilde{\mf{g}}^* \times \tilde{\mf{g}}^* \]
for the Steinberg variety, where
\[ (\tilde{\mu}, \tilde{\mu}) \colon \tilde{\mf{g}}^* \times \tilde{\mf{g}}^* \to \mf{g}^* \]
is the moment map for the diagonal action of $G$ on $\tilde{\mc{B}}$. We write
\[ \mrm{D}^b_{\mrm{St}}\coh(\tilde{\mc{D}} \boxtimes \tilde{\mc{D}}, F_\bullet) \subset \mrm{D}^b\coh(\tilde{\mc{D}} \boxtimes \tilde{\mc{D}}, F_\bullet)\]
for the full subcategory of complexes whose cohomologies have singular support contained in $\mrm{St}$.

\begin{defn}
For $\mc{K} \in \mrm{D}^b_{\mrm{St}}\coh(\tilde{\mc{D}} \boxtimes \tilde{\mc{D}}, F_\bullet)$ and $\mc{M} \in \mrm{D}^b\coh(\tilde{\mc{D}}, F_\bullet)$, the \emph{convolution of $\mc{K}$ and $\mc{M}$} is
 \[ \mc{K} * \mc{M} = \mrm{R}\mrm{pr}_{1\bigcdot}(\mc{K} \overset{\mrm{L}}\otimes_{\mrm{pr}_2^{-1}\tilde{\mc{D}}} \mrm{pr}_2^{-1}\mc{M}) \in \mrm{D}^b\coh(\tilde{\mc{D}}, F_\bullet),\]
where $\mrm{pr}_i \colon \mc{B} \times \mc{B} \to \mc{B}$ are the natural projections. Here we regard $\mc{K}$ as a right $\mrm{pr}_2^{-1}\tilde{\mc{D}}$-module via the side-changing isomorphism
\[ \tilde{\mc{D}}^{\mathit{op}} \cong \omega_{\mc{B}} \otimes \tilde{\mc{D}} \otimes \omega_{\mc{B}}^{-1} = \pi_{\bigcdot}(\mc{O}_{\tilde{\mc{B}}} \otimes \mb{C}_{2\rho})^H \otimes \pi_{\bigcdot}(\mc{D}_{\tilde{\mc{B}}})^H \otimes \pi_{\bigcdot}(\mc{O}_{\tilde{\mc{B}}} \otimes \mb{C}_{-2\rho})^H \cong \pi_{\bigcdot}(\mc{D}_{\tilde{\mc{B}}})^H = \tilde{\mc{D}},\]
where the second last isomorphism is given by multiplication in $\pi_{\bigcdot}(\mc{D}_{\tilde{\mc{B}}})$.
\end{defn}

Now consider the $\tilde{\mc{D}} \boxtimes \tilde{\mc{D}} = \tilde{\mc{D}} \boxtimes \tilde{\mc{D}}^{\mathit{op}}$-module
\[ \tilde{\Xi} = \tilde{\mc{D}} \otimes_{U(\mf{g})} \tilde{\mc{D}},\]
equipped with the natural filtration $F_\bullet\tilde{\Xi}$ induced by the order filtration on $\tilde{\mc{D}}$. Note that, since the side-changing isomorphism $\tilde{\mc{D}} \cong \tilde{\mc{D}}^{\mathit{op}}$ acts by $-1$ on the associated graded, $\Gr^F\tilde{\Xi} = \mc{O}_{\mrm{St}}$.

\begin{lem} \label{lem:localize globalize convolution}
We have a natural isomorphism
\[ \tilde{\mc{D}} \overset{\mrm{L}}\otimes_{U(\mf{g})} \mrm{R}\Gamma(\mc{M}) \cong \tilde{\Xi} * \mc{M} \]
for $\mc{M} \in \mrm{D}^b\coh(\tilde{\mc{D}}, F_\bullet)$.
\end{lem}
\begin{proof}
We claim that
\begin{equation} \label{eq:localize globalize convolution 1}
\tilde{\Xi} := \tilde{\mc{D}} \otimes_{U(\mf{g})} \tilde{\mc{D}} = \tilde{\mc{D}} \overset{\mrm{L}}\otimes_{U(\mf{g})} \tilde{\mc{D}}
\end{equation}
as objects in the filtered derived category; the desired isomorphism then follows formally. Now,
\begin{equation} \label{eq:localize globalize convolution 2}
\Gr^F(\tilde{\mc{D}} \overset{\mrm{L}}\otimes_{U(\mf{g})} \tilde{\mc{D}}) = \mc{O}_{\tilde{\mf{g}}^*} \overset{\mrm{L}}\otimes_{\mc{O}_{\mf{g}^*}} \mc{O}_{\tilde{\mf{g}}^*} = \mrm{L}(\tilde{\mu}, \tilde{\mu})^{\bigcdot}\mc{O}_0.
\end{equation}
Since $(\tilde{\mu}, \tilde{\mu})$ is a flat map, \eqref{eq:localize globalize convolution 2} has cohomology in degree zero only, from which \eqref{eq:localize globalize convolution 1} follows.
\end{proof}

Lemmas \ref{lem:localize globalize vanishing} and \ref{lem:localize globalize convolution} reduce Theorem \ref{thm:filtered exactness} to the vanishing of the cohomology sheaves $\mc{H}^i\Gr^F(\tilde{\Xi} * \mc{M})$ for $i > 0$. The proof of this vanishing, which takes up the rest of this section, has two main steps.

The first step is to show that the filtered $\tilde{\mc{D}} \boxtimes \tilde{\mc{D}}$-module $\tilde{\Xi}$ is of Hodge-theoretic origin. More precisely, we show (Theorem \ref{thm:xitilde hodge filtration}) that there is a pro-object
\[ \tilde{\Xi}^H \in \pro \mhm(\mc{D}_{\widetilde{0}} \boxtimes \mc{D}_{\widetilde{0}}),\]
satisfying a technical condition (which we call \emph{good}) ensuring that it has a well-defined underlying Hodge-filtered $\tilde{\mc{D}} \boxtimes \tilde{\mc{D}}$-module $(\tilde{\Xi}^H, F_\bullet)$, and that $(\tilde{\Xi}^H, F_\bullet) \cong (\tilde{\Xi}, F_\bullet)$. Note that we already used the trick of associating infinite length filtered $\mc{D}$-modules to certain pro-mixed Hodge modules in \S\ref{subsec:semi-continuity}. The object $\tilde{\Xi}^H$ is a familiar one in representation theory: it is a Hodge lift of the big pro-projective of Soergel theory.

The final step is to apply the monodromic Kodaira vanishing theorem (Theorem \ref{thm:monodromic kodaira}) to deduce the desired cohomology vanishing for $\tilde{\Xi} * \mc{M}$. This is carried out in \S\ref{subsec:pf of thm:filtered exactness}. The general vanishing result for convolutions is Proposition \ref{prop:convolution vanishing}, which holds for pro-objects in $\mhm(\mc{D}_{\widetilde{0}} \boxtimes \mc{D}_{\widetilde{0}})$ satisfying a further condition on stalks (see Definition \ref{defn:S-stalk property}).

\subsection{The big projective} \label{subsec:big projective}

We now turn to the study of the object $\tilde{\Xi}$. We begin with the quotient
\[ \Xi = \tilde{\mc{D}} \otimes_{U(\mf{g})} \mc{D}_0 = \tilde{\Xi} \otimes_{S(\mf{h}), -\rho} \mb{C}.\]
We call $\Xi$ the \emph{big projective}. As for $\tilde{\Xi}$, we endow $\Xi$ with the filtration $F_\bullet \Xi$ induced from the order filtrations on $\tilde{\mc{D}}$ and $\mc{D}_0$. Unlike $\tilde{\Xi}$, the $\tilde{\mc{D}} \boxtimes \mc{D}_0$-module $\Xi$ is of finite length.

Now, recall that the center $Z(U(\mf{g}))$ maps to $\tilde{\mc{D}}$ via
\[ Z(U(\mf{g})) \cong S(\mf{h})^W \xrightarrow{h \mapsto h + \rho(h)} S(\mf{h}) \subset \tilde{\mc{D}}.\]
We deduce that $\mf{h} \subset \tilde{\mc{D}}$ acts on the first factor of $\tilde{\mc{D}}$ in $\Xi$ with generalized eigenvalue $-\rho$. So $\Xi$ lies in the category
\[ \Mod(\mc{D}_{\widetilde{0}} \boxtimes \mc{D}_0, G)_{rh} \subset \Mod_\mon^G(\mc{D}_{\tilde{\mc{B}} \times \tilde{\mc{B}}})_{rh} \]
of regular holonomic (equivalently, coherent) $G$-equivariant $\mc{D}$-modules on $\tilde{\mc{B}} \times \tilde{\mc{B}}$ that are $(-\rho)$-monodromic along the first $\tilde{\mc{B}}$ factor and $(-\rho)$-twisted along the second. The main result of this subsection is:

\begin{thm} \label{thm:xi hodge filtration}
There exists a $G$-equivariant mixed Hodge module
\[ \Xi^H \in \mhm(\mc{D}_{\widetilde{0}} \boxtimes \mc{D}_0, G) \]
and an isomorphism of filtered $\tilde{\mc{D}} \boxtimes \mc{D}_0$-modules $(\Xi^H, F_\bullet\Xi^H) \cong (\Xi, F_\bullet \Xi)$, where $F_\bullet \Xi^H$ is the Hodge filtration and $F_\bullet \Xi$ is the order filtration.
\end{thm}

Before giving the proof, we first remark that the category $\Mod(\mc{D}_{\widetilde{0}}\boxtimes \mc{D}_0, G)_{rh}$ is equivalent to the principal block of category $\mc{O}$. Indeed, if we choose a Borel subgroup $B \subset G$ with unipotent radical $N$, then restriction to the corresponding fiber of the first projection $\tilde{\mc{B}} \times \tilde{\mc{B}} \to \tilde{\mc{B}}$ defines an equivalence
\[ \Mod(\mc{D}_{\widetilde{0}} \boxtimes \mc{D}_0, G)_{rh} \overset{\sim}\to \Mod(\mc{D}_0, N)_{rh}.\]
Tensoring with the $N$-equivariant line bundle $\mc{O}(\rho)$ and applying the Riemann-Hilbert correspondence, we have a further equivalence
\[\Mod(\mc{D}_0, N)_{rh} \cong \Mod(\mc{D}_{\mc{B}}, N)_{rh} \cong \mrm{Perv}_N(\mc{B})\]
with the category of $N$-equivariant perverse sheaves on $\mc{B}$. Similarly, at the level of derived categories, we have an equivalence
\[ \mrm{D}^b_G\Mod(\mc{D}_{\widetilde{0}} \boxtimes \mc{D}_0)_{rh} \cong \mrm{D}^b_{c, N}(\mc{B}),\]
between the $G$-equivariant derived category and the category of objects in the derived category of sheaves on $\mc{B}$ whose cohomology is constructible with respect to the stratification by $N$-orbits. We exploit below the extensive literature on the categories $\mrm{Perv}_N(\mc{B})$ and $\mrm{D}^b_{c, N}(\mc{B})$ (e.g., \cite{BGS, MV}).

The first step in the proof of Theorem \ref{thm:xi hodge filtration} is to identify the $\mc{D}$-module $\Xi$ in categorical terms. Recall that the group $G$ acts on $\mc{B} \times \mc{B}$ with finitely many orbits, indexed by the Weyl group $W$; for $w \in W$, we write
\[ j_w \colon X_w \hookrightarrow \mc{B} \times \mc{B} \]
for the inclusion of the corresponding orbit. We fix the indexing so that $X_1 = \mc{B}$ is the closed orbit and $X_{w_0}$ is the open orbit, where $w_0 \in W$ is the longest element. For each orbit $X_w$, there is a unique rank $1$ $G$-equivariant $(-\rho, -\rho)$-twisted local system $\gamma_w$ on $X_w$ defining standard, costandard and irreducible objects
\[ j_{w!}\gamma_w, \quad j_{w*}\gamma_w \quad \text{and} \quad j_{w!*}\gamma_w \in \Mod(\mc{D}_{\widetilde{0}} \boxtimes \mc{D}_0, G)_{rh}.\]
The objects $j_{w!*}\gamma_w$ exhaust the irreducible objects in this category.

\begin{lem} \label{lem:xi projective}
The module $\Xi$ is a projective cover of the irreducible object $j_{1!*}\gamma_1$.
\end{lem}
\begin{proof}
We have, by definition,
\[ \Xi = (\tilde{\mc{D}} \boxtimes \mc{D}_0) \otimes_{U(\mf{g})} \mb{C}.\]
So for any $\mc{M} \in \Mod(\mc{D}_{\widetilde{0}} \boxtimes \mc{D}_0, G)$,
\[ \hom(\Xi, \mc{M}) = \Gamma(\mc{M})^G.\]
Since the right hand side is exact in $\mc{M}$, we deduce that $\Xi$ is projective. Moreover, $\Gamma(j_{w!*}\gamma_w) = 0$ for $w \neq 1$ and $\Gamma(j_{1!*}\gamma_1)^G \cong \mb{C}$. So $\Xi$ is a projective cover of $j_{1!*}\gamma_1$ as claimed.
\end{proof}

By \cite[Lemma 4.5.3]{BGS}, we deduce the following.

\begin{lem} \label{lem:xi hodge cover}
There exists a mixed Hodge module $\Xi^H \in \mhm(\mc{D}_{\widetilde{0}}\boxtimes \mc{D}_0, G)$ such that
\begin{enumerate}
\item the underlying equivariant $\tilde{\mc{D}}\boxtimes \mc{D}_0$-module is isomorphic to $\Xi$, and
\item the unique non-zero $\tilde{\mc{D}} \boxtimes \mc{D}_0$-module morphism $\Xi \to j_{1*}\gamma_1$ defines a morphism of mixed Hodge modules
\begin{equation} \label{eq:xi hodge cover 1}
\Xi^H \to j_{1*}\gamma_1(-\dim \mc{B}).
\end{equation}
\end{enumerate}
\end{lem}

Here we endow $\gamma_w$ with the usual mixed Hodge module structure of weight $\dim X_w$ with Hodge filtration
\[ F_p \gamma_w = \begin{cases} \gamma_w, & \text{if $p \geq 0$,} \\ 0, & \text{if $p < 0$}. \end{cases}\]
The Tate twist in \eqref{eq:xi hodge cover 1} is arranged so that the Hodge filtrations start in degree $0$. We will show that the conclusion of Theorem \ref{thm:xi hodge filtration} holds for any mixed Hodge module $\Xi^H$ as in Lemma \ref{lem:xi hodge cover}. We fix such a $\Xi^H$ from now on.

\begin{lem} \label{lem:xi stalks}
For any $w \in W$, the stalk $j_w^*\Xi^H$ is given by
\[ j_w^*\Xi^H = \gamma_w(\ell(w) - \dim\mc{B}).\]
\end{lem}
\begin{proof}
By, e.g., \cite{MV}, the stalks of any projective object in $\Mod(\mc{D}_{\widetilde{0}} \boxtimes \mc{D}_0, G)_{rh}$ are (perverse) local systems. By Lemma \ref{lem:xi projective}, this applies to $\Xi^H$. To fix the ranks and the Hodge structures, we note that $j_{1*}\gamma_1(-\ell(w))$ appears as a multiplicity one composition factor in $j_{w*}\gamma_w$. So, as a Hodge structure,
\[ \hom(j_w^*\Xi^H, \gamma_w) = \hom(\Xi^H, j_{w*}\gamma_w) = \mb{C}(-\ell(w) + \dim\mc{B}),\]
which shows that the Hodge structure on the stalk is as claimed.
\end{proof}

Recall now the $\mb{C}$-linear duality functor on mixed Hodge modules. In the monodromic setting, we may regard this as a functor
\[ \mb{D} \colon \mhm(\mc{D}_{\widetilde{0}} \boxtimes \mc{D}_0, G)^{\mathit{op}} \overset{\sim}\to \mhm(\mc{D}_{\widetilde{0}}\boxtimes \mc{D}_0, G),\]
compatible with the duality functor for filtered $\tilde{\mc{D}} \boxtimes \mc{D}_0$-modules given by
\[\mb{D}(\mc{M}, F_\bullet) = \mrm{R}\shom_{(\tilde{\mc{D}} \boxtimes \mc{D}_0, F_\bullet)}((\mc{M}, F_\bullet), (\tilde{\mc{D}} \boxtimes \mc{D}_0, F_{\bullet - 4\dim \mc{B} - \dim H}))[2 \dim \mc{B} + \dim H].\]
Note that the cohomological and filtration shifts are arranged so that
\[\mb{D} \mc{O}(-\rho, -\rho) = \mc{O}(-\rho, -\rho)\{2 \dim \mc{B}\}\]
with respect to the usual (trivial) filtration on $\mc{O}(-\rho, -\rho)$. Here, for $n \in \mb{Z}$, we write
\[(\mc{M}, F_\bullet)\{n\} = (\mc{M}, F_{\bullet - n})\]
for the shift of the filtration $n$ steps. A key step in the proof of Theorem \ref{thm:xi hodge filtration} is to observe that both $(\Xi, F_\bullet)$ and $\Xi^H$ behave well with respect to these dualities.

\begin{lem} \label{lem:xi koszul}
The filtered $\tilde{\mc{D}} \boxtimes \mc{D}_0$-module $(\Xi, F_\bullet)$ is Cohen-Macaulay and satisfies
\[ \mb{D}(\Xi, F_\bullet) \cong (\Xi, F_\bullet)\{2\dim \mc{B}\}.\] 
\end{lem}
\begin{proof}
Recall that Cohen-Macaulayness of $(\Xi, F_\bullet)$ just says that $\mb{D}(\Xi, F_\bullet)$ is concentrated in a single cohomological degree. Consider the filtered Koszul complex
\begin{equation} \label{eq:xi koszul 1}
0 \to (\tilde{\mc{D}} \boxtimes \mc{D}_0) \otimes \wedge^n \mf{g}\{ n\} \to \cdots \to (\tilde{\mc{D}} \boxtimes \mc{D}_0) \otimes \mf{g}\{1\}  \to \tilde{\mc{D}} \boxtimes \mc{D}_0\to (\Xi, F_\bullet) \to 0,
\end{equation}
where $n = \dim \mf{g}$. Passing to associated gradeds, we get the Koszul resolution
\[
0 \to \mc{O}_{\tilde{\mf{g}}^* \times \tilde{\mc{N}}^*} \otimes \wedge^n \mf{g}\{n\} \to \cdots \to \mc{O}_{\tilde{\mf{g}}^* \times \tilde{\mc{N}}^*} \otimes \mf{g}\{ 1\}  \to \mc{O}_{\tilde{\mf{g}}^* \times \tilde{\mc{N}}^*} \to \mc{O}_{\tilde{\mf{g}}^* \times_{\mf{g}^*} \tilde{\mc{N}}^*} \to 0,
\]
which is exact since the map $(\tilde{\mu}, \mu) \colon \tilde{\mf{g}}^* \times \tilde{\mc{N}}^* \to \mf{g}^*$ is flat. Here $\tilde{\mc{N}}^* = T^*\mc{B} \xrightarrow{\mu} \mf{g}^*$ is the Springer resolution. So \eqref{eq:xi koszul 1} is filtered exact, i.e., the Koszul complex is a resolution of $(\Xi, F_\bullet)$. Taking the filtered dual, we find that $\mb{D}(\Xi, F_\bullet)$ is quasi-isomorphic to the complex
\[
\begin{aligned}
(\tilde{\mc{D}} \boxtimes \mc{D}_0) \{n + 2 \dim \mc{B}\} \to & (\tilde{\mc{D}} \boxtimes \mc{D}_0) \otimes \mf{g}^*\{n - 1 + 2 \dim \mc{B}\}  \to \cdots \\ 
&\to (\tilde{\mc{D}} \boxtimes \mc{D}_0) \otimes \wedge^n\mf{g}^* \{2 \dim \mc{B}\}, 
\end{aligned}
\]
where we note that $2 \dim \mc{B} + \dim H = n$. But this is precisely the Koszul resolution for the module
\[ (\Xi, F_\bullet) \otimes \wedge^n\mf{g}^*\{2 \dim \mc{B}\} \cong (\Xi, F_\bullet)\{2 \dim \mc{B}\}\]
so this proves the lemma.
\end{proof}

\begin{lem} \label{lem:xi hodge dual}
The mixed Hodge module dual $\mb{D}\Xi^H$ admits a non-zero map
\[ \mb{D} \Xi^H \to j_{1*}\gamma_1(\dim \mc{B}).\]
Hence, $\mb{D}\Xi^H (-2 \dim \mc{B})$ is another mixed Hodge module satisfying the conditions of Lemma \ref{lem:xi hodge cover}.
\end{lem}
\begin{proof}
By Lemma \ref{lem:xi stalks}, we have an injective map
\[ j_{w_0!}\gamma_{w_0} = j_{w_0!}j_{w_0}^* \Xi^H \to \Xi^H,\]
where $w_0 \in W$ is the longest element (so $j_{w_0}$ is the inclusion of the open orbit). But now we have an inclusion
\[ j_{1*}\gamma_1 \hookrightarrow j_{w_0!}\gamma_{w_0} \]
as the lowest piece of the weight filtration. So
\[ j_{1*}\gamma_1 \hookrightarrow \Xi^H.\]
Taking duals, we deduce the lemma.
\end{proof}

The final ingredient in the proof of Theorem \ref{thm:xi hodge filtration} is to give some control over the lowest piece of the Hodge filtration in $\Xi^H$. To this end, we make the following general observation.

\begin{lem} \label{lem:lowest hodge vanishing}
Let $X$ be a smooth projective variety, $H$ a torus and $\tilde{X} \to X$ an $H$-torsor. Suppose that $\lambda \in \mf{h}^*_\mb{R}$ is ample, let $\mc{M} \in \mhm_{\widetilde{\lambda}}(\tilde{X})$ be a monodromic mixed Hodge module and let $c \in \mb{Z}$ be the lowest index such that $F_c\mc{M} \neq 0$. Then
\[ \mrm{H}^i(X, F_c\mc{M} \otimes \omega_X) = 0 \quad \text{for $i > 0$}.\]
\end{lem}
\begin{proof}
Since $\mc{M}$ is an iterated extension of twisted Hodge modules, we may assume without loss of generality that $\mc{M} \in \mhm_\lambda(\tilde X)$. Applying Theorem \ref{thm:twisted kodaira} to the morphism from $X$ to a point, we have
\[ \mb{H}^i(X, [F_{p - n}\mc{M} \to F_{p - n + 1}\mc{M} \otimes \Omega^1_X \to \cdots \to F_p\mc{M} \otimes \Omega^n_X]) = 0 \quad \text{for $i > 0$},\]
where $n = \dim X$ and the complex is placed in degrees $-n, \ldots, 0$. Setting $p = c$, this gives
\[ \mrm{H}^i(X, F_c\mc{M} \otimes \Omega^n_X) = \mrm{H}^i(X, F_c\mc{M} \otimes \omega_X) = 0 \quad \text{for $i > 0$}\]
as claimed.
\end{proof}

Now, for
\[ \mc{M} \in \mhm(\mc{D}_{\widetilde{0}} \boxtimes \mc{D}_0, G) \subset \mhm_{\widetilde{-\rho}, -\rho}(\tilde{\mc{B}} \times \tilde{\mc{B}}),\]
 we have
\[ \mc{M} \otimes \omega_{\mc{B} \times \mc{B}}^{-1} = \mc{M} \otimes \mc{O}(2\rho, 2\rho) \in \mhm_{\widetilde{\rho}, \rho}(\tilde{\mc{B}} \times \tilde{\mc{B}}).\]
So Lemma \ref{lem:lowest hodge vanishing} gives
\begin{equation} \label{eq:lowest hodge vanishing}
 \mrm{H}^i(\mc{B} \times \mc{B}, F_c\mc{M}) = 0 \quad \text{for $i > 0$}
\end{equation}
where $F_c\mc{M}$ is the lowest piece of the Hodge filtration.

\begin{lem} \label{lem:xi lowest piece}
The unique $G$-invariant vector in $\Gamma(F_0j_{1*}\gamma_1(-\dim \mc{B}))$ lifts to a $G$-invariant vector in $\Gamma(F_0\Xi^H)$.
\end{lem}
\begin{proof}
We need to show that the map
\[ \Gamma(F_0\Xi^H) \to \Gamma(F_0 j_{1*}\gamma_1(-\dim \mc{B})) \]
is surjective. By \eqref{eq:lowest hodge vanishing}, it is enough to show that $F_p \Xi^H = 0$ for $p < 0$. Lemma \ref{lem:xi stalks} implies that $\Xi^H$ has a filtration by the objects $j_!\gamma_w(\ell(w) - \dim \mc{B})$. But now
\[ F_p j_!\gamma_w = 0 \quad \text{for $p < \codim X_w$},\]
and hence
\[ F_p j_!\gamma_w(\ell(w) - \dim \mc{B}) = 0 \quad \text{for $p < 0$}\]
so we are done.
\end{proof}

\begin{proof}[Proof of Theorem \ref{thm:xi hodge filtration}]
Given $\Xi^H$ as above, by Lemma \ref{lem:xi lowest piece}, there exists a map
\begin{equation} \label{eq:xi hodge filtration 1}
\Xi = (\tilde{\mc{D}} \boxtimes \mc{D}_0)\otimes_{U(\mf{g})}\mb{C} \to \Xi^H
\end{equation}
respecting the map to $j_{1*}\gamma_1$ and sending the generator in
\[\mb{C} \subset F_0\Xi \subset \Xi = (\tilde{\mc{D}} \boxtimes \mc{D}_0)\otimes_{U(\mf{g})} \mb{C} \]
to a $G$-invariant vector in $\Gamma(F_0\Xi^H)$. By uniqueness of projective covers, any such map must be an isomorphism of $\mc{D}$-modules. Since the filtration $F_\bullet \Xi$ is generated by $F_0\Xi$, we deduce that \eqref{eq:xi hodge filtration 1} is also a filtered map, i.e., $F_p \Xi \subset F_p\Xi^H$. Passing to duals and applying Lemma \ref{lem:xi koszul}, we get a map
\[ (\mb{D}\Xi^H, F_\bullet) = \mb{D}(\Xi^H, F_\bullet) \to \mb{D}(\Xi, F_\bullet) \cong (\Xi, F_\bullet)\{2 \dim \mc{B}\},\]
an isomorphism on the underlying $\mc{D}$-modules. By Lemma \ref{lem:xi hodge dual}, $\mb{D}\Xi^H(-2\dim \mc{B})$ is another Hodge module satisfying the conditions of Lemma \ref{lem:xi hodge cover}, so we conclude that we have a map in the other direction,
\[ (\Xi, F_\bullet)\{2 \dim \mc{B}\} \to (\mb{D}\Xi^H, F_\bullet),\]
which is also an isomorphism on the underlying $\mc{D}$-modules. The composition
\[ (\Xi, F_\bullet)\{2\dim \mc{B}\} \to (\mb{D}\Xi^H, F_\bullet) \to (\Xi, F_\bullet)\{2\dim \mc{B}\} \]
must be the identity (up to scale), and hence our original map \eqref{eq:xi hodge filtration 1} must be a filtered isomorphism. This concludes the proof.
\end{proof}

\subsection{The $S$-stalk property} \label{subsec:S-stalk}

We next note the following special property of $\Xi$, which we will need in the proof of Theorem \ref{thm:filtered exactness}. Fix any subset $S$ of the simple roots for $G$. Then there is an associated conjugacy class of parabolic subgroups of $G$, parametrized by a partial flag variety $\mc{P}_S$ with projection map $\pi_S \colon \mc{B} \to \mc{P}_S$. We fix the indexing so that, for any choice of Borel $B \subset G$, with roots $\Phi_-$, the corresponding parabolic subgroup $P$ containing $B$ has roots $\Phi_- + \Phi_S$, where $\Phi_S \subset \Phi$ is the set of roots spanned by $S$. For $w \in W$, let us write $i_{w, S}$ and $j_{w, S}$ for the inclusions
\[ X_w \xrightarrow{i_{w, S}} (\id, \pi_S)^{-1}(\id, \pi_S)(X_w) \xrightarrow{j_{w, S}} \mc{B} \times \mc{B},\]
where
\[ (\id, \pi_S) \colon \mc{B} \times \mc{B} \to \mc{B} \times \mc{P}_S\]
is the obvious map.

\begin{defn} \label{defn:S-stalk property}
Let $\mc{K} \in \Mod(\mc{D}_{\widetilde{0}} \boxtimes \mc{D}_{\widetilde{0}}, G)_{rh}$. We say that $\mc{K}$ has the \emph{$S$-stalk property} if
\[ \mc{H}^i(i_{w, S}^!j_{w, S}^*\mc{K}) = 0 \quad \text{for $i > 0$ and all $w \in W$}.\]
\end{defn}

Note that the $S$-stalk property is vacuous if $S = \emptyset$, while if $S$ is the set of all simple roots it asserts that $\mc{K}$ has a costandard filtration.

\begin{prop} \label{prop:xi mixed stalks}
The big projective $\Xi$ has the $S$-stalk property for any $S$.
\end{prop}
\begin{proof}
Observe that we have an equivalence of categories
\begin{equation} \label{eq:xi mixed stalks 1}
\begin{aligned}
\Mod(\mc{D}_{(\id, \pi_S)^{-1}(\id, \pi_S)(X_w), \widetilde{0}, 0}, G)_{rh} &\cong \Mod(\mc{D}_{\mc{B}_L, \widetilde{-\rho_P}} \boxtimes \mc{D}_{\mc{B}_L,-\rho_P}, L)_{rh} \\
&\cong \Mod(\mc{D}_{\mc{B}_L, \widetilde{0}} \boxtimes \mc{D}_{\mc{B}_L,0}, L)_{rh}
\end{aligned}
\end{equation}
given by restricting to a fiber $\tilde{\mc{B}}_L \times \tilde{\mc{B}}_L$ of $\tilde{\mc{B}} \times \tilde{\mc{B}} \to \mc{P}_S \times \mc{P}_S$ and tensoring with a rank $1$ local system; here $L \subset G$ is a Levi subgroup of a parabolic $P$ in the conjugacy class determined by $S$, $\mc{B}_L$ is its flag variety and $-\rho_P = \rho_L - \rho \in \mb{X}^*(L) \otimes \mb{R}$ is half the sum of the roots in the unipotent radical of $P$. Now, by \cite[Corollary 3.3.2]{BGS} and Lemma \ref{lem:xi projective}, we have
\[ \mrm{Ext}^i(\Xi, j_{w!*}\gamma_w) = \begin{cases} \mb{C}, & \text{if $w = 1$ and $i = 0$}, \\ 0, &\text{otherwise},\end{cases}\]
where the Exts are taken in the equivariant derived category $\mrm{D}^b_G\Mod(\mc{D}_{\widetilde{0}} \boxtimes \mc{D}_0)_{rh}$. Hence, the $\mc{D}$-module
\[ j_{w, S}^*\Xi \in \Mod(\mc{D}_{(\id, \pi_S)^{-1}(\id, \pi_S)(X_w), \widetilde{0}, 0}, G)_{rh}\]
must be a projective object satisfying
\[ \hom(j_{w, S}^*\Xi, i_{w', S*}\gamma_{w'}) = \mb{C} \]
for all orbits $X_{w'} \subset (\id, \pi_S)^{-1}(\id, \pi_S)(X_w)$. Under the equivalence \eqref{eq:xi mixed stalks 1}, one easily sees that the only such projective object is the big projective
\[ \Xi_L \in \Mod(\mc{D}_{\mc{B}_L, \widetilde{0}} \boxtimes \mc{D}_{\mc{B}_L, 0}, L)_{rh} \]
for $L$. So, as $\mc{D}$-modules
\[ i_{w, S}^!j_{w, S}^*\Xi = i_{w, S}^!\Xi_L = \mb{D} i_{w, S}^*\Xi_L \cong \gamma_w\]
by Lemmas \ref{lem:xi koszul} and \ref{lem:xi stalks} applied to $L$, so we are done.
\end{proof}

\subsection{The big pro-projective} \label{subsec:monodromic big projective}

In this subsection, we bootstrap Theorem \ref{thm:xi hodge filtration} to obtain a similar result about $\tilde{\Xi}$. Unlike $\Xi$, the $\tilde{\mc{D}} \boxtimes\tilde{\mc{D}}$-module $\tilde{\Xi}$ is \emph{not} of finite length, hence it cannot arise as the filtered module underlying any mixed Hodge module. To obtain an analog of Theorem \ref{thm:xi hodge filtration} for $\tilde{\Xi}$, we need to incorporate certain pro-objects into the theory.

\begin{defn}
Fix $\lambda, \mu \in \mf{h}^*_\mb{R}$ and let
\[ \mc{K} = \varprojlim_n \mc{K}_n \in \pro \mhm(\mc{D}_{\widetilde{\mu}} \boxtimes \mc{D}_{\widetilde{-\lambda}}, G).\]
We say that $\mc{K}$ is \emph{good} if the pro-objects
\[ \mrm{Tor}_i^{S(\mf{h})}(\mc{K} , \mb{C}_{-\lambda - \rho}) = \varprojlim_n \mrm{Tor}_i^{S(\mf{h})}(\mc{K}_n, \mb{C}_{-\lambda - \rho}) \in \pro \Mod(\mc{D}_{\widetilde{\mu}} \boxtimes \mc{D}_{-\lambda})\]
are constant for all $i$. Here $S(\mf{h})$ acts on $\mc{K}$ via the inclusion of $S(\mf{h})$ into the second factor of $\tilde{\mc{D}}$.
\end{defn}

Note that, since the forgetful functor from mixed Hodge modules to $\mc{D}$-modules is faithful and exact, if $\mc{K}$ is good then the pro-mixed Hodge modules
\[ \mrm{Tor}_i^{S(\mf{h}(1))}(\mc{K}, \mb{C}) \in \pro \mhm(\mc{D}_{\widetilde{\mu}} \boxtimes \mc{D}_{-\lambda}, G) \]
are also constant; here we recall from Proposition \ref{prop:monodromic mhm} that the nilpotent part of the $\mf{h}$-action on $\mc{K}$ lifts to a mixed Hodge module action of $\mf{h}(1)$.

\begin{prop} \label{prop:good pro-objects}
Let $\mc{K} \in \pro \mhm(\mc{D}_{\widetilde{\mu}} \boxtimes \mc{D}_{\widetilde{-\lambda}}, G)$ be a good pro-object. Then:
\begin{enumerate}
\item \label{itm:good pro-object 1} The pro-object $F_p \mc{K}$ is constant for all $p$. Hence, we have an underlying filtered $\tilde{\mc{D}} \boxtimes \tilde{\mc{D}}$-module
\[ (\mc{K}, F_\bullet) := \left(\bigcup_p F_p\mc{K}, F_\bullet \mc{K}\right).\]
\item \label{itm:good pro-object 2} If $(\mc{K}', F_\bullet) \in \coh(\tilde{\mc{D}}\boxtimes \tilde{\mc{D}}, F_\bullet)$ and $(\mc{K}', F_\bullet) \to (\mc{K}, F_\bullet)$ is a morphism such that
\begin{equation} \label{eq:good pro-object 3}
\mc{K}' \otimes_{S(\mf{h})} \mb{C}_{-\lambda - \rho} \to \mc{K} \otimes_{S(\mf{h})} \mb{C}_{-\lambda - \rho}
\end{equation}
is filtered surjective, then $\mc{K}' \to \mc{K}$ is filtered surjective. Hence, $(\mc{K}, F_\bullet) \in \coh(\tilde{\mc{D}} \boxtimes \tilde{\mc{D}}, F_\bullet)$.
\item \label{itm:good pro-object 3} If, moreover, \eqref{eq:good pro-object 3}
is a filtered quasi-isomorphism, then $\mc{K}' \to \mc{K}$ is a filtered isomorphism.
\end{enumerate}
\end{prop}
\begin{proof}
Let $I = \ker(S(\mf{h}(1)) \to \mb{C})$ and set $S(\mf{h}(1))_n = S(\mf{h}(1))/I^n$. Since the $S(\mf{h}(1))$-module $S(\mf{h}(1))_n$ is an iterated extension of Tate twists of $\mb{C}$ and $\mc{K}$ is good, it follows that the pro-objects
\[ \mrm{Tor}_i^{S(\mf{h}(1))}(\mc{K}, S(\mf{h}(1))_n) \]
are constant for all $n$. Since $\mf{h}(1)$ acts nilpotently on every object in
\[ \mhm(\mc{D}_{\widetilde{\mu}} \boxtimes \mc{D}_{\widetilde{-\lambda}}, G),\]
it follows that the canonical map
\[ \mc{K} \to \varprojlim_n \mc{K} \otimes_{S(\mf{h}(1))} S(\mf{h}(1))_n \]
is an isomorphism. Now, for all $n$, we have a long exact sequence
\begin{equation} \label{eq:good pro-object 2}
\begin{aligned}
\cdots \to \mrm{Tor}_1^{S(\mf{h}(1))}&(\mc{K}, S(\mf{h}(1))_n) \to \mc{K} \otimes_{S(\mf{h}(1))} S^n(\mf{h}(1)) \\
&\to \mc{K} \otimes_{S(\mf{h}(1))} S(\mf{h}(1))_{n + 1} \to \mc{K} \otimes_{S(\mf{h}(1))} S(\mf{h}(1))_n\to 0.
\end{aligned}
\end{equation}
Since
\[ F_p(\mc{K} \otimes_{S(\mf{h}(1))} S^n(\mf{h}(1))) = F_{p - n}(\mc{K} \otimes_{S(\mf{h}(1))} \mb{C}) \otimes S^n(\mf{h}) = 0 \quad \text{for $n \gg 0$},\]
we deduce that
\[ F_p \mc{K} = F_p (\mc{K} \otimes_{S(\mf{h}(1))} S(\mf{h}(1))_n) \quad \text{for $n \gg 0$},\]
and hence conclude \eqref{itm:good pro-object 1}. Similarly, to prove \eqref{itm:good pro-object 2} (resp., \eqref{itm:good pro-object 3}), note that, using \eqref{eq:good pro-object 2} and induction on $n$, we have that
\[ \mc{K}' \otimes_{S(\mf{h}(1))} S(\mf{h}(1))_n \to \mc{K} \otimes_{S(\mf{h}(1))} S(\mf{h}(1))_n \]
is filtered surjective (resp., a filtered isomorphism) for all $n$, and hence so is $\mc{K}' \to \mc{K}$. Coherence of $(\mc{K}, F_\bullet)$ follows since, for $p$ large enough, the morphism
\[ (\tilde{\mc{D}} \boxtimes \mc{D}_0) \otimes F_p \mc{K} \to (\tilde{\mc{D}} \boxtimes \mc{D}_0) \otimes F_p(\mc{K} \otimes_{S(\mf{h})} \mb{C}) \to \mc{K} \otimes_{S(\mf{h})}\mb{C} \]
is filtered surjective, and hence so is
\[ (\tilde{\mc{D}} \boxtimes \tilde{\mc{D}}) \otimes F_p \mc{K} \to (\mc{K}, F_\bullet),\]
by \eqref{itm:good pro-object 2}. 
\end{proof}

Consider once more the filtered $\tilde{\mc{D}}\boxtimes\tilde{\mc{D}}$-module $\tilde{\Xi}$. We have the following analog of Theorem \ref{thm:xi hodge filtration}.

\begin{thm} \label{thm:xitilde hodge filtration}
There exists a good pro-object
\[ \tilde{\Xi}^H \in \pro \mhm(\mc{D}_{\widetilde{0}} \boxtimes \mc{D}_{\widetilde{0}}, G)\]
equipped with an isomorphism $(\tilde{\Xi}^H, F_\bullet) \cong (\tilde{\Xi}, F_\bullet)$.
\end{thm}
\begin{proof}
Consider the $\rho$-shifted action of $S(\mf{h})$ on $\tilde{\Xi}$ so that
\[ \Xi = \tilde{\Xi} \otimes_{S(\mf{h})} \mb{C}\]
where $\mf{h}$ acts on $\mb{C}$ trivially. Define
\[ \tilde{\Xi}^\wedge = \varprojlim_n \tilde{\Xi} \otimes_{S(\mf{h})} S(\mf{h})_n \in \pro \Mod(\mc{D}_{\widetilde{0}} \boxtimes \mc{D}_{\widetilde{0}}, G)_{rh},\]
where $S(\mf{h})_n = S(\mf{h})/I^n$ as in the proof of Proposition \ref{prop:good pro-objects}. Now, $\tilde{\Xi}$ is flat over $S(\mf{h})$, so $\tilde{\Xi}^\wedge \overset{\mrm{L}}\otimes_{S(\mf{h})} \mb{C} = \Xi$. In particular
\[ \mrm{Ext}^i_{\pro \Mod(\mc{D}_{\widetilde{0}} \boxtimes \mc{D}_{\widetilde{0}}, G)_{rh}}(\tilde{\Xi}^{\wedge},  j_{w!*}\gamma_w) = \mrm{Ext}^i_{\Mod(\mc{D}_{\tilde{0}} \boxtimes \mc{D}_{0}, G)_{rh}}(\Xi, j_{w!*}\gamma_w) \]
for all $w$ and all $i$. So $\tilde{\Xi}^\wedge$ is a pro-projective cover of $j_{1!*}\gamma_1$ by Lemma \ref{lem:xi projective}. Hence, by the argument of \cite[Lemma 4.5.3]{BGS}, there exists a pro-object
\[ \tilde{\Xi}^H \in \pro \mhm(\mc{D}_{\widetilde{0}} \boxtimes \mc{D}_{\widetilde{0}}, G)\]
equipped with a map $\tilde{\Xi}^H \to j_{1!*}\gamma_1(-\dim \mc{B})$ such that $\tilde{\Xi}^H \cong \tilde{\Xi}^{\wedge}$ as pro-$\mc{D}$-modules. We deduce that $\tilde{\Xi}^H$ is good and that
\[ \Xi^H = \tilde{\Xi}^H \otimes_{S(\mf{h}(1))} \mb{C} \]
is a mixed Hodge module satisfying the conditions of Lemma \ref{lem:xi hodge cover}.

Now, since $F_0\mf{h}(1) = 0$ and $F_{-1}\Xi^H = 0$, we have
\[ F_0\tilde{\Xi}^H = F_0 (\tilde{\Xi}^H \otimes_{S(\mf{h}(1))} \mb{C}) = F_0\Xi^H.\]
So as in the proof of Theorem \ref{thm:xi hodge filtration}, we obtain a filtered morphism
\begin{equation} \label{eq:xitilde hodge filtration 1}
(\tilde{\Xi}, F_\bullet) \to (\tilde{\Xi}^H, F_\bullet) 
\end{equation}
sending the generator of $F_0\tilde{\Xi}$ to a $G$-invariant section of $F_0\tilde{\Xi}^H = F_0\Xi^H$. Since, by Theorem \ref{thm:xi hodge filtration}, \eqref{eq:xitilde hodge filtration 1} becomes a filtered isomorphism after tensoring over $S(\mf{h}(1))$ with $\mb{C}$, we conclude that \eqref{eq:xitilde hodge filtration 1} is itself a filtered isomorphism by Proposition \ref{prop:good pro-objects}.
\end{proof}

The object $\tilde{\Xi}^H$ inherits a version of the $S$-stalk property from $\Xi^H$. For a pro-object $\mc{K} \in \pro\mhm(\mc{D}_{\widetilde{\mu}} \boxtimes \mc{D}_{\widetilde{-\lambda}}, G)$ and a subset $S$ of the simple roots for $G$, let us say that $\mc{K}$ has the \emph{$S$-stalk property} if we can write
\[ \mc{K} = \varprojlim_n \mc{K}_n , \quad \mc{K}_n \in \mhm(\mc{D}_{\widetilde{\mu}} \boxtimes \mc{D}_{\widetilde{-\lambda}}, G),\]
so that each $\mc{K}_n$ has the $S$-stalk property.

\begin{prop} \label{prop:xitilde mixed stalks}
The pro-object $\tilde{\Xi}^H$ has the $S$-stalk property for any subset $S$ of the simple roots.
\end{prop}
\begin{proof}
We have
\[ \tilde{\Xi}^H = \varprojlim_n \tilde{\Xi}^H \otimes_{S(\mf{h}(1))} S(\mf{h}(1))_n, \]
where each term on the right is an iterated extension of Tate twists of $\Xi^H$, and hence is a mixed Hodge module with the $S$-stalk property by Proposition \ref{prop:xi mixed stalks}.
\end{proof}

\begin{rmk} \label{rmk:standards}
Theorem \ref{thm:xitilde hodge filtration} implies a theorem of \cite{BIR} that the Hodge associated gradeds of the ``free-monodromic costandard modules'' $j_{w*}j_w^*\tilde{\Xi}^H$ are given by
\[ \Gr^Fj_{w*}j_w^*\tilde{\Xi}^H = \mc{O}_{\mrm{St}_w},\]
where $\mrm{St}_w \subset \mrm{St} := \tilde{\mf{g}}^* \times_{\mf{g}^*} \tilde{\mf{g}}^*$ is the irreducible component labelled by $w$. Indeed, we have $\Gr^F\tilde{\Xi}^H = \mc{O}_{\mrm{St}}$ by Theorem \ref{thm:xitilde hodge filtration} and hence a surjection $\mc{O}_{\mrm{St}} \to \Gr^F j_{w*}j_w^*\tilde{\Xi}^H$. If we choose a homogeneous element $h \in S(\mf{h} \oplus \mf{h}) = \mb{C}[\mf{h}^* \times \mf{h}^*]$ such that $h$ vanishes identically on the graph of $w'$ for all $w' \neq w$ but not on the graph of $w$, then it is easy to check using strictness for morphisms of mixed Hodge modules that $\Gr^F\tilde{\Xi}^H[h^{-1}] \to \Gr^Fj_{w*}j_w^*\tilde{\Xi}^H[h^{-1}]$ is an isomorphism and that $\Gr^F j_{w*}j_w^*\tilde{\Xi}^H \to \Gr^Fj_{w*}j_w^*\tilde{\Xi}^H[h^{-1}]$ is injective. Hence,
\[ \Gr^Fj_{w*}j_w^*\tilde{\Xi}^H = \mrm{im}(\mc{O}_{\mrm{St}} \to \mc{O}_{\mrm{St}}[h^{-1}]) = \mc{O}_{\mrm{St}_w}.\]
\end{rmk}

\subsection{Proof of Theorem \ref{thm:filtered exactness}} \label{subsec:pf of thm:filtered exactness}

We now give the proof of Theorem \ref{thm:filtered exactness}. By Lemmas \ref{lem:localize globalize vanishing} and \ref{lem:localize globalize convolution}, Theorem \ref{thm:xitilde hodge filtration}, and Proposition \ref{prop:xitilde mixed stalks}, the theorem follows from:

\begin{prop} \label{prop:convolution vanishing}
Assume $\mc{K} \in \pro \mhm(\mc{D}_{\widetilde{\mu}} \boxtimes \mc{D}_{\widetilde{-\lambda'}}, G)$ is a good pro-object and $\mc{M} \in \mhm(\mc{D}_{\widetilde{\lambda}})$. Assume $\lambda - \lambda' \in \mf{h}^*_\mb{R}$ is dominant and let $S$ denote the set of singular simple roots for $\lambda - \lambda'$. If $\mc{K}$ has the $S$-stalk property, then
\[ \mc{H}^i \Gr^F(\mc{K} * \mc{M}) = 0 \quad \text{for $i > 0$}.\]
\end{prop}
\begin{proof}
The idea behind the proof is to rewrite the filtered convolution
\[ \mc{K} * \mc{M} = \mrm{R}\mrm{pr}_{1\bigcdot}(\mc{K} \overset{\mrm{L}}\otimes_{\mrm{pr}_2^{-1}\tilde{\mc{D}}} \mrm{pr}_2^{-1}\mc{M}) \]
in terms of mixed Hodge module operations and apply monodromic Kodaira vanishing.

First, observe that if $p_0$ is the minimal integer such that $F_{p_0}\mc{M} \neq 0$, then the complex $\Gr^F_p (\mc{K} * \mc{M})$ depends only on $F_q \mc{K}$ for $q \leq p - p_0$. So if we write $\mc{K} = \varprojlim_n \mc{K}_n$, where $\mc{K}_n \in \mhm(\mc{D}_{\widetilde{\mu}} \boxtimes \mc{D}_{\widetilde{-\lambda'}}, G)$ has the $S$-stalk property, then
\[ \mc{H}^i \Gr^F(\mc{K} * \mc{M}) = \varprojlim_n \mc{H}^i \Gr^F(\mc{K}_n * \mc{M}).\]
So we may assume without loss of generality that $\mc{K} \in \mhm(\mc{D}_{\widetilde{\mu}} \boxtimes \mc{D}_{\widetilde{-\lambda'}}, G)$.

Unpacking the definition of the side-changing isomorphism $\tilde{\mc{D}}^{\mathit{op}} \cong \tilde{\mc{D}}$, we have
\begin{align*}
 \mc{K} \overset{\mrm{L}}\otimes_{\mrm{pr}_2^{-1}\tilde{\mc{D}}} \mrm{pr}_2^{-1}\mc{M} &=  \mrm{pr}_2^{-1}\omega_{\mc{B}} \overset{\mrm{L}}\otimes_{\mrm{pr}_2^{-1}\tilde{\mc{D}}}(\mc{K} \overset{\mrm{L}}\otimes_{\mrm{pr}_2^{-1}\mc{O}}\mrm{pr}_2^{-1}(\mc{M} \otimes_\mc{O} \mc{O}(2\rho))) \\
 &= \mrm{pr}_2^{-1}\omega_{\mc{B}} \overset{\mrm{L}}\otimes_{\mrm{pr}_2^{-1}\tilde{\mc{D}}} \mrm{L}\Delta_{23}^{\bigcdot}(\mc{K} \boxtimes (\mc{M} \otimes \mc{O}(2\rho))),
\end{align*}
where
\[ \Delta_{23} \colon \mc{B} \times \mc{B} \to \mc{B} \times \mc{B} \times \mc{B} \]
is given by $(x, y) \mapsto (x, y, y)$. So, in the notation of \S\ref{sec:twisted kodaira}, we have
\begin{equation} \label{eq:convolution formula 1}
\mc{K} * \mc{M} = \widetilde{\mrm{pr}}_{1\dagger} \mrm{L} \Delta_{23}^{\bigcdot}(\mc{K} \boxtimes (\mc{M} \otimes \mc{O}(2\rho)))\{\dim \tilde{\mc{B}}\},
\end{equation}
where
\[ \widetilde{\mrm{pr}}_1 \colon \tilde{\mc{B}} \times \tilde{\mc{B}} \to \tilde{\mc{B}} \]
is the first projection.

Consider now the mixed Hodge modules $\mc{K}_{\tilde{\mc{B}} \times \tilde{\mc{B}}}$ and $\mc{M}_{\tilde{\mc{B}}}$ on (products of) the base affine space corresponding to $\mc{K}$ and $\mc{M}$. A priori, we have
\[ \mc{K}_{\tilde{\mc{B}} \times \tilde{\mc{B}}} \in \mhm_{\widetilde{\mu - \rho}, \widetilde{-\lambda - \rho}}(\tilde{\mc{B}} \times \tilde{\mc{B}}) \quad \text{and} \quad \mc{M}_{\tilde{\mc{B}}} \in \mhm_{\widetilde{\lambda - \rho}}(\tilde{\mc{B}}).\]
Since $2\rho \in \mb{X}^*(H)$, however, we have $\mhm_{\widetilde{\lambda - \rho}}(\tilde{\mc{B}}) = \mhm_{\widetilde{\lambda + \rho}}(\tilde{\mc{B}})$ as subcategories of $\mhm(\tilde{\mc{B}})$; let us write $\mc{M}_{\tilde{\mc{B}}} \otimes \mc{O}(2\rho)$ for $\mc{M}_{\tilde{\mc{B}}}$ regarded as an object in $\mhm_{\widetilde{\lambda + \rho}}(\tilde{\mc{B}})$. We then have
\[ u(\mc{K}_{\tilde{\mc{B}} \times \tilde{\mc{B}}}) = \mc{K} \quad \text{and} \quad u(\mc{M}_{\tilde{\mc{B}}} \otimes \mc{O}(2\rho)) = \mc{M} \otimes \mc{O}(2\rho),\]
where we write $u$ for operation of passing to the underlying filtered $\tilde{\mc{D}}$-module. Now, since $\mc{K}$ is $G$-equivariant, it follows that the morphism
\[ \tilde{\Delta}_{23} \colon \tilde{\mc{B}} \times \tilde{\mc{B}} \to \tilde{\mc{B}} \times \tilde{\mc{B}} \times \tilde{\mc{B}}\]
is non-characteristic for $\mc{K}_{\tilde{\mc{B}} \times \tilde{\mc{B}}} \boxtimes (\mc{M}_{\tilde{\mc{B}}} \otimes \mc{O}(2\rho))$. In particular, we have an isomorphism in the filtered derived category
\[ u(\tilde{\Delta}_{23}^\circ(\mc{K}_{\tilde{\mc{B}} \times \tilde{\mc{B}}} \boxtimes (\mc{M}_{\tilde{\mc{B}}} \otimes\mc{O}(2\rho)))) = \mrm{L}\Delta_{23}^{\bigcdot}(\mc{K} \boxtimes (\mc{M} \otimes \mc{O}(2\rho))),\]
where
\begin{align*}
\tilde{\Delta}_{23}^\circ(\mc{K}_{\tilde{\mc{B}} \times \tilde{\mc{B}}} \boxtimes (\mc{M}_{\tilde{\mc{B}}} \otimes \mc{O}(2\rho))) := \tilde{\Delta}_{23}^*(\mc{K}_{\tilde{\mc{B}} \times \tilde{\mc{B}}} \boxtimes (\mc{M}_{\tilde{\mc{B}}} \otimes \mc{O}(&2\rho)))[-\dim \tilde{\mc{B}}] \\
&\in \mhm_{\widetilde{\mu - \rho}, \widetilde{\lambda - \lambda'}}(\tilde{\mc{B}} \times \tilde{\mc{B}}).
\end{align*}
So from \eqref{eq:convolution formula 1}, we deduce
\begin{equation} \label{eq:convolution formula 2}
\mc{K} * \mc{M} = \widetilde{\mrm{pr}}_{1\dagger} u(\tilde{\Delta}_{23}^\circ(\mc{K}_{\tilde{\mc{B}} \times \tilde{\mc{B}}} \boxtimes (\mc{M}_{\tilde{\mc{B}}} \otimes \mc{O}(2\rho)))(\dim \tilde{\mc{B}})).
\end{equation}

Now, if $\lambda - \lambda'$ is regular dominant, then $(\mu - \rho, \lambda - \lambda')$ is relatively ample for the first projection $\mrm{pr}_1$, so we can conclude by applying Theorem \ref{thm:monodromic kodaira}. In general, an extra step is required.

Let us write $H_S$ for the quotient torus of $H$ with character group
\[ \mb{X}^*(H_S) = \{ \nu \in \mb{X}^*(H) \mid \langle \nu, \check\alpha \rangle = 0 \text{ for } \alpha \in S\}.\]
Note that the kernel of $H \to H_S$ is connected, so we are in the setting of \S\ref{subsec:monodromic kodaira}. We have
\[ \mb{X}^*(H_S) = \mrm{Pic}^G(\mc{P}_S),\]
where $\mc{P}_S$ is the partial flag variety corresponding to the set of simple roots $S$, so there is a tautological $H_S$-torsor $\tilde{\mc{P}}_S$ over $\mc{P}_S$, equipped with an $H$-equivariant morphism $\tilde{\mc{B}} \to \tilde{\mc{P}}_S$. Consider the diagram
\[
\begin{tikzcd}
\tilde{\mc{B}} \times \tilde{\mc{B}} \ar[r, "\tilde{p}"] \ar[d] & \tilde{\mc{B}} \times \tilde{\mc{P}}_S \ar[r, "\tilde{q}"] \ar[d] & \tilde{\mc{B}} \ar[d] \\
\mc{B} \times \mc{B} \ar[r, "p"] & \mc{B} \times \mc{P}_S \ar[r, "q"] & \mc{B}.
\end{tikzcd}
\]
Since $S$ is the set of singular simple roots for $\lambda - \lambda'$, we have $\lambda - \lambda' \in \mf{h}^*_{S, \mb{R}} \subset \mf{h}^*_\mb{R}$. So by Proposition \ref{prop:monodromic filtered pushforward} we have
\[ \mc{K} * \mc{M} = \tilde{q}_\dagger u(\tilde{p}_*\tilde{\Delta}_{23}^\circ(\mc{K}_{\tilde{\mc{B}} \times \tilde{\mc{B}}} \boxtimes (\mc{M}_{\tilde{\mc{B}}} \otimes \mc{O}(2\rho)))).\]

Now, since $\mc{K}$ has the $S$-stalk property, it has a filtration by complexes of the form
\[ j_{w, S!}i_{w, S*} \mc{K}_w[n], \quad \text{for $w \in W$, $n \geq 0$, $\mc{K}_w \in \mhm_{\widetilde{\mu - \rho}, \widetilde{-\lambda' - \rho}}(\tilde{X}_w)$}, \]
where $\tilde{X}_w$ is the pre-image of $X_w$ in $\tilde{\mc{B}} \times \tilde{\mc{B}}$. So it suffices to prove the vanishing with $\mc{K}' = j_{w, S!}i_{w, S*}\mc{K}_w$ in place of $\mc{K}$. Let us write 
\[ \mc{N} = (\tilde{\Delta}_{23}|_{\tilde{X}_w})^\circ (\mc{K}_{w, \tilde{X}_w} \boxtimes (\mc{M}_{\tilde{\mc{B}}} \otimes \mc{O}(2\rho))) \in \mhm_{\widetilde{\mu - \rho}, \widetilde{\lambda - \lambda'}}(\tilde{X_w}),\]
so that
\[ \tilde{\Delta}_{23}^\circ (\mc{K}_{\tilde{\mc{B}} \times \tilde{\mc{B}}} \boxtimes(\mc{M}_{\tilde{\mc{B}}} \otimes \mc{O}(2\rho))) = \tilde{j}_{w, S!} \tilde{i}_{w, S*}\mc{N},\]
where $\tilde{j}_{w, S}$ and $\tilde{i}_{w, S}$ are the inclusions
\[ \tilde{X}_w \xrightarrow{\tilde{i}_{w, S}} \tilde{p}^{-1}\tilde{p}(\tilde{X}_w) \xrightarrow{\tilde{j}_{w, S}} \tilde{\mc{B}} \times \tilde{\mc{B}}.\]
By Proposition \ref{prop:monodromic ! vs *}, we have
\[\tilde{p}_* \tilde{j}_{w, S!} \tilde{i}_{w, S*}\mc{N} = \tilde{p}_! \tilde{j}_{w, S!} \tilde{i}_{w, S*}\mc{N}[|S|] = \tilde{j}_{p(X_w)!} \tilde{p}_{w!} \tilde{i}_{w, S*}\mc{N}[|S|] =  \tilde{j}_{p(X_w)!} \tilde{p}_{w*} \tilde{i}_{w, S*}\mc{N}\]
where the notation is as in the diagram
\[
\begin{tikzcd}
\tilde{p}^{-1}\tilde{p}(\tilde{X}_w) \ar[r, "\tilde{j}_{w, S}"] \ar[d, "\tilde{p}_w"] & \tilde{\mc{B}} \times \tilde{\mc{B}} \ar[d, "\tilde{p}"] \\
\tilde{p}(\tilde{X}_w) \ar[r, "\tilde{j}_{p(X_w)}"] & \tilde{\mc{B}} \times \tilde{\mc{P}}_S.
\end{tikzcd}
\]
Since $\tilde{p}_w \circ \tilde{i}_{w, S}$ is an affine morphism and $\tilde{j}_{p(X_w)}$ is an affine immersion, we deduce that
\[ \mc{H}^i(\tilde{p}_*\tilde{\Delta}_{23}^\circ (\mc{K}_{\tilde{\mc{B}} \times \tilde{\mc{B}}} \boxtimes(\mc{M}_{\tilde{\mc{B}}} \otimes \mc{O}(2\rho)))) = 0 \quad \text{for $i > 0$}.\]
Finally, since $\lambda - \lambda'$ is dominant and $S$ is its set of singular roots, $(\mu - \rho, \lambda - \lambda')$ is $q$-ample, so
\[ \mc{H}^i\Gr^F(\mc{K} *\mc{M}) = \mc{H}^i\Gr^F \tilde{q}_\dagger u(\tilde{p}_*\tilde{\Delta}_{23}^\circ (\mc{K}_{\tilde{\mc{B}} \times \tilde{\mc{B}}} \boxtimes(\mc{M}_{\tilde{\mc{B}}} \otimes \mc{O}(2\rho)))(\dim \tilde{\mc{B}})) = 0\]
for $i > 0$ by Theorem \ref{thm:monodromic kodaira}.
\end{proof}

\section{Global generation of the Hodge filtration} \label{sec:hodge generation}

We now turn to the proof of Theorems \ref{thm:hodge generation} and \ref{thm:localization hodge module}. We first note that Theorem \ref{thm:localization hodge module} implies Theorem \ref{thm:hodge generation}.

\begin{proof}[Proof of Theorem \ref{thm:hodge generation} modulo Theorem \ref{thm:localization hodge module}]
Assume $\mc{M}$ is globally generated, i.e., that the morphism $\Delta \Gamma(\mc{M}) \to \mc{M}$ is surjective. By Theorem \ref{thm:localization hodge module}, this underlies a morphism of mixed Hodge modules and is hence filtered surjective. But this is precisely the statement that $F_\bullet \mc{M}$ is globally generated, so we are done.
\end{proof}

It remains to prove Theorem \ref{thm:localization hodge module}. The proof proceeds in three steps. In \S\ref{subsec:localization convolution}, we first rewrite the filtered module $\Delta\Gamma(\mc{M}, F_\bullet)$ and its morphism to $(\mc{M}, F_\bullet)$ in terms of convolutions of $(\mc{M}, F_\bullet)$ with certain good pro-objects in $\pro \mhm(\mc{D}_{\widetilde{0}} \boxtimes \mc{D}_{\widetilde{0}}, G)$. This step makes use of Proposition \ref{prop:convolution vanishing} and the $S$-stalk property of the big projective $\tilde{\Xi}^H$. In \S\ref{subsec:generation deforming}, we use a simple case of the deformation techniques of \S\ref{sec:deformations} to replace the objects in $\pro \mhm(\mc{D}_{\widetilde{0}} \boxtimes \mc{D}_{\widetilde{0}}, G)$ with objects in $\pro \mhm(\mc{D}_{\widetilde{\lambda}} \boxtimes \mc{D}_{\widetilde{-\lambda}}, G)$. Finally, in \S\ref{subsec:mhm convolution}, we construct a mixed Hodge module with underlying filtered module $\Delta\Gamma(\mc{M}, F_\bullet)$ by relating the filtered convolution to a convolution for (pro-)mixed Hodge modules.

\subsection{The localization as a convolution} \label{subsec:localization convolution}

The first step in the proof of Theorem \ref{thm:localization hodge module} is to rewrite $\Delta \Gamma(\mc{M}, F_\bullet)$ and its morphism to $(\mc{M}, F_\bullet)$ in terms of convolutions.

Recall the inclusion of the diagonal
\[ j_1 \colon \mc{B} \to \mc{B} \times \mc{B}.\]

\begin{lem} \label{lem:diagonal Dtilde}
The pro-object $j_{1*}j_1^*\tilde{\Xi}^H \in \pro \mhm(\mc{D}_{\widetilde{0}} \boxtimes \mc{D}_{\widetilde{0}}, G)$ is good, and we have a filtered isomorphism
\[ (j_{1*}j_1^*\tilde{\Xi}^H, F_\bullet) \cong (\tilde{\mc{D}}, F_\bullet)\]
such that
\[ \tilde{\mc{D}} \otimes_{U(\mf{g})} \tilde{\mc{D}} = \tilde{\Xi}^H \to j_{1*}j_1^*\tilde{\Xi}^H \cong \tilde{\mc{D}} \]
is the multiplication map.
\end{lem}
\begin{proof}
We have
\[ (j_{1*}j_1^*\tilde{\Xi}^H) \overset{\mrm{L}}\otimes_{S(\mf{h}(1))}\mb{C} = j_{1*}j_1^*(\tilde{\Xi}^H \overset{\mrm{L}}\otimes_{S(\mf{h}(1))}\mb{C}) = j_{1*}j_1^*\Xi^H = j_{1*}\gamma_1(-\dim \mc{B}),\]
so $j_{1*}j_1^*\tilde{\Xi}^H$ is good. Moreover, the filtered $\tilde{\mc{D}} \boxtimes \mc{D}_0$-module underlying $j_{1*}\gamma_1(-\dim \mc{B})$ is $(\mc{D}_0, F_\bullet)$. As in the proof of Theorem \ref{thm:xitilde hodge filtration}, we deduce that
\[ F_0 j_{1*}j_1^*\tilde{\Xi}^H = F_0 j_{1*}j_1^*\Xi^H = j_{1\bigcdot}\mc{O}_{\mc{B}}.\]
Now, writing $\mc{U}_{\mc{B} \times \mc{B}}(\mf{g}) = \mc{O}_{\mc{B} \times \mc{B}} \otimes U(\mf{g})$, we have
\[ \tilde{\mc{D}} = (\tilde{\mc{D}} \boxtimes \tilde{\mc{D}}) \otimes_{\mc{U}_{\mc{B} \times \mc{B}}(\mf{g})} j_{1\bigcdot}\mc{O}_\mc{B}.\]
So we obtain a canonical filtered morphism
\begin{equation} \label{eq:diagonal Dtilde 1}
 \tilde{\mc{D}} \to (j_{1*}j_1^*\tilde{\Xi}^H, F_\bullet),
 \end{equation}
which is an isomorphism after tensoring over $S(\mf{h}(1))$ with $\mb{C}$, and hence an isomorphism itself by Proposition \ref{prop:good pro-objects}. Finally, to prove the last statement, we note that the multiplication map is the unique non-zero filtered morphism $\tilde{\mc{D}} \otimes_{U(\mf{g})} \tilde{\mc{D}} \to \tilde{\mc{D}}$ up to scale; the assertion therefore holds after rescaling \eqref{eq:diagonal Dtilde 1} if necessary. 
\end{proof}

Now let us fix $\lambda \in \mf{h}^*_\mb{R}$ dominant, let $S$ denote the set of simple roots $\alpha$ such that $\langle \lambda, \check\alpha \rangle = 0$, let $\mc{P}_S$ denote the corresponding partial flag variety, and let
\[ j_S \colon X_S: = \mc{B} \times_{\mc{P}_S} \mc{B} \to \mc{B} \times \mc{B} \]
denote the inclusion.

\begin{lem} \label{lem:xitilde S}
The pro-object $j_{S*}j_S^*\tilde{\Xi}^H \in \pro \mhm(\mc{D}_{\widetilde{0}} \boxtimes \mc{D}_{\widetilde{0}}, G)$ is good. Moreover, for any $\mc{M} \in \mhm(\mc{D}_{\widetilde{\lambda}})$, we have a filtered isomorphisn
\begin{equation} \label{eq:xitilde S 1}
\Delta\Gamma(\mc{M}, F_\bullet) \cong \mc{H}^0((j_{S*}j_S^*\tilde{\Xi}^H, F_\bullet)*(\mc{M}, F_\bullet))
\end{equation}
such that the counit $\Delta\Gamma(\mc{M}, F_\bullet) \to (\mc{M}, F_\bullet)$ of the adjunction between $\Delta$ and $\Gamma$ agrees with the canonical map
\[ (j_{S*}j_S^*\tilde{\Xi}^H, F_\bullet)*(\mc{M}, F_\bullet) \to (j_{1*}j_1^*\tilde{\Xi}^H, F_\bullet) * (\mc{M}, F_\bullet) = (\tilde{\mc{D}}, F_\bullet) *(\mc{M}, F_\bullet) = (\mc{M}, F_\bullet).\]
\end{lem}
\begin{proof}
That $j_{S*}j_S^*\tilde{\Xi}^H$ is good follows by the same argument as Lemma \ref{lem:diagonal Dtilde}. To obtain the isomorphism \eqref{eq:xitilde S 1}, note that by Lemma \ref{lem:localize globalize convolution} and Theorem \ref{thm:filtered exactness}, we have
\[ \tilde{\mc{D}} \otimes_{U(\mf{g})} \Gamma(\mc{M}) = \mc{H}^0(\tilde{\mc{D}} \overset{\mrm{L}}\otimes_{U(\mf{g})} \mrm{R}\Gamma(\mc{M})) = \mc{H}^0(\tilde{\Xi} * \mc{M}).\]
We first claim that the induced morphism
\begin{equation} \label{eq:xitilde S 2}
\tilde{\mc{D}} \otimes_{U(\mf{g})} \Gamma(\mc{M}) = \mc{H}^0(\tilde{\Xi} * \mc{M}) \to \mc{H}^0(j_{S*}j_S^*\tilde{\Xi}^H*\mc{M})
\end{equation}
is filtered surjective.

To see this, let $\mc{K}$ (resp., $\tilde{\mc{K}}$) denote the kernel of the surjective morphism $\Xi^H \to j_{S*}j_S^*\Xi^H$ (resp., $\tilde{\Xi}^H \to j_{S*}j_S^*\tilde{\Xi}^H$). We have
\[ i_{w, S}^!j_{w, S}^*\mc{K} = \begin{cases} 0, &\text{if $X_w \subset X_S$} \\ i_{w, S}^!j_{w, S}^*\Xi^H, & \text{otherwise}.\end{cases}\]
Since $\Xi$ has the $S$-stalk property by Proposition \ref{prop:xi mixed stalks}, we deduce that $\mc{K}$ does as well. Moreover,
\[ \tilde{\mc{K}} = \varprojlim_n \tilde{\mc{K}} \otimes_{S(\mf{h}(1))} S(\mf{h}(1))_n \]
where each term is an iterated extension of Tate twists of $\mc{K}$, so $\tilde{\mc{K}}$ has the $S$-stalk property. Hence, by Proposition \ref{prop:convolution vanishing},
\[ \mc{H}^i\Gr^F((\tilde{\mc{K}}, F_\bullet) * (\mc{M}, F_\bullet)) = 0 \quad \text{for $i > 0$},\]
so \eqref{eq:xitilde S 2} is filtered surjective as claimed.

To prove \eqref{eq:xitilde S 1}, it therefore remains to identify the filtered quotient
\[ \mc{H}^0(j_{S*}j_S^*\tilde{\Xi}^H * \mc{M})\]
of $\tilde{\mc{D}} \otimes_{U(\mf{g})} \Gamma(\mc{M})$ with $\Delta\Gamma(\mc{M})$. Observe that $j_{S*}j_S^*\tilde{\Xi}^H$ (resp., $\tilde{\mc{K}}$) has a filtration by the objects $j_{w!}j_w^*\tilde{\Xi}^H$ for $X_w \subset X_S$ (resp., $X_w \not\subset X_S$). Since $j_w^*\tilde{\Xi}^H \in \pro \mhm_{\widetilde{-\rho}, \widetilde{-\rho}}^G(\tilde{X}_w)$, it follows that
\[ a^L(h) + a^R(w^{-1}h) = -\rho(h) - \rho(w^{-1} h) \quad \text{for $h \in \mf{h}$},\]
where we write $a^L$ (resp., $a^R$) for the action of the first (resp., second) copy of $S(\mf{h}) \subset \tilde{\mc{D}}$ on $j_{w*}j_w^*\tilde{\Xi}^H$. Since the side-changing isomorphism $\tilde{\mc{D}} \cong \tilde{\mc{D}}^{\mathit{op}}$ acts on $S(\mf{h})$ by $h \mapsto -h - 2\rho(h)$, we deduce that $h \in \mf{h}$ acts on the complex of $\tilde{\mc{D}}$-modules underlying
\[ (j_{w!}j_w^*\tilde{\Xi}^H, F_\bullet) * (\mc{M}, F_\bullet)\]
with generalized eigenvalue $\lambda(w^{-1}h) - \rho(h) = (w\lambda - \rho)(h)$. Since $w\lambda = \lambda$ if and only if $X_w \subset X_S$, we deduce that, ignoring the filtrations,
\[ \mc{H}^0(\mc{K} * \mc{M}) = \bigoplus_{\substack{\mu \in W\lambda \\ \mu \neq \lambda}} (\tilde{\mc{D}} \otimes_{U(\mf{g})} \Gamma(\mc{M}))_{\widetilde{\mu}}\]
and
\[ \mc{H}^0(j_{S*}j_S^*\tilde{\Xi}^H * \mc{M}) = (\tilde{\mc{D}} \otimes_{U(\mf{g})} \Gamma(\mc{M}))_{\widetilde{\lambda}} = \Delta\Gamma(\mc{M}).\]
This proves \eqref{eq:xitilde S 2}. Finally, the identification of the counit is clear from Lemma \ref{lem:diagonal Dtilde}.
\end{proof}

\subsection{Deforming to $\lambda$} \label{subsec:generation deforming}

The next step in the proof of Theorem \ref{thm:localization hodge module} is:

\begin{lem} \label{lem:xitilde deformation}
There exist good pro-objects $\mc{K}_{S,\lambda}$ and $\mc{K}_{1, \lambda}$ in $\pro \mhm(\mc{D}_{\widetilde{\lambda}} \boxtimes \mc{D}_{\widetilde{-\lambda}}, G)$, a morphism $\mc{K}_{S, \lambda} \to \mc{K}_{1, \lambda}$, and isomorphisms of filtered $\tilde{\mc{D}} \boxtimes \tilde{\mc{D}}$-modules
\begin{equation} \label{eq:xitilde deformation 1}
(\mc{K}_{S, \lambda}, F_\bullet) \cong (j_{S*}j_S^*\tilde{\Xi}^H, F_\bullet) \quad \text{and} \quad (\mc{K}_{1, \lambda}, F_\bullet) \cong (j_{1*}j_1^*\tilde{\Xi}^H, F_\bullet)
\end{equation}
identifying $\mc{K}_{S, \lambda} \to \mc{K}_{1, \lambda}$ with $j_{S*}j_S^*\tilde{\Xi}^H \to j_{1*}j_1^*\tilde{\Xi}^H$.
\end{lem}
\begin{proof}
For $\mu \in \mb{X}^*(H_S) = \mrm{Pic}^G(\mc{P}_S)$, the line bundle $\mc{O}(\mu, -\mu)$ is $G$-equivariantly trivial on $X_S$. Since $\langle \lambda, \check\alpha\rangle = 0$ for $\alpha \in S$, we have $\lambda \in \mb{X}^*(H_S) \otimes \mb{R}$. So, in the notation of \S\ref{sec:deformations}, there exists $f \in \Gamma_\mb{R}^G(X_S)^\mon$ such that $\varphi(f) = (\lambda, -\lambda)$. We set
\[ \mc{K}_{S, \lambda} = j_{S*}fj_S^*\tilde{\Xi}^H \quad \text{and} \quad \mc{K}_{1, \lambda} = j_{1*}fj_1^*\tilde{\Xi}^H \in \pro \mhm(\mc{D}_{\widetilde{\lambda}} \boxtimes\mc{D}_{\widetilde{-\lambda}}, G)\]
and let $\mc{K}_{S, \lambda} \to \mc{K}_{1, \lambda}$ be the obvious morphism. Clearly both pro-objects are good.

To construct the filtered isomorphisms \eqref{eq:xitilde deformation 1}, consider the filtered $\tilde{\mc{D}} \boxtimes \tilde{\mc{D}}$-modules $j_{S+}f^sj_S^*\tilde{\Xi}^H[s]$ and $j_{1+}f^sj_1^*\tilde{\Xi}^H[s]$ of \S\ref{subsec:semi-continuity} and the quotient maps
\[ j_{S+}f^sj_S^*\tilde{\Xi}^H[s] \to j_{S+}j_S^*\tilde{\Xi}^H \quad \text{and} \quad j_{1+}f^sj_1^*\tilde{\Xi}^H[s] \to j_{1+}j_1^*\tilde{\Xi}^H\]
sending $s$ to $0$. Note that since $j_1$ and $j_S$ are closed immersions, we have $j_{1+} = j_{1*}$ and $j_{S+} = j_{S*}$. We claim that there exist filtered splittings
\begin{equation} \label{eq:xitilde deformation 2}
 j_{S*}j_S^* \tilde{\Xi}^H \to j_{S+}f^sj_S^*\tilde{\Xi}^H[s] \quad \text{and} \quad j_{1*}j_1^*\tilde{\Xi}^H \to j_{1+}f^sj_1^*\tilde{\Xi}^H[s]
\end{equation}
such that the diagram
\[
\begin{tikzcd}
j_{S*}j_S^* \tilde{\Xi}^H \ar[r] \ar[d] & j_{1*}j_1^*\tilde{\Xi}^H \ar[d] \\
j_{S+}f^sj_S^*\tilde{\Xi}^H[s] \ar[r] & j_{1+}f^sj_1^*\tilde{\Xi}^H[s] 
\end{tikzcd}
\]
commutes. We then obtain the isomorphisms \eqref{eq:xitilde deformation 1} by composing with the quotient maps
\[ j_{S+}f^sj_S^*\tilde{\Xi}^H[s] \to j_{S+} fj_S^*\tilde{\Xi}^H = \mc{K}_{S, \lambda} \quad \text{and} \quad j_{1+}fj_1^*\tilde{\Xi}^H[s] \to j_{1+}fj_1^*\tilde{\Xi}^H = \mc{K}_{1, \lambda}\]
setting $s$ to $1$.

To prove the claim, consider the pro-objects
\[ j_{S*}f^s j_S^*\tilde{\Xi}^H[[s]] = \varprojlim_n j_{S*}\frac{f^s j_S^*\tilde{\Xi}^H[s]}{(s^n)}, \quad \text{and} \quad  j_{1*} f^s j_1^*\tilde{\Xi}^H[[s]] = \varprojlim_n j_{1*} \frac{f^sj_1^*\tilde{\Xi}^H[s]}{(s^n)}\]
in $\pro \mhm(\mc{D}_{\widetilde{0}} \boxtimes \mc{D}_{\widetilde{0}}, G)$. These pro-objects are not good, but nevertheless
\[ F_p j_{S*}f^s j_S^*\tilde{\Xi}^H[[s]] = F_p j_{S+} f^s j_S^*\tilde{\Xi}^H[s] \quad \text{and} \quad F_p j_{1*}f^s j_1^*\tilde{\Xi}^H[[s]] = F_p j_{1+}f^sj_1^*\tilde{\Xi}^H[s] \]
are constant for all $p$. Now, arguing as in the proof of Proposition \ref{prop:xi mixed stalks}, we see that $j_{S*}j_S^*\Xi^H$ and $j_{1*}j_1^*\Xi^H$ are projective objects in the categories of modules in $\Mod(\mc{D}_{\widetilde{0}} \boxtimes \mc{D}_0, G)$ supported on $X_S$ and $X_1$ respectively. So as in the proof of Theorem \ref{thm:xitilde hodge filtration}, we conclude that $j_{S*}j_S^*\tilde{\Xi}^H$ and $j_{1*}j_1^*\tilde{\Xi}^H$ are projective objects in the corresponding subcategories of $\pro \Mod(\mc{D}_{\widetilde{0}} \boxtimes \mc{D}_{\widetilde{0}}, G)$. So we have surjective morphisms of pro-mixed Hodge structures
\[
\hom(j_{1*}j_1^*\tilde{\Xi}^H, j_{1*}f^s j_1^*\tilde{\Xi}^H[[s]] ) \to \hom(j_{1*}j_1^*\tilde{\Xi}^H, j_{1*}j_1^*\tilde{\Xi}^H)
\]
and
\[\hom(j_{S*}j_S^*\tilde{\Xi}^H, j_{S*}f^s j_S^*\tilde{\Xi}^H[[s]]) \to \hom(j_{S*}j_S^*\tilde{\Xi}^H, j_{S*}j_S^*\tilde{\Xi}^H \times_{j_{1*}j_1^*\tilde{\Xi}^H} j_{1*} f^sj_1^*\tilde{\Xi}^H[[s]]).\]
Hence, we may lift the identity in $F_0\hom(j_{1*}j_1^*\tilde{\Xi}^H, j_{1*}j_1^*\tilde{\Xi}^H)$ to a morphism in
\[ g \in F_0\hom(j_{1*}j_1^*\tilde{\Xi}^H, j_{1*}f^s j_1^*\tilde{\Xi}^H[[s]] ),\]
and then the morphism
\[ (\id, g) \in F_0\hom(j_{S*}j_S^*\tilde{\Xi}^H, j_{S*}j_S^*\tilde{\Xi}^H \times_{j_{1*}j_1^*\tilde{\Xi}^H} j_{1*} f^sj_1^*\tilde{\Xi}^H[[s]]) \]
to a morphism $h \in F_0 \hom(j_{S*}j_S^*\tilde{\Xi}^H, j_{S*}f^sj_S^*\tilde{\Xi}^H[[s]])$. These define the desired splittings \eqref{eq:xitilde deformation 2}.
\end{proof}

\subsection{Convolution of mixed Hodge modules} \label{subsec:mhm convolution}

The final step in the proof of Theorem \ref{thm:localization hodge module} is to apply the following proposition.

Suppose that $\lambda, \mu \in \mf{h}^*_\mb{R}$, $\mc{K} \in \pro\mhm(\mc{D}_{\widetilde{\mu}} \boxtimes \mc{D}_{\widetilde{-\lambda}}, G)$ and $\mc{M} \in \mhm(\mc{D}_{\widetilde{\lambda}})$. In the notation of the proof of Proposition \ref{prop:convolution vanishing}, we define
\[ \mc{K} * \mc{M} = \widetilde{\mrm{pr}}_{1*} \tilde{\Delta}_{23}^\circ(\mc{K}_{\tilde{\mc{B}} \times \tilde{\mc{B}}} \boxtimes (\mc{M}_{\tilde{\mc{B}}} \otimes \mc{O}(2\rho)))(\dim \tilde{\mc{B}}) \in \pro \mrm{D}^b\mhm(\mc{D}_{\widetilde{\mu}}).\]

\begin{prop} \label{prop:good convolution}
Assume that the pro-object $\mc{K}$ is good. Then the pro-objects $\mc{H}^i(\mc{K} * \mc{M})$ are constant and
\begin{equation} \label{eq:good convolution 1}
(\mc{H}^i(\mc{K} * \mc{M}), F_\bullet) = \mc{H}^i((\mc{K}, F_\bullet) * (\mc{M}, F_\bullet)).
\end{equation}
\end{prop}
\begin{proof}
We will show that the pro-object $\mc{H}^i(\mc{K} * \mc{M})$ is constant; the property \eqref{eq:good convolution 1} then follows from \eqref{eq:convolution formula 2}.

First, we observe that if 
\[ 0 \to \mc{M}_1 \to \mc{M}_2 \to \mc{M}_3 \to 0\]
is a short exact sequence and $\mc{K} * \mc{M}_1$ and $\mc{K} * \mc{M}_3$ have constant cohomologies, then so does $\mc{K} * \mc{M}_2$. So we may assume without loss of generality that $\mc{M} \in \mhm(\mc{D}_\lambda)$. In this case, we have that
\[ \tilde{\Delta}_{23}^\circ (\mc{K}_{\tilde{\mc{B}} \times \tilde{\mc{B}}} \boxtimes (\mc{M}_{\tilde{\mc{B}}} \otimes \mc{O}(2\rho))) \overset{\mrm{L}}\otimes_{S(\mf{h}(1))} \mb{C} = \tilde{\Delta}_{23}^\circ ((\mc{K}_{\tilde{\mc{B}} \times \tilde{\mc{B}}} \overset{\mrm{L}}\otimes_{S(\mf{h}(1))} \mb{C})  \boxtimes (\mc{M}_{\tilde{\mc{B}}} \otimes \mc{O}(2\rho)))\]
has constant cohomologies. Now, writing $\pi_1$ for the projection
\[ \pi_1 \colon \tilde{\mc{B}} \times \tilde{\mc{B}} \to \tilde{\mc{B}} \times \mc{B},\]
we have, by Lemma \ref{lem:monodromic push/pull}
\[ \mc{N} \overset{\mrm{L}}\otimes_{S(\mf{h}(1))} \mb{C} = \pi_1^\circ \pi_{1*}\mc{N} \otimes \det(\mf{h}) (\dim \mf{h}) \]
for any $\mc{N} \in \mhm_{\widetilde{\mu - \rho}, \widetilde{0}}(\tilde{\mc{B}} \times \tilde{\mc{B}})$. Since
\[ \pi^\circ \colon \mhm_{\widetilde{\mu - \rho}}(\tilde{\mc{B}} \times \mc{B}) \to \mhm_{\widetilde{\mu - \rho}, \widetilde{0}}(\tilde{\mc{B}} \times \tilde{\mc{B}}) \]
is faithful and exact, we deduce that
\[ \pi_{1*}\tilde{\Delta}_{23}^\circ (\mc{K}_{\tilde{\mc{B}} \times \tilde{\mc{B}}} \boxtimes (\mc{M}_{\tilde{\mc{B}}} \otimes \mc{O}(2\rho))) \]
has constant cohomologies. So the pushforward to $\tilde{\mc{B}}$, and hence $\mc{K} * \mc{M}$, also has constant cohomologies as claimed.
\end{proof}

We can now complete the proof of Theorem \ref{thm:localization hodge module}.

\begin{proof}[Proof of Theorem \ref{thm:localization hodge module}]
By Lemmas \ref{lem:diagonal Dtilde}, \ref{lem:xitilde S} and \ref{lem:xitilde deformation}, we may identify the counit
\[ \Delta\Gamma(\mc{M}, F_\bullet) \to (\mc{M}, F_\bullet) \]
with the morphism
\[ \mc{H}^0((\mc{K}_{S, \lambda}, F_\bullet) * (\mc{M}, F_\bullet)) \to \mc{H}^0((\mc{K}_{1, \lambda}, F_\bullet) * (\mc{M}, F_\bullet)).\]
But by Proposition \ref{prop:good convolution}, this is obtained from the morphism
\[ \mc{H}^0(\mc{K}_{S, \lambda} * \mc{M}) \to \mc{H}^0(\mc{K}_{1, \lambda} * \mc{M}) \]
of mixed Hodge modules, so we are done.
\end{proof}

\begin{rmk}
As we have set things up, it is not necessarily the case that
\[ \mc{H}^0(\mc{K}_{1, \lambda} * \mc{M}) \cong \mc{M} \]
as mixed Hodge modules; we have only shown that this holds at the level of filtered $\tilde{\mc{D}}$-modules. In fact, the mixed Hodge structure on the left hand side depends on our choice of Hodge structure on the big pro-projective $\tilde{\Xi}^H$. Since we are mainly interested in the Hodge filtration, this does not cause us any issues. Nevertheless, it is possible to choose the Hodge structure on $\tilde{\Xi}^H$ so that one \emph{does} obtain an isomorphism of Hodge modules as above. We leave the details to the interested reader.
\end{rmk}

\section{The Harish-Chandra setting} \label{sec:harish-chandra}

In this section, we consider what Hodge theory can say about the unitary representations of a real reductive group. In \S\ref{subsec:intro harish-chandra}, we recall how localization theory endows irreducible Harish-Chandra modules with Hodge filtrations and state our Hodge-theoretic unitarity criterion (whose proof is deferred to \S\S\ref{sec:tempered}--\ref{sec:pf of unitarity}). We conclude in \S\ref{subsec:cohomological induction} with a brief discussion of the link between our story and cohomological induction.

\subsection{The unitarity criterion} \label{subsec:intro harish-chandra}

For our purposes, a \emph{real reductive group}, or \emph{real group}, is a Lie group $G_\mb{R}$ such that there exists a complex reductive group $G$ and an anti-holomorphic involution $\sigma \colon \bar{G} \to G$ so that $G_\mb{R}$ is a finite cover of a union of connected components of $G^\sigma$. We fix always a compact real form $U_\mb{R} = G^{\sigma_c} \subset G$ whose complex conjugation $\sigma_c \colon \bar{G} \to G$ commutes with $\sigma$ and write $\theta = \sigma \circ \sigma_c \colon G \to G$ for the induced algebraic involution (the \emph{Cartan involution}). We write $K_\mb{R} \subset G_\mb{R}$ for the pre-image of $U_\mb{R}$ (a maximal compact subgroup of $G_\mb{R}$) and $K$ for its complexification (a finite covering of a union of connected components of $G^\theta$). We also write $\mf{g} = \mrm{Lie}(G)$, $\mf{g}_\mb{R} = \mrm{Lie}(G_\mb{R})$ and $\mf{u}_\mb{R} = \mrm{Lie}(U_\mb{R})$.

We remark that many references (e.g., \cite{ALTV} and \cite{DV1}) adopt a more restrictive definition of real group; namely, they require that $G_\mb{R} = G^\sigma$. We will call such groups \emph{linear}. The main advantage of linear groups is that the combinatorics governing their representation theory (especially their Kazhdan-Lusztig theory) is much better understood than for a general real group. A related distinction is that the Cartan subgroups of a linear group are always commutative (in fact, are maximal tori defined over $\mb{R}$), while this need not be the case in the non-linear setting. Our arguments are sufficiently general, however, being mostly geometric rather than combinatorial in nature, that non-linear groups do not cause us any significant additional difficulties (the only exception being \S\ref{subsec:split tempered}, where a small amount of extra work is required).

We are interested in the classical problem of classifying the irreducible unitary representations of $G_\mb{R}$. While this is {\it a priori} a question in analysis, it was reduced by Harish-Chandra to the more algebraic problem of determining the irreducible unitary Harish-Chandra modules. We recall that a \emph{Harish-Chandra module} is an admissible, finitely generated $(\mf{g}, K)$-module $V$, and that $V$ is called \emph{unitary} (or \emph{unitarizable}) if it admits a positive definite $(\mf{g}_\mb{R}, K_\mb{R})$-invariant Hermitian form.

Since the irreducible Harish-Chandra modules are known, we are left with the problem of determining which of these are unitary. It is not difficult to reduce this question to the case of real infinitesimal character, i.e., to Harish-Chandra modules on which the center $Z(U(\mf{g}))$ acts by a character $\chi_\lambda$ for $\lambda \in \mf{h}^*_\mb{R}$ (see, for example, \cite[Chapter 16]{knapp1986book}). A key insight (due to \cite{ALTV}) is that in this setting one always has a $(\mf{u}_\mb{R}, K_\mb{R})$-invariant Hermitian form, to which the $(\mf{g}_\mb{R}, K_\mb{R})$-invariant form can be related when it exists. More precisely, we have the following.

\begin{prop}[e.g., \cite{ALTV}] \label{prop:gr vs ur}
Let $V$ be an irreducible $(\mf{g}, K)$-module with real infinitesimal character.
\begin{enumerate}
\item There exists a non-degenerate $(\mf{u}_\mb{R}, K_\mb{R})$-invariant Hermitian form $\langle\,,\,\rangle_{\mf{u}_\mb{R}}$ on $V$.
\item The module $V$ carries a non-zero $(\mf{g}_\mb{R}, K_\mb{R})$-invariant Hermitian form if and only if $V \cong \theta^*V$ (where $\theta^*$ denotes the twist of the $\mf{g}$-action via the involution $\theta$). Moreover, if we fix an isomorphism $\theta \colon V \to \theta^*V$ such that $\theta^2 = 1$, then
\[ \langle u, \bar{v}\rangle_{\mf{g}_\mb{R}} := \langle \theta(u), \bar{v}\rangle_{\mf{u}_\mb{R}} \]
is such a form.
\end{enumerate}
\end{prop}

Thus, the problem of determining the unitarity of $V$ is reduced to the problem of computing the signature of a $\mf{u}_\mb{R}$-invariant form and the action of $\theta$; of course, since $V$ is assumed irreducible, the $(\mf{u}_\mb{R}, K_\mb{R})$-invariant Hermitian form is unique up to a real scalar. In practice, it is important to normalize the sign of this form. In \cite{ALTV}, Adams, van Leeuwen, Trapa and Vogan showed that, at least when the group $G_\mb{R}$ is linear, the form can be normalized to be positive definite on the minimal $K$-types; using this normalization, they were able to write down a recursive algorithm to compute its signature. Conjecture \ref{conj:schmid-vilonen}, on the other hand, addresses the problem as follows.

First, the Beilinson-Bernstein localization theory provides a functor
\[ \Gamma \colon \Mod(\mc{D}_{\widetilde{\lambda}}, K) \to \Mod(\mf{g}, K)_{\widetilde{\chi_\lambda}},\]
which is an equivalence for $\lambda \in \mf{h}^*$ regular and integrally dominant. It is a well-known fact  that a finitely generated $(\mf{g}, K)$-module with generalized infinitesimal character is automatically admissible. Moreover, the subgroup $G^\theta \subset G$, and hence $K$, acts on the (complex) flag variety $\mc{B}$ with finitely many orbits, so any finitely generated module in $\Mod(\mc{D}_{\widetilde{\lambda}}, K)$ is necessarily regular holonomic. Therefore, $\Gamma$ restricts to a functor
\begin{equation} \label{eq:hc global sections}
 \Gamma \colon \HC(\mc{D}_{\widetilde{\lambda}}, K) := \Mod(\mc{D}_{\widetilde{\lambda}}, K)_{rh} \to \HC(\mf{g}, K)_{\widetilde{\chi_\lambda}},
\end{equation}
which is again an equivalence for $\lambda \in \mf{h}^*$ regular and integrally dominant. Here $\HC(\mf{g}, K)_{\widetilde{\chi_\lambda}}$ denotes the category of Harish-Chandra modules with generalized infinitesimal character $\chi_\lambda$. We call the objects in the domain \emph{Harish-Chandra sheaves}. When $\lambda$ is integrally dominant but not necessarily regular, the functor \eqref{eq:hc global sections} is exact and sends irreducibles to irreducibles or zero, and each irreducible object in the target is the image of a unique irreducible object in the domain. Thus, each irreducible Harish-Chandra module $V$ with real infinitesimal character may be written uniquely as $V = \Gamma(\mc{M})$, where $\mc{M} \in \HC(\mc{D}_\lambda, K)$ is irreducible and $\lambda \in \mf{h}^*_\mb{R}$ is (real) dominant.

As explained in \cite[\S 2.3]{DV1}, for example, any such $\mc{M}$ is the intermediate extension of an equivariantly irreducible unitary local system on a $K \times H$-orbit in $\tilde{\mc{B}}$, and hence has a lift to a polarizable Hodge module in $\mhm(\mc{D}_\lambda, K)$, unique up to Hodge twists. Unless otherwise specified, we will generally endow $\mc{M}$ with the standard Hodge structure of weight $\dim \operatorname{Supp}\mc{M}$ with Hodge filtration starting in degree $\codim \operatorname{Supp}\mc{M}$. Integrating the polarization $S$ on $\mc{M}$ produces a non-degenerate $(\mf{u}_\mb{R}, K_\mb{R})$-invariant Hermitian form $\Gamma(S)$ on $V = \Gamma(\mc{M})$. Once the Hodge structure is fixed, the polarization is unique up to a \emph{positive} real scalar, so this gives a completely general way to fix the sign of the $\mf{u}_\mb{R}$-invariant form. In the linear case, we have previously shown~\cite[Theorem 4.5 and Proposition 4.7]{DV1} that, with our standard choice of Hodge structure on $\mc{M}$, the form $\Gamma(S)$ is the unique one that is positive definite on the minimal $K$-types, i.e., it is precisely the same form considered in \cite{ALTV}.

Now let us consider the $(\mf{g}_\mb{R}, K_\mb{R})$-invariant forms in terms of geometry. Observe that the involution $\theta$ acts canonically on the flag variety $\mc{B}$ and hence on the universal Cartan $H = \hom_\mb{Z}(\mrm{Pic}^G(\mc{B}), \mb{C}^\times)$. We will denote the involution on $H$ by $\delta$ to avoid confusion with the action on $\theta$-stable maximal tori in $G$. Note that $\delta$ preserves the sets of positive roots and dominant weights in $\mf{h}^*$. Tautologically, we have $\theta^*\tilde{\mc{B}} \cong \tilde{\mc{B}}$ as $H$-torsors over $\mc{B}$ (after twisting the $H$-action on one side by $\delta$), so $\theta$ lifts to a compatible involution $\theta \colon \tilde{\mc{B}} \to \tilde{\mc{B}}$. This lift is not unique: we can compose it with the action of any element $h \in H$ such that $\delta(h) = h^{-1}$ to get another one. Fixing a lift, we get a pullback functor
\begin{equation} \label{eq:hodge involution}
\theta^* \colon \mhm(\mc{D}_\lambda, K) \to \mhm(\mc{D}_{\delta\lambda}, K)
\end{equation}
compatible with the functor $\theta^* \colon \HC(\mf{g}, K) \to \HC(\mf{g}, K)$. For mixed objects, the functor \eqref{eq:hodge involution} depends subtly on the choice of lift, but restricted to pure ones it does not, cf., \S\ref{subsec:choices}.

Since the pure Hodge module $\mc{M}$ is determined uniquely by $V = \Gamma(\mc{M})$, we deduce from Proposition \ref{prop:gr vs ur} that $V$ admits a $(\mf{g}_\mb{R}, K_\mb{R})$-invariant Hermitian form if and only if $\delta \lambda = \lambda$ and $\theta^*\mc{M} \cong \mc{M}$. In this case, if we fix an isomorphism $\theta \colon \mc{M} \to \theta^*\mc{M}$ such that $\theta^2 = 1$, then
\[ \langle u, \bar{v}\rangle_{\mf{g}_\mb{R}} := \Gamma(S)(\theta(u), \bar{v}) \]
is the unique (up to scale) non-degenerate $(\mf{g}_\mb{R}, K_\mb{R})$-invariant Hermitian form on $V$. Conjecture \ref{conj:schmid-vilonen} predicts the following result.

\begin{thm} \label{thm:unitarity criterion}
Let $\lambda \in \mf{h}^*_\mb{R}$ be dominant, let $\mc{M} \in \mhm(\mc{D}_\lambda, K)$ be an irreducible object such that the $(\mf{g}, K)$-module $V = \Gamma(\mc{M})$ is Hermitian, and fix an involution $\theta \colon \mc{M} \to \theta^*\mc{M}$ as above. Then $V$ is unitary if and only if $\theta$ acts on $\Gr^F_p V$ with eigenvalue $(-1)^p$ for all $p$, or with eigenvalue $(-1)^{p + 1}$ for all $p$.
\end{thm}

We prove Theorem \ref{thm:unitarity criterion} unconditionally in \S\ref{sec:pf of unitarity}.

\begin{rmk} \label{rmk:hermitian}
We note that the condition $\delta \lambda = \lambda$, $\theta^*\mc{M} \cong \mc{M}$ for the existence of an invariant Hermitian form is very easy to check in practice. For example, in the setting of linear groups, recall from \cite[\S 2.4]{DV1} that $\mc{M}$ is classified by the parameter $(Q, \lambda, \Lambda)$, where $Q \subset \mc{B}$ is a $K$-orbit, $\lambda \in \mf{h}^*_\mb{R}$ is the infinitesimal character and $(\lambda - \rho, \Lambda)$ is a character for the associated Garish-Chandra pair $(\mf{h}, H^{\theta_Q})$. The condition $\delta \lambda = \lambda$ simply picks out a linear subspace for the infinitesimal character, and the condition $\theta^*\mc{M} \cong \mc{M}$ is simply that $\theta(Q) = Q$ and $\delta \Lambda = \Lambda$.
\end{rmk}

\subsection{Cohomological induction} \label{subsec:cohomological induction}

To illustrate the power of our unitarity criterion (Theorem \ref{thm:unitarity criterion}), we use it here to give a very short proof of the classical (but highly non-trivial) fact that cohomological induction preserves unitarity in the good range \cite[Theorem 1.3]{vogan-annals}. We also give a simple formula for the behavior of the Hodge filtration under cohomological induction, which we will use in the proof of Theorem \ref{thm:unitarity criterion} itself.

We first recall the definition of (left) cohomological induction (see, for example, \cite[Chapter V]{knapp-vogan}). In the setting of the previous subsection, let $P \subset G$ be a $\theta$-stable parabolic subgroup with Levi factor $L$. We write $K_P = K \times_G P$ and $K_L = K \times_G L$ for the pre-images of $P$ and $L$ in $K$. The restriction functor
\[ \Mod(\mf{g}, K) \to \Mod(\mf{g}, K_L) \]
has a left adjoint $\Pi$ with left derived functors $\Pi_j$. If $V$ is an $(\mf{l}, K_L)$-module, the \emph{$j^\text{th}$ cohomological induction} of $V$ is
\[ \mc{L}_j(V) := \Pi_j(U(\mf{g}) \otimes_{U(\mf{p})} (V \otimes \det (\mf{g}/\mf{p})))  \in \Mod(\mf{g}, K)\]
where $\mf{p} = \mrm{Lie}(P)$ and $\mf{l} = \mrm{Lie}(L)$, and we regard $V$ as a $U(\mf{p})$-module via the quotient map $\mf{p} \to \mf{l}$. Cohomological induction sends Harish-Chandra modules to Harish-Chandra modules.

The operation of cohomological induction has the following geometric description. Suppose that $V = \Gamma(\mc{B}_L, \mc{M})$ for $\mc{M} \in \HC(\mc{D}_{\mc{B}_L, \lambda}, K_L)$, where $\mc{B}_L$ is the flag variety of $L$. We suppose for the sake of convenience that $\lambda$ is integrally dominant for $L$; in particular, the higher cohomology of $\mc{M}$ vanishes. The choice of $\theta$-stable parabolic $P$ determines a $\theta$-fixed point $x \in \mc{P}$ in the corresponding partial flag variety, an equivariant embedding $\mc{B}_L = \pi^{-1}(x) \hookrightarrow \mc{B}$, and a closed immersion 
\[ i \colon K \times^{K_P} \mc{B}_L = \pi^{-1}(Q) \hookrightarrow \mc{B},\]
where $\pi \colon \mc{B} \to \mc{P}$ is the projection and $Q = Kx \cong K/K_P \subset \mc{P}$ is the $K$-orbit of $x$. We write $\mc{M}_K$ for the $K$-equivariant twisted $\mc{D}$-module on $\pi^{-1}(Q)$ restricting to $\mc{M}$ on $\mc{B}_L$. Pushing forward to $\mc{B}$, we obtain a module
\[ i_*\mc{M}_K \in \HC(\mc{D}_{\mc{B}, \lambda + \rho_P}, K),\]
where $\rho_P = \rho - \rho_L \in \mf{h}^*$ is half the sum of the roots in $\mf{g}/\mf{p}$. The shift by $\rho_P$ just comes from the $\rho$-shift built into the definition of the rings $\mc{D}_\lambda$.

The following well-known result follows reasonably formally from the various adjunctions in play, cf., e.g., \cite[Theorem 4.3]{hmsw1} and \cite[Theorem 4.4]{bien}.

\begin{prop} \label{prop:geometric cohomological induction}
We have
\[ \mc{L}_j(V) = \mrm{H}^{S - j}(\mc{B}, i_*\mc{M}_K),\]
where $S = \dim K/K_P$.
\end{prop}

Motivated by this result, we will often say that the Harish-Chandra sheaf $i_*\mc{M}_K$ is \emph{cohomologically induced from $\mc{M}$}. When $\lambda + \rho_P$ is integrally dominant for $G$, Proposition \ref{prop:geometric cohomological induction} immediately implies the standard fact that $\mc{L}_j(V) = 0$ for $j \neq S$ and that the $(\mf{g}, K)$-module $\mc{L}_S(V)$ is either irreducible or zero. We will assume we are in this situation from now on.

Let us now assume that the irreducible $(\mf{l}, K_L)$-module $V$ (and hence $\mc{L}_S(V)$) is irreducible with real infinitesimal character. In particular, $V$ and $\mc{L}_S(V)$ are endowed with Hodge filtrations $F_\bullet$ as in the previous subsection.

\begin{prop} \label{prop:cohomological induction hodge}
Assume that $\lambda + \rho_P$ is dominant for $G$. Then, with respect to our standard choices of Hodge structure on irreducible Harish-Chandra modules,
\[ \Gr^F \mc{L}_S(V) = \mrm{Ind}_{K_P}^K(\mrm{Sym}(\mf{g}/\mf{k}) \otimes_{\mrm{Sym}(\mf{p}/\mf{k}_P)} \Gr^F V \otimes \det (\mf{g}/(\mf{k} + \mf{p})))\{c\}\]
where $\mf{k} = \mrm{Lie}(K)$, $\mf{k}_P = \mrm{Lie}(K_P)$ and $\{c\}$ denotes a grading shift by $c = \codim Q$.
\end{prop}
\begin{proof}
First note that, for any $y \in \mc{B}_L = \pi^{-1}(x) \subset \mc{B}$, we have a fiber square
\[
\begin{tikzcd}
\left(\dfrac{\mf{g}}{\mf{b}}\right)^* \ar[r] \ar[d] & \left(\dfrac{\mf{k} + \mf{p}}{\mf{b}}\right)^* \ar[d] \\
\mf{g}^* \ar[r] & (\mf{k} + \mf{p})^*
\end{tikzcd}
\]
where $\mf{b}$ is the Lie algebra of the Borel subgroup $\mrm{Stab}_G(y)$. Since the vector spaces in the top row are the cotangent spaces of $\mc{B}$ and $\pi^{-1}(Q)$ at $y$, we obtain a diagram
\begin{equation} \label{eq:cohomological induction hodge 1}
\begin{tikzcd}
T^*\pi^{-1}(Q) \ar[d, "\mu_Q"] & T^*\mc{B} \times_\mc{B} \pi^{-1}(Q) \ar[l, "\tilde{p}"'] \ar[r, "\tilde{q}"] \ar[d] & T^*\mc{B} \ar[d, "\mu"] \\
K \times^{K_P} (\mf{k} + \mf{p})^* & K \times^{K_P} \mf{g}^* \ar[l, "p"'] \ar[r, "q"] & \mf{g}^*
\end{tikzcd}
\end{equation}
in which the left hand square is Cartesian.

Now, we have $V = \Gamma(\mc{M})$, where $\mc{M} \in \mhm(\mc{D}_{\lambda, \mc{B}_L}, K_L)$ is irreducible with standard Hodge structure and $\lambda$ is dominant for $L$, and $\mc{L}_S(V) = \Gamma(i_*\mc{M}_K)$ by Proposition \ref{prop:geometric cohomological induction}. By construction, the Hodge structure on $i_*\mc{M}_K$ is the standard one; since $\lambda + \rho_L$ is dominant for $G$, the Hodge filtration on $\mc{L}_S(V)$ is therefore just given by global sections of the Hodge filtration on $i_*\mc{M}_K$. In particular,
\[ \Gr^F \mc{L}_S(V) = \Gamma(\mc{B}, \Gr^F i_*\mc{M}_K) = \mu_{\bigcdot}(\Gr^F i_*\mc{M}_K).\]
But by Proposition \ref{prop:associated graded pushforward}, we have
\[ \Gr^F i_*\mc{M}_K = \tilde{q}_{\bigcdot}(\tilde{p}^{\bigcdot} \Gr^F \mc{M}_K \otimes \omega_{\pi^{-1}(Q)/\mc{B}})\{c\} \]
so, from \eqref{eq:cohomological induction hodge 1},
\[ \mu_{\bigcdot}(\Gr^F i_*\mc{M}_K) = q_{\bigcdot} p^{\bigcdot}\mu_{Q \bigcdot}(\Gr^F \mc{M}_K \otimes \omega_{\pi^{-1}(Q)/\mc{B}}) \{c\}= \Gamma(Q, \Gr^F \mc{M}_K \otimes \omega_{Q/\mc{P}})\{c\}.\]
Since $Q = K/K_P$, this can be rewritten as
\[ \Gr^F \mc{L}_S(V) = \mrm{Ind}_{K_P}^K(\mrm{Sym}(\mf{g}) \otimes_{\mrm{Sym}(\mf{k} + \mf{p})} \Gr^F V \otimes \det (\mf{g}/(\mf{k} + \mf{p})))\{c\}.\]
Since $\mf{k} \subset \mf{k} + \mf{p}$ acts trivially on $\Gr^F V$, the statement of the proposition follows.
\end{proof}

\begin{cor}[{\cite[Theorem 1.3]{vogan-annals}}] \label{cor:cohomological induction}
Assume $\lambda + \rho_P$ is dominant for $G$. If the $(\mf{l}, K_L)$-module $V$ is unitary then so is the $(\mf{g}, K)$-module $\mc{L}_S(V)$.
\end{cor}
\begin{proof}
If $V$ is unitary, then by Theorem \ref{thm:unitarity criterion} we may choose an action of $\theta$ on $\mc{M}$ such that the action on $\Gr^F_p V$ is $(-1)^p$. This induces an involution on $i_*\mc{M}_K$ and hence on $\mc{L}_S(V)$. The induced action on $\Gr^F \mc{L}_S V$ may be read off directly from Proposition \ref{prop:cohomological induction hodge}: since $\theta$ acts on $\mf{g}/\mf{k}$ with eigenvalue $-1$ and $\det(\mf{g}/(\mf{k} + \mf{p}))$ by $(-1)^c$, we deduce that $\theta$ acts on $\Gr^F_p \mc{L}_S(V)$ with eigenvalue $(-1)^p$. Hence, $\mc{L}_S(V)$ is unitary by Theorem \ref{thm:unitarity criterion}.
\end{proof}

\section{Proof of the unitarity criterion: the tempered case} \label{sec:tempered}

In this section we prove a special case of Conjecture \ref{conj:schmid-vilonen}, the case of tempered Harish-Chandra sheaves. These are sheaves whose associated representation is tempered. The irreducible tempered Harish-Chandra modules were first classified in \cite{KZ1982}; the corresponding classification of tempered Harish-Chandra sheaves can be found in \cite{hmsw}. Tempered representations are known to be unitary, so we are reduced to verifying that the criterion in Theorem~\ref{thm:unitarity criterion} holds for them. In the next section, we will deduce Theorem~\ref{thm:unitarity criterion} in general from the tempered case using techniques developed in previous sections.

The outline of the section is as follows. In \S\ref{subsec:K-orbits}--\S\ref{subsec:cleanness}, we briefly recall the geometric description of tempered Harish-Chandra sheaves and state our main theorem concerning their Hodge filtrations. We study in detail the case of split groups in \S\ref{subsec:split tempered}, before and prove our results for general tempered sheaves in \S\ref{subsec:pf of tempered 1}--\ref{subsec:pf of tempered 2}.

\subsection{Tempered Harish-Chandra sheaves} \label{subsec:K-orbits}

In this subsection, we recall the notion of tempered Harish-Chandra modules and state our main results on their Hodge filtrations.

We begin by recalling the classification of irreducible Harish-Chandra sheaves. Let $Q \subset \mc{B}$ be a $K$-orbit and choose any point $x \in Q$. Then, by a result of Matsuki, there exists a $\theta$-stable maximal torus $T \subset G$ fixing $x$. Writing $B_x = \mrm{Stab}_G(x)$ for the associated Borel subgroup and $\tau_x \colon B_x \to H$ for the quotient by the unipotent radical, we obtain an isomorphism
\[ \tau_x \colon T \overset{\sim}\to H.\]
Transporting $\theta|_T$ along this map, we obtain an involution $\theta_Q \colon H \to H$. One can easily check that $\theta_Q$ is independent of the choice of $x \in Q$. By construction, $\theta_Q$ acts naturally on the root datum of $G$.

Let us now fix $\lambda \in \mf{h}^*_\mb{R}$ dominant and consider an irreducible $K$-equivariant $(\lambda- \rho)$-twisted local system $\gamma$ on $Q$. In many cases (e.g., for linear groups) the pairs $(Q, \gamma)$ have an explicit combinatorial parametrization, closely related to the Langlands classification; see, e.g., \cite[\S 2.4]{DV1}. For our purposes, this geometric description will suffice.

\begin{defn}
Let $j \colon Q \to \mc{B}$ denote the inclusion of the $K$-orbit $Q$. We say that $\gamma$ (or $j_{!*}\gamma$) is
\begin{enumerate}
\item \emph{relevant} if $\lambda$ is dominant and $\Gamma(j_{!*}\gamma) \neq 0$, and
\item \emph{tempered} if it is relevant and $\lambda = \theta_Q\lambda$.
\end{enumerate}
\end{defn}

In the linear case, the relevant local systems are precisely the final parameters in the sense of the {\tt atlas} software. In general, the map $\gamma \mapsto \Gamma(j_{!*}\gamma)$ defines a bijection between the relevant local systems and the irreducible Harish-Chandra modules with real infinitesimal character. By \cite[Corollary 11.7]{hmsw}, the Harish-Chandra sheaf $j_{!*}\gamma$ is tempered if and only if the Harish-Chandra module $\Gamma(j_{!*}\gamma)$ is tempered.

\begin{thm}
\label{thm:tempered hodge}
Let $j_{!*}\gamma$ be an irreducible tempered Harish-Chandra sheaf. Then:
\begin{enumerate}
\item \label{itm:tempered hodge 1}
The Hodge filtration of $j_{!*}\gamma$ is generated over $\mc{D}_\lambda$ by its lowest piece.
\item \label{itm:tempered hodge 2} Conjecture \ref{conj:schmid-vilonen} holds for $\mc{M} = j_{!*}\gamma$.
\end{enumerate}
\end{thm}

We remark that statement \eqref{itm:tempered hodge 1} in Theorem \ref{thm:tempered hodge} is a local statement about the generation of the Hodge filtration of a tempered Harish-Chandra sheaf as a $\mc{D}_\lambda$-module. In the case of spherical principal series for split groups (and certain closely related representations of non-linear groups), we also prove the global statement that the Hodge filtration of the corresponding tempered Harish-Chandra module is generated by its lowest piece (Theorem \ref{thm:split tempered}), but we do not know this in general.

The rest of this section is devoted to the proof of Theorem \ref{thm:tempered hodge}.

\subsection{Real roots and cleanness} \label{subsec:cleanness}

We begin with some recollection on the structure of relevant and tempered local systems. We will use these both in the proof of Theorem \ref{thm:tempered hodge} and in the reduction of Theorem \ref{thm:hodge and signature K'} to this case.

Let $Q \subset \mc{B}$ be a $K$-orbit.

\begin{defn}
We say that $\alpha \in \Phi$ is \emph{real} for $Q$ if $\theta_Q \alpha = -\alpha$.
\end{defn}

Let us fix $\lambda \in \mf{h}^*_\mb{R}$ dominant and set
\[ \bar{S} = \{ \alpha \in \Phi_+ \mid \langle \lambda, \check\alpha \rangle = 0 \text{ and } \theta_Q \alpha \in \Phi_-\}.\]
We will also write $S \subset \bar{S}$ for the set of simple roots in $\bar{S}$ and
\[ \pi_S \colon \mc{B} \to \mc{P}_S\]
for the projection to the corresponding partial flag variety (cf., \S\ref{subsec:S-stalk} for our conventions on partial flag varieties). The existence of a relevant (resp., tempered) local system imposes strict conditions on the geometry of $Q$ and $\pi_S$.

\begin{prop} \label{prop:relevant+tempered}
Assume that $\gamma$ is a relevant $(\lambda - \rho)$-twisted local system on $Q$.
\begin{enumerate}
\item \label{itm:relevant+tempered 1} If $\alpha \in \bar{S}$ then $\alpha$ is real.
\item \label{itm:relevant+tempered 2} The $K$-orbit $Q$ is open in $Q_S := \pi_S^{-1}\pi_S(Q)$, and $\gamma$ extends cleanly to $Q_S$ (that is, the $!$ and $*$ extensions coincide).
\item \label{itm:relevant+tempered 3} If $\gamma$ is tempered, then $\bar{S}$ is the set of all real roots and $\pi_S(Q) \subset \mc{P}_S$ is closed.
\end{enumerate}
\end{prop}
\begin{proof}
Item \eqref{itm:relevant+tempered 1} is a consequence of \cite[Proposition 2.17 and Lemma 7.5]{hmsw}, while \eqref{itm:relevant+tempered 2} follows from \cite[Theorems 8.7 and 9.1]{hmsw}. To prove \eqref{itm:relevant+tempered 3}, suppose that $\gamma$ is tempered, i.e., that $\lambda = \theta_Q\lambda$. Then we must have $\langle \lambda, \check\alpha \rangle = 0$ for any real root $\alpha$, so $\bar{S}$ is the set of real roots by \eqref{itm:relevant+tempered 1}. To show that $\pi_S(Q) \subset \mc{P}_S$ is closed, fix $x \in Q$ and a $\theta$-stable maximal torus $T$ fixing $x$. Then, according to our conventions, the Borel subgroup $B_x = \mrm{Stab}_G(x) \supset T$ has roots $\Phi_-$, and the corresponding parabolic subgroup $P_x = \mrm{Stab}_G(\pi_S(x))$ has roots $\Phi_- \cup (\mrm{span}(S)\cap \Phi)$. One easily checks that, in the situation where $\bar{S}$ is the set of real roots, $\mrm{span}(S) \cap \Phi_+ = \bar{S}$. So $P_x$ is $\theta$-stable, and hence its $K$-orbit $\pi_S(Q) \subset \mc{P}_S$ is closed.
\end{proof}

By Proposition \ref{prop:relevant+tempered} \eqref{itm:relevant+tempered 3}, any tempered Harish-Chandra sheaf is cohomologically induced from a clean local system on the open orbit for a split Levi (i.e., one for which every root is real). In particular, the supporting orbit $Q$ fibers as the open orbit for a split group over a closed orbit in a partial flag variety. The next step is to analyze the split case in detail.

\subsection{Tempered Hodge modules for split groups} \label{subsec:split tempered}

Let us now suppose that $G^\sigma$ is split modulo its center, $\lambda \in \mf{h}^*_\mb{R}$ is dominant and $j_{!*}\gamma$ is an irreducible tempered Harish-Chandra sheaf with infinitesimal character $\lambda$ supported on the open $K$-orbit $Q \subset \mc{B}$. By splitness, every root $\alpha \in \Phi$ is real and hence satisfies $\langle \lambda, \check\alpha \rangle = 0$, and by Proposition \ref{prop:relevant+tempered} \eqref{itm:relevant+tempered 3}, the local system $\gamma$ is clean. The aim of this subsection is to give an explicit formula for the $\mc{D}_\lambda$-module $j_{!*}\gamma$ and its Hodge filtration.

The main example to keep in mind is when $\Gamma(j_{!*}\gamma)$ is the tempered spherical principal series representation of $G_\mb{R}$ with real infinitesimal character. In this case, $\lambda = 0$ and $\gamma \cong \mc{O}_Q$ as a $K$-equivariant vector bundle. When $G_\mb{R} = G^\sigma$ is linear, this is the \emph{only} example (up to tensoring with a character of $G_\mb{R}$), but for non-linear groups there will generally be others.

Let us fix $x \in Q$ and write $\eta = \gamma_x$ for the fiber of $\gamma$ at $x$. This is an irreducible representation of $\mrm{Stab}_K(x)$. We have the following.

\begin{prop} \label{prop:non-linear lowest hodge}
The lowest piece $F_0 j_{!*}\gamma$ of the Hodge filtration is a trivial vector bundle on $\mc{B}$. In particular,
\[ \Gamma(\mc{B}, F_0 j_{!*}\gamma) \cong \eta \]
is an irreducible representation of $K$.
\end{prop}
\begin{proof}
We first claim that, for any simple root $\alpha$, the restriction of $F_0j_{!*}\gamma$ to any fiber $C$ of the associated $\mb{P}^1$-fibration $\pi_\alpha \colon \mc{B} \to \mc{P}_\alpha$ is a trivial vector bundle. To see this, first note that by equivariance, it is enough to check this for the fiber containing $x$. Since $\alpha$ is real, we have a root homomorphism
\[ \phi_\alpha \colon \mrm{SL}_2 \to G \]
such that $\theta\phi_\alpha(g) = \phi_\alpha((g^t)^{-1})$ and $\phi_\alpha(B) \subset B_x$, where $B \subset \mrm{SL}_2$ is the subgroup of lower triangular matrices and $B_x = \mrm{Stab}_G(x)$. The curve $C$ is the image of the induced map
\[ \phi_\alpha \colon \mb{P}^1 = \mrm{SL}_2/B \to G/B_x = \mc{B}.\]
Now, the morphism $\phi_\alpha$ is non-characteristic for $j_{!*}\gamma$, so
\[ \phi_\alpha^{\bigcdot} F_0 j_{!*}\gamma = F_0 j_{\alpha, !*}\phi_\alpha^\circ \gamma,\]
where $j_\alpha$ is the inclusion of the open $\mrm{SO}_2$-orbit in $\mb{P}^1$, which we may identify with $\mb{C}^\times$ using an appropriate coordinate on $\mb{P}^1$, and $\phi_\alpha^\circ \gamma$ is the restriction of $\gamma$ to this $\mb{C}^\times$. Since $j_{!*}\gamma$ is tempered, $\phi_\alpha^\circ \gamma$ must be a $(-1)$-twisted local system on $\mb{C}^\times$ extending cleanly to $\mb{P}^1$. Thus, we may write
\[ j_{\alpha, !*}\phi_\alpha^\circ \gamma = \bigoplus_k j_{\alpha, !*}\mc{O}_{\mb{C}^\times}z^{\mu_k} \otimes \mc{O}(-1),\]
for some $\mu_k \in (0, 1)$, where $\mc{O}_{\mb{C}^\times}z^{\mu_k}$ is the rank $1$ local system on $\mb{C}^\times$ with monodromy $e^{-2\pi i \mu_k}$. By direct calculation, we have
\[ F_0 j_{\alpha !*}\mc{O}_{\mb{C}^\times}z^{\mu_k} \cong \mc{O}(1),\]
so $F_0 j_{\alpha, !*}\phi_\alpha^\circ \gamma$ is a trivial vector bundle on $\mb{P}^1$ as claimed.

Now, since the $\mc{D}_0$-module $j_{!*}\gamma$ is globally generated, the lowest piece $F_0 j_{!*}\gamma$ of the Hodge filtration is globally generated as an $\mc{O}_\mc{B}$-module by Theorem \ref{thm:hodge generation}. So we have a surjection
\begin{equation} \label{eq:non-linear lowest hodge 1}
\Gamma(\mc{B}, F_0 j_{!*}\gamma) \otimes \mc{O}_\mc{B} \to F_0 j_{!*}\gamma.
\end{equation}
We will show that \eqref{eq:non-linear lowest hodge 1} is also injective, and hence that $F_0 j_{!*}\gamma$ is a trivial vector bundle as claimed. The kernel $\mc{K}$ of \eqref{eq:non-linear lowest hodge 1} is a torsion-free sheaf on $\mc{B}$ that is a vector sub-bundle of $\Gamma(\mc{B}, F_0 j_{!*}\gamma) \otimes \mc{O}_\mc{B}$ when restricted to the open set $Q$. It is therefore enough to show that the fiber
\[ \mc{K}_x = \ker( \Gamma(\mc{B}, F_0 j_{!*}\gamma) \to \gamma_x) \]
at our chosen point $x \in Q$ is zero. To see this, suppose that $v \in \mc{K}_x \subset \Gamma(\mc{B}, F_0 j_{!*}\gamma)$ and let $Z \subset Q$ denote the vanishing locus of $v$ as a section of the vector bundle $\gamma$. Since $F_0 j_{!*}\gamma$ is trivial along each simple root $\mb{P}^1$, $Z$ satisfies
\[ Z = \pi_\alpha^{-1}\pi_\alpha(Z) \cap Q \]
for any simple root $\alpha$. But the only non-empty subset of $Q$ with this property is $Q$ itself, so $v$ vanishes on all of $Q$. Since $F_0j_{!*}\gamma$ is torsion-free, it follows that $v = 0$.
\end{proof}

\begin{rmk}
When $\Gamma(j_{!*}\gamma)$ is the tempered spherical principal series, $\eta = \mb{C}_0$ is the trivial representation of $K$; one can see this either by direct calculation or by noting that $\Gamma(j_{!*}\gamma)$ is the restriction of a representation of the linear group $G^\sigma$ and appealing to \cite[Theorem 4.5]{DV1}.
\end{rmk}

Now, let us write $U_\lambda(\mf{g}) = U(\mf{g}) \otimes_{Z(U(\mf{g})), \chi_\lambda} \mb{C}$ for the quotient of $U(\mf{g})$ through which $Z(U(\mf{g}))$ acts with infinitesimal character $\chi_\lambda$. The main result of this subsection is:

\begin{thm} \label{thm:split tempered}
We have the following.
\begin{enumerate}
\item There are filtered isomorphisms
\[ j_{!*}\gamma \cong \mc{D}_\lambda \otimes_{U(\mf{k})} \eta \quad \text{and} \quad \Gamma(j_{!*}\gamma) \cong U_\lambda(\mf{g}) \otimes_{U(\mf{k})} \eta,\]
identifying the Hodge filtration on $j_{!*}\gamma$ (resp., $\Gamma(j_{!*}\gamma)$) with the filtration induced by the order filtration on $\mc{D}_\lambda$ (resp., the PBW filtration on $U_\lambda(\mf{g})$).
\item The associated gradeds are
\[ \Gr^F j_{!*}\gamma \cong \mc{O}_{\tilde{\mc{N}}_K^*} \otimes \eta \quad \text{and} \quad \Gr^F \Gamma(j_{!*}\gamma) \cong \mc{O}_{\mc{N}_K^*} \otimes \eta,\]
where $\mc{N}_K^* = \mc{N}^* \cap (\mf{g}/\mf{k})^* \subset \mf{g}^*$ is the $K$-nilpotent cone (with reduced scheme structure) and $\tilde{\mc{N}}_K^* \subset T^*\mc{B}$ is its scheme-theoretic pre-image under the Springer map $\mu \colon T^*\mc{B} \to \mc{N}^* \subset \mf{g}^*$.
\end{enumerate}
\end{thm}

The proof of Theorem \ref{thm:split tempered}, which follows similar lines to Theorem \ref{thm:xi hodge filtration}, occupies the rest of the subsection.

Let us write $\mc{M} = \mc{D}_\lambda \otimes_{U(\mf{k})} \eta$ and $F_\bullet \mc{M}$ for the filtration generated by $\eta$ as in the statement of Theorem \ref{thm:split tempered}. We also write
\[ U_\lambda(\mf{k}) = U(\mf{k}) \otimes_{U(\mf{k} \cap \mf{z}(\mf{g}))} \mb{C}_{\chi_\lambda}.\]
Note that $U_\lambda(\mf{k}) \cong U(\mf{k}^0)$, where $\mf{k}^0 = \mf{k}/(\mf{k} \cap \mf{z}(\mf{g}))$. The action of $U(\mf{k})$ on $\eta$ factors through $U_\lambda(\mf{k})$.

\begin{lem} \label{lem:split tempered vanishing}
We have
\[ \mc{M} = (\mc{D}_\lambda, F_\bullet) \overset{\mrm{L}}\otimes_{(U_\lambda(\mf{k}), F_\bullet)} \eta \]
in the filtered derived category.
\end{lem}
\begin{proof}
We need to show that
\begin{equation} \label{eq:split tempered vanishing 1}
 0 = \mc{H}^i\Gr^F ((\mc{D}_\lambda, F_\bullet) \overset{\mrm{L}}\otimes_{(U_\lambda(\mf{k}), F_\bullet)} \eta) = \mc{H}^i(\mrm{L}\mu_{K^0}^{\bigcdot} \mc{O}_0 \otimes \eta)
\end{equation}
for $i \neq 0$, where $\mu_{K^0} \colon T^*\mc{B} \to (\mf{k}^0)^*$ is the moment map for $K^0 = K/(K \times_G Z(G))$. Now, the fiber $\mu_{K^0}^{-1}(0) \subset T^*\mc{B}$ is the union of conormal bundles to $K$-orbits, so
\[ \codim_{T^*\mc{B}} \mu_{K^0}^{-1}(0) = \dim \mc{B}.\]
Since $G^\sigma$ is split modulo its center, we have $\dim \mc{B} = \dim \mf{k}^0$. So $\mu_{K^0}$ must be flat in a neighbourhood of the origin (and hence flat everywhere by $\mb{C}^\times$-equivariance). We deduce that \eqref{eq:split tempered vanishing 1} holds.
\end{proof}

\begin{lem} \label{lem:split tempered gr}
We have
\[ \Gr^F \mc{M} \cong \mc{O}_{\tilde{\mc{N}}_K^*} \otimes \eta \quad \text{and} \quad \Gr^F \Gamma(\mc{M}) \cong \mc{O}_{\mc{N}_K^*} \otimes \eta.\]
\end{lem}
\begin{proof}
Let us write $\mf{g}_{\mathit{ad}} = \mf{g}/\mf{z}(\mf{g})$. Since $G^\sigma$ is split modulo its center, a Cartan subspace $\mf{a}^*$ for $(\mf{g}_{\mathit{ad}}/\mf{k}^0)^*$ is also a Cartan subspace for $\mf{g}_{\mathit{ad}}^*$ and the little Weyl group $N_K(\mf{a}^*)/Z_K(\mf{a}^*)$ is the full Weyl group of $G$. Thus, by Chevalley, the restriction map
\[ \mb{C}[\mf{g}_{\mathit{ad}}^*]^G \to \mb{C}[(\mf{g}_{\mathit{ad}}/\mf{k}^0)^*]^K \]
is an isomorphism. Consider the diagram
\[
\begin{tikzcd}
(\mf{g}_{\mathit{ad}}/\mf{k}^0)^* \ar[r] \ar[d] & \mf{g}_{\mathit{ad}}^* \ar[d] \\
\spec \mb{C}[(\mf{g}_{\mathit{ad}}/\mf{k}^0)^*]^K \ar[r, equal] & \spec \mb{C}[\mf{g}_{\mathit{ad}}^*]^G.
\end{tikzcd}
\]
By Kostant \cite[Theorems 0.1 and 0.2]{K} and Kostant-Rallis \cite[Theorems 14 and 15]{KR}, the vertical maps are flat and their scheme-theoretic fibers $\mc{N}^*$ and $\mc{N}^*_K$ are reduced. Thus
\[ \mc{O}_{\mc{N}_K^*} = \mc{O}_{\mc{N}^*} \overset{\mrm{L}}\otimes_{\mc{O}_{\mf{g}_{\mathit{ad}}^*}} \mc{O}_{(\mf{g}_{\mathit{ad}}/\mf{k}^0)^*} = \mc{O}_{\mc{N}^*} \overset{\mrm{L}}\otimes_{\mc{O}_{(\mf{k}^0)^*}} \mc{O}_0.\]
We thus have, by Lemma \ref{lem:split tempered vanishing},
\[ \Gr^F \mc{M} = \mrm{L}\mu_{K^0}^{\bigcdot}(\mc{O}_0 \otimes \eta) = \mrm{L}\mu^{\bigcdot}\mc{O}_{\mc{N}^*_K} \otimes \eta = \mc{O}_{\tilde{\mc{N}}_K^*} \otimes \eta.\]
So
\[ \mrm{R}\Gamma(\Gr^F\mc{M}) = \mrm{R}\mu_{\bigcdot}\mrm{L}\mu^{\bigcdot}\mc{O}_{\mc{N}^*_K} \otimes \eta = \mc{O}_{\mc{N}^*_K} \otimes \eta,\]
since $\mrm{R}\mu_{\bigcdot}\mc{O}_{T^*\mc{B}} = \mc{O}_{\mc{N}^*}$, and hence, since the higher cohomology vanishes,
\[ \Gr^F \Gamma(\mc{M}) = \mc{O}_{\mc{N}^*_K} \otimes \eta\]
as claimed.
\end{proof}

Recall the filtered duality for $\mc{D}_\lambda$-modules. In our context, this takes the form
\[ \mb{D}(\mc{P}, F_\bullet) := \mrm{R}\shom_{(\mc{D}_\lambda, F_\bullet)}((\mc{P}, F_\bullet), (\mc{D}_\lambda, F_\bullet))\{-2\dim \mc{B}\}[\dim \mc{B}],\]
for $(\mc{P}, F_\bullet)$ a filtered $\mc{D}_\lambda$-module. The dual $\mb{D}(\mc{P}, F_\bullet)$ is an object in the filtered derived category of modules over $\mc{D}_\lambda^{\mathit{op}} = \mc{D}_{-\lambda}$. Note that since $\langle \lambda, \check\alpha \rangle = 0$ for all roots $\alpha \in \Phi$, both $\lambda$ and $-\lambda$ are dominant.

\begin{lem} \label{lem:split tempered duality}
We have
\[ \mb{D}(\mc{M}, F_\bullet) \cong (\mc{D}_{-\lambda}, F_\bullet) \otimes_{(U(\mf{k}), F_\bullet)} \eta^*\{- \dim \mc{B}\} \]
in the filtered derived category of $\mc{D}_{-\lambda}$-modules.
\end{lem}
\begin{proof}
Consider the filtered Koszul resolution
\[ U_\lambda(\mf{k}) \otimes \wedge^n \mf{k}^0 \otimes \eta \{-n\} \to U_\lambda(\mf{k}) \otimes \wedge^{n - 1}\mf{k}^0 \otimes \eta \{-n + 1\} \to \cdots \to U_\lambda(\mf{k}) \otimes \eta \]
of $\eta$ as a filtered $U_\lambda(\mf{k}) \cong U(\mf{k}^0)$-module, where $n = \dim \mf{k}^0 = \dim \mc{B}$. By Lemma \ref{lem:split tempered vanishing}, we have a corresponding Koszul resolution
\begin{equation} \label{eq:split tempered duality 2}
\mc{D}_\lambda \otimes \wedge^n \mf{k}^0 \otimes \eta \{-n\} \to \mc{D}_\lambda \otimes \wedge^{n - 1} \mf{k}^0 \otimes \eta \{-n + 1\} \to \cdots \to \mc{D}_\lambda \otimes \eta
\end{equation}
for $\mc{M}$. Taking the filtered dual, we obtain the complex
\begin{equation} \label{eq:split tempered duality 3}
\mc{D}_{-\lambda} \otimes \eta^*\{-2n\} \to \cdots \to \mc{D}_{-\lambda} \otimes \wedge^{n - 1} (\mf{k}^0)^* \otimes \eta^* \{-n - 1\} \to \mc{D}_{-\lambda} \otimes \wedge^n (\mf{k}^0)^* \otimes \eta^* \{-n\}.
\end{equation}
But \eqref{eq:split tempered duality 3} is precisely the Koszul resolution \eqref{eq:split tempered duality 2} for $(\mc{D}_{-\lambda}, F_\bullet) \otimes_{(U(\mf{k}), F_\bullet)} \eta^*$ tensored with $\wedge^n (\mf{k}^0)^* \{-n\}$, so we deduce the lemma.
\end{proof}

Now, since $\eta = \Gamma(F_0 j_{!*}\gamma)$, we have a tautological filtered morphism $(\mc{M}, F_\bullet) \to (j_{!*}\gamma, F_\bullet)$.

\begin{lem} \label{lem:split tempered D-module}
The morphism
\[ \mc{M} \to j_{!*}\gamma \]
is an isomorphism at the level of (unfiltered) $\mc{D}_\lambda$-modules.
\end{lem}
\begin{proof}
Since $j_{!*}\gamma$ is irreducible, the non-zero morphism $\mc{M} \to j_{!*}\gamma$ is surjective, and hence its restriction $j^!\mc{M} \to \gamma$ to the open orbit is surjective. Now, since the local system $j_{!*}\gamma$ is clean, we have $j_{!*}\gamma = j_!\gamma$, so
\[ \hom(j_{!*}\gamma, \mc{M}) = \hom(\gamma, j^!\mc{M}) \to \hom(\gamma, \gamma) = \hom(j_{!*}\gamma, j_{!*}\gamma) \]
is also surjective. So $j_{!*}\gamma$ is a summand of $\mc{M}$. We will show that $\mrm{End}(\mc{M}) = \mb{C}$ and hence that the only non-zero summand is $\mc{M}$ itself.

Consider first the case where $\Gamma(j_{!*}\gamma)$ is the tempered spherical principal series, i.e., when $\eta = \mb{C}_0$ is the trivial $K$-type. Then
\[ \mrm{End}(\mc{M}) = \hom_K(\mb{C}_0, \Gamma(\mc{M})) \cong \hom_K(\mb{C}_0, \Gr^F \Gamma(\mc{M})) = \hom_K(\mb{C}_0, \mc{O}_{\mc{N}_K^*})\]
by Lemma \ref{lem:split tempered gr}. Since $K$ acts on the reduced and connected variety $\mc{N}_K^*$ with finitely many orbits, the only invariant functions are the constants, so the lemma is proved in this case. Note that in this case $\gamma = \mc{O}_Q$, so we also deduce an isomorphism of $K$-representations
\begin{equation} \label{eq:split tempered D-module}
\mc{O}_{\mc{N}^*} = \Gamma(j_{!*}\gamma) = \Gamma(j_*\gamma) = \Gamma(Q, \mc{O}_Q) = \mrm{Ind}_{\mrm{Stab}_K(x)}^K \mb{C}_0.
\end{equation}

Now consider the general case. We now have
\[ \mrm{End}(\mc{M}) = \hom_K(\eta, \mc{O}_{\mc{N}_K^*} \otimes \eta) = \hom_K(\eta, \mrm{Ind}_{\mrm{Stab}_K(x)}^K \mb{C}_0 \otimes \eta) = \mrm{End}_{\mrm{Stab}_K(x)}(\eta)\]
by \eqref{eq:split tempered D-module}. Since $\eta$ is irreducible as a $\mrm{Stab}_K(x)$-module by construction, we conclude that $\mrm{End}(\mc{M}) = \mb{C}$ by Schur's lemma.
\end{proof}

\begin{proof}[Proof of Theorem \ref{thm:split tempered}]
Given the lemmas above, it remains to show that the tautological morphism $(\mc{M}, F_\bullet) \to (j_{!*}\gamma, F_\bullet)$ is a filtered isomorphism. We proceed exactly as in the proof of Theorem \ref{thm:xi hodge filtration}: passing to duals and applying Lemma \ref{lem:split tempered duality}, we have a filtered morphism
\begin{equation} \label{eq:split tempered 1}
(j_{!*}\gamma^*, F_{\bullet + \dim \mc{B}}) = \mb{D}(j_{!*}\gamma, F_\bullet) \to \mb{D}(\mc{M}, F_\bullet) = (\mc{D}_{-\lambda} \otimes_{U(\mf{k})} \eta^*, F_{\bullet + \dim \mc{B}}),
\end{equation}
which is an isomorphism at the level of underlying $\mc{D}_{-\lambda}$-modules by Lemma \ref{lem:split tempered D-module}. The map on global sections of $F_0$ is an injective morphism between irreducible representations of $K$, and hence an isomorphism; that is, we have $\Gamma(F_0 j_{!*}\gamma^*) \cong \eta^*$. Thus, we have a filtered splitting of \eqref{eq:split tempered 1}, so \eqref{eq:split tempered 1} is a filtered isomorphism, and hence so is our original morphism $(\mc{M}, F_\bullet) \to (j_{!*}\gamma, F_\bullet)$.
\end{proof}

\subsection{Proof of Theorem \ref{thm:tempered hodge} (\ref{itm:tempered hodge 1})}
\label{subsec:pf of tempered 1}

Let us now consider a general irreducible tempered Hodge module $\cM = j_{!*}\gamma \in \mhm(\mc{D}_\lambda, K)$, for $\gamma$ a twisted local system on a $K$-orbit $Q$. We will show in this subsection that the Hodge filtration of $\mc{M}$ is generated by its lowest piece.

As in \S\ref{subsec:cleanness}, let $S$ denote the set of real simple roots for $Q$ and $\pi_S \colon \mc{B} \to \mc{P}_S$ the projection to the corresponding partial flag variety. By Proposition \ref{prop:relevant+tempered}, the image $\pi_S(Q)$ is closed, $\pi_S^{-1}\pi_S(Q) = \bar{Q}$ and the local system $\gamma$ is clean.

Consider first the Hodge module $\mc{N} = j'_{!*}\gamma$ on $\bar{Q}$. For any $y \in \pi_S(Q)$, we may identify the fiber $\pi_S^{-1}(y)$ with the flag variety $\mc{B}_L$ of a $\theta$-stable Levi $L$, whose roots are all real. Moreover, the stabilizer $\mrm{Stab}_K(y)$ acts via a surjective homomorphism to $K_L = K \times_G L$. Since the restriction is non-characteristic, we deduce that $\mc{N}|_{\mc{B}_L} = j_{L, !*}\gamma|_{Q \cap \mc{B}_L}$ is a tempered Hodge module for the pair $(L, K_L)$ supported on the open $K_L$-orbit. Since $(L, K_L)$ is split modulo its center, by Theorem \ref{thm:split tempered}, the Hodge filtration of $\mc{N}|_{\mc{B}_L}$, and hence the Hodge filtration of $\mc{N}$ itself, is generated by its lowest piece.

Now write $i \colon \bar{Q} \to \mc{B}$ for the inclusion, so $\mc{M} = i_*\mc{N}$. To show that the Hodge filtration on $\mc{M}$ is generated by its lowest piece as a $\mc{D}_\lambda$-module, it is enough to show that the associated graded $\Gr^F \mc{M}$ is generated by its lowest graded piece as an $\mc{O}_{T^*\mc{B}} = \mrm{Sym}(\mc{T}_\mc{B})$-module. By Proposition \ref{prop:associated graded pushforward}, we have
\[ \Gr^F\mc{M} = \Gr(i)_* \Gr^F\mc{N} = i_{\bigcdot}(\mrm{Sym}(i^{\bigcdot}\mc{T}_{\mc{B}}) \otimes_{\mrm{Sym}(\mc{T}_{\bar{Q}})} \Gr^F\mc{N} \otimes \omega_{\bar{Q}/\mc{B}})\{c\},\]
where $\omega_{\bar{Q}/\mc{B}}$ is the determinant of the normal bundle and $c = \codim Q$. But we have shown above that $\Gr^F \mc{N}$ is generated by its lowest graded piece as a $\mrm{Sym}(\mc{T}_{\bar{Q}})$-module, hence the same is true for $\Gr^F\mc{M}$. We conclude that the Hodge filtration of $\mc{M}$ is generated by its lowest piece $F_c\mc{M}$, proving part \eqref{itm:tempered hodge 1} of Theorem~\ref{thm:tempered hodge}.
 
\subsection{Proof of Theorem \ref{thm:tempered hodge} (\ref{itm:tempered hodge 2})} \label{subsec:pf of tempered 2}

We will proceed by appealing to the known unitarity of tempered Harish-Chandra modules. Let us assume that $\mc{M}$ is equipped with its standard Hodge module structure of weight $\dim Q$ with Hodge filtration starting in degree $c = \codim Q$; all other Hodge module structures are Hodge twists of this one, so if Conjecture \ref{conj:schmid-vilonen} is true for one of these it is true for all of them. Now, since the Harish-Chandra module $V = \Gamma(\mc{M})$ is tempered, it is unitary and, in particular, Hermitian. So by the discussion in \S\ref{subsec:intro harish-chandra}, there exists an isomorphism $\theta \colon \theta^*\mc{M} \cong \mc{M}$, squaring to the identity, and
\[ \langle u, \bar{v} \rangle_{\mf{g}_\mb{R}} =  \Gamma(S)(\theta(u), \bar{v}),\]
is a non-degenerate $(\mf{g}_\mb{R}, K_\mb{R})$-invariant form on $V$, where $S$ is the polarization of $\mc{M}$ as a pure Hodge module. Since $V$ is unitary, the form $\langle\,,\,\rangle_{\mf{g}_\mb{R}}$ must be either positive or negative definite. By Proposition \ref{prop:polarization non-degenerate}, $\Gamma(S)$ is positive definite on $\Gamma(F_c\mc{M})$ so $\theta$ must act here with a single sign. Let us fix our choice of $\theta$ so that this sign is $+1$. In this case, we deduce:

\begin{lem} \label{lem:tempered theta signature}
For $\epsilon = \pm 1$, the polarization $\Gamma(S)$ is $\epsilon$-definite on the $\epsilon$-eigenspace $V^{\epsilon\theta}$.
\end{lem}

Theorem \ref{thm:tempered hodge} \eqref{itm:tempered hodge 2} now follows from Lemma \ref{lem:tempered theta signature} and:

\begin{lem} \label{lem:tempered hodge sign}
The involution $\theta$ acts on $\Gr^F_p V$ by $(-1)^{p + c}$ for all $p$.
\end{lem}

\begin{proof}
Observe that $V$ is cohomologically induced from a tempered representation attached to the open orbit in a Levi that is split modulo its center (by, say, Proposition \ref{prop:relevant+tempered}). So, by Proposition \ref{prop:cohomological induction hodge}, it is enough to prove the lemma in the split setting of \S\ref{subsec:split tempered}. In this case, by Theorem \ref{thm:split tempered}, we have that
\[ \Gr^F V = \mc{O}_{\mc{N}_K^*} \otimes \eta \]
is naturally a graded quotient of $\mrm{Sym}(\mf{g}/\mf{k}) \otimes \eta$ for some irreducible representation $\eta$ of $K$. Since $\theta$ acts on $\mf{g}/\mf{k}$ with eigenvalue $-1$, the result follows.
\end{proof}

\section{Proof of the unitarity criterion: the general case} \label{sec:pf of unitarity}

In this section, we give the proof of Theorem \ref{thm:unitarity criterion}, the unitarity criterion for general Harish-Chandra modules. The proof relies on most of the previous results of the paper: Theorems \ref{thm:semi-continuity} and \ref{thm:jantzen} on the behavior of Hodge filtrations and polarizations under deformation, Theorem \ref{thm:filtered exactness} on cohomology vanishing for the Hodge filtration, and Theorem \ref{thm:tempered hodge}, the case of tempered Harish-Chandra modules (which in turn relied on Theorem \ref{thm:hodge generation} on global generation of the Hodge filtration).

Given these ingredients, the proof more or less follows the sketch outlined by Adams, Trapa and Vogan in the manuscript \cite{ATV}. First, in \S\ref{subsec:hodge and signature}, we reduce Theorem \ref{thm:unitarity criterion} to a numerical statement about Hodge and signature characters of Harish-Chandra modules (Theorem \ref{thm:hodge and signature K'}), a weak form of Conjecture \ref{conj:schmid-vilonen}. We prove this numerical statement in \S\S\ref{subsec:pf of signature K}--\ref{subsec:pf of signature K'} by running a version of the \cite{ALTV} algorithm to reduce to the case of tempered representations. In order to make the argument work, we need to know the cones of positive deformations of \S\ref{sec:deformations} in the case of Harish-Chandra sheaves: we compute these explicitly in \S\ref{subsec:hc deformations}.

\subsection{Hodge and signature characters} \label{subsec:hodge and signature}

Theorem \ref{thm:unitarity criterion} is a consequence of a weak version of Conjecture \ref{conj:schmid-vilonen}, which we state in this subsection.

Fix $\lambda \in \mf{h}^*_\mb{R}$ dominant and $(\mc{M}, S)$ a polarized Hodge module in $\mhm(\mc{D}_{\lambda}, K)$. We wish to relate the signature of the $\mf{u}_\mb{R}$-invariant form $\Gamma(S)$ on the Harish-Chandra module $\Gamma(\mc{M})$ to the Hodge structure on $\mc{M}$. To this end, we introduce the following numerical invariants.

First, we write
\[ \chi^{\mathit{sig}}(\mc{M}, \zeta) = \sum_{\mu \in \widehat{K}} (m^+(\mc{M}, \mu) + \zeta m^-(\mc{M}, \mu))[\mu] \in \widehat{\mrm{Rep}}(K)[\zeta]/(\zeta^2 - 1),\]
where the sum is over isomorphism classes of irreducible representations $\mu$ of $K$, $(m^+(\mc{M}, \mu), m^-(\mc{M}, \mu))$ is the signature of the non-degenerate Hermitian form $\Gamma(S)$ on the (finite dimensional) multiplicity space $\hom_K(\mu, \Gamma(\mc{M}))$ and
\[ \widehat{\mrm{Rep}}(K) = \prod_{\mu \in \widehat{K}}\mb{Z}\cdot [\mu] \]
is the completion of the Grothendieck group $\mrm{Rep}(K)$ of finite dimensional $K$-modules. We will call the quantity $\chi^{\mathit{sig}}(\mc{M}, \zeta)$ the \emph{signature $K$-character}; it keeps track of numerical information about the $\mf{u}_\mb{R}$-invariant form $\Gamma(S)$.

We also write
\[ \chi^H(\mc{M}, u) = \sum_{p \in \mb{Z}} [\Gr^F_p\Gamma(\mc{M})]u^p \in \mrm{Rep}(K)((u)).\]
We call $\chi^H(\mc{M}, u)$ the \emph{Hodge $K$-character}; it keeps track of numerical information about the Hodge filtration on $\Gamma(\mc{M})$. Note that
\[ \Gr^F_p \Gamma(\mc{M}) = \Gamma(\Gr^F_p\mc{M}) \]
by Theorem \ref{thm:filtered exactness}.

Since $\Gamma(\mc{M})$ is an admissible $(\mf{g}, K)$-module, the Hodge $K$-character lies in the subgroup
\[ \mrm{Rep}(K)((u))^{\mathit{adm}} \subset \mrm{Rep}(K)((u))\]
consisting of formal Laurent series
\[ a_n u^n + a_{n + 1}u^{n + 1} + \cdots, \quad a_i \in \mrm{Rep}(K)\]
such that the class $[\mu]$ of each irreducible $K$-module $\mu$ appears in only finitely many of the $a_i$. There is a well-defined reduction map
\begin{align*}
\mrm{Rep}(K)((u))^{\mathit{adm}} &\to \widehat{\mrm{Rep}}(K)[\zeta]/(\zeta^2 - 1) \\
\chi(u) &\mapsto \chi(\zeta) \mod \zeta^2 - 1.
\end{align*}

We prove the following result in \S\ref{subsec:pf of signature K}.

\begin{thm} \label{thm:hodge and signature K}
Assume $\lambda \in \mf{h}^*_\mb{R}$ is dominant, and let $(\mc{M}, S)$ be a polarized object of weight $w$ in $\mhm(\mc{D}_\lambda, K)$. Then
\[ \chi^{\mathit{sig}}(\mc{M}, \zeta) = \zeta^c\chi^H(\mc{M}, \zeta) \mod \zeta^2 - 1,\]
where $c = \dim \mc{B} - w$.
\end{thm}

Theorem \ref{thm:hodge and signature K} is an obvious consequence of Conjecture \ref{conj:schmid-vilonen}, which we regard as strong evidence for the full conjecture. Theorem \ref{thm:tempered hodge} implies that Theorem \ref{thm:hodge and signature K} holds when the Harish-Chandra sheaf $\mc{M}$ is tempered.

Theorem \ref{thm:hodge and signature K} states that the Hodge filtration controls the signature of the $(\mf{u}_\mb{R}, K_\mb{R})$-invariant Hermitian form on $\Gamma(\mc{M})$. In order to deduce the signature of the $(\mf{g}_\mb{R}, K_\mb{R})$-invariant form (when it exists), we need the following mild refinement. Recall that the Harish-Chandra module $V = \Gamma(\mc{M})$ carries a non-degenerate $(\mf{g}_\mb{R}, K_\mb{R})$-invariant Hermitian form if and only if $\lambda = \delta \lambda$ and $\mc{M} \cong \theta^*\mc{M}$ and that, given an isomorphism $\theta \colon \theta^*\mc{M} \cong \mc{M}$ squaring to the identity,
\[ \langle u, \bar{v} \rangle_{\mf{g}_\mb{R}} = \Gamma(S)(\theta(u), \bar{v})\]
is such a form. We thus need a version of Theorem \ref{thm:hodge and signature K} that keeps track of the action of $\theta$ on $\mc{M}$.

We can package this neatly using the trick of passing to the extended group (cf., \cite[Chapter 12]{ALTV}). Let $K' = K \times \{1, \theta\}$; we have compatible actions of $K'$ on $\mc{B}$, $\tilde{\mc{B}}$ and $H$. The choice of involution on $\mc{M}$ upgrades it to an object in $\mhm(\mc{D}_\lambda, K')$. More generally, we can consider polarized objects in the larger monodromic category
\[ \mc{M} \in \mhm(\tilde{\mc{D}}, K').\]
We will say that $\mc{M}$ is \emph{dominantly twisted} if
\[ \mc{M} = \bigoplus_{\lambda} \mc{M}_\lambda, \quad \mc{M}_\lambda \in \mhm(\mc{D}_\lambda)\]
with each $\lambda \in \mf{h}^*_\mb{R}$ dominant. In this setting, we may define signature and Hodge $K'$-characters
\[ \chi^{\mathit{sig}}(\mc{M}, \zeta) \in \widehat{\mrm{Rep}}(K')[\zeta]/(\zeta^2 - 1) \quad \mbox{and} \quad \chi^H(\mc{M}, u) \in \mrm{Rep}(K')((u))^{\mathit{adm}}\]
just as before. The $K'$-equivariant version of Theorem \ref{thm:hodge and signature K} also holds.

\begin{thm} \label{thm:hodge and signature K'}
Let $\mc{M}$ be an irreducible dominantly twisted object in $\mhm(\tilde{\mc{D}}, K')$ with polarization $S$. If $\mc{M}$ is pure of weight $w$, then
\[ \chi^{\mathit{sig}}(\mc{M}, \zeta) = \zeta^c\chi^H(\mc{M}, \zeta) \mod \zeta^2 - 1,\]
where $c = \dim \mc{B} - w$.
\end{thm}

We explain how to extend the proof of Theorem \ref{thm:hodge and signature K} to the $K'$-equivariant setting in \S\ref{subsec:pf of signature K'}. Theorem \ref{thm:hodge and signature K'} implies the unitarity criterion:

\begin{proof}[Proof of Theorem \ref{thm:unitarity criterion} modulo Theorem \ref{thm:hodge and signature K'}]
The irreducible Harish-Chandra module $V$ is unitary if and only if the $\mf{g}_\mb{R}$-invariant form
\begin{equation} \label{eq:unitarity criterion 1}
\langle u, \bar{v}\rangle_{\mf{g}_\mb{R}} = \Gamma(S)(\theta(u), \bar{v})
\end{equation}
is either positive or negative definite. This form is positive definite if and only if, for each $K'$-type $\mu$, the form $\Gamma(S)$ is $\epsilon(\mu)$-definite on the multiplicity space $\hom_{K'}(\mu, V)$, where $\epsilon(\mu)$ is the eigenvalue of $\theta \in Z(K')$ acting on $\mu$. But by Theorem \ref{thm:hodge and signature K'}, this holds if and only if $\Gr^F_p\hom_{K'}(\mu, V) \neq 0$ implies $\epsilon(\mu) = (-1)^{p + c}$ for all $p$, which in turn holds if and only if $\theta$ acts on $\Gr^F_p V$ with eigenvalue $(-1)^{p + c}$ for all $p$. Similarly, the $\mf{g}_\mb{R}$-invariant form \eqref{eq:unitarity criterion 1} is negative definite if and only if $\theta$ acts on $\Gr^F_p V$ with eigenvalue $(-1)^{p + c + 1}$ for all $p$, so this proves the theorem.
\end{proof}

\subsection{Positive deformations for $K$-orbits} \label{subsec:hc deformations}

We now specialize the theory of \S\ref{sec:deformations} to the setting of certain unions of $K$-orbits and write down the deformation spaces and cones of positive elements explicitly.

We consider the following locally closed subvarieties of the flag variety $\mc{B}$ (cf., Proposition \ref{prop:relevant+tempered}). Fix a $K$-orbit $Q \subset \mc{B}$ and a set $S$ of simple roots. This defines a partial flag variety $\mc{P}_S$ and a fibration $\pi_S \colon \mc{B} \to \mc{P}_S$ as in \S\ref{subsec:S-stalk}. Define
\[ Q_S := \pi_S^{-1}\pi_S(Q).\]
We note that $Q_S$ is $K$-stable. If $\theta(Q) = Q$ and $\delta(S) = S$, then $Q_S$ is also stable under $K'$.

In what follows, we write $\Phi_S = \Phi \cap \mrm{span}(S)$ for the root system spanned by $S$,
\[ \mf{h}^*_{S, \mb{R}}  = \{\mu \in \mf{h}^*_\mb{R} \mid \langle \mu, \check\alpha \rangle = 0 \text{ for } \alpha \in S\} \]
and
\[ \mf{h}^*_{S, \mb{R}, +} = \{\mu \in \mf{h}^*_\mb{R} \mid \langle \mu, \check\alpha \rangle > 0 \text{ for }\alpha \in \Phi_+ - \Phi_S\}\]
for the subspace (resp., cone) in $\mf{h}^*_\mb{R} = \mrm{Pic}^G(\mc{B})$ spanned by the pullbacks of line bundles (resp., ample line bundles) on $\mc{P}_S$ to $\mc{B}$. We also write $\tilde Q$ and $\tilde{Q}_S$ for the pre-images of $Q$ and $Q_S$ in $\tilde{\mc{B}}$.

\begin{prop} \label{prop:hc deformations}
Assume that $\Phi_S$ stable under $\theta_Q$. Then $Q_S$ is affinely embedded in $\mc{B}$, and:
\begin{enumerate}
\item \label{itm:hc deformations 1} The map
\[ \varphi \colon \Gamma_\mb{R}^K(\tilde{Q}_S)^\mon \to \mf{h}^*_\mb{R} \]
factors through an isomorphism
\[ \Gamma_\mb{R}^K(\tilde{Q}_S)^\mon \cong (\mf{h}^*_{S, \mb{R}})^{-\theta_Q}\]
onto the $(-1)$-eigenspace of $\theta_Q$ acting on $\mf{h}^*_{S, \mb{R}} \subset \mf{h}^*_\mb{R}$.
\item \label{itm:hc deformations 2} Under the above identification, we have
\[ \Gamma_\mb{R}^K(\tilde{Q}_S)^\mon_+ \supset (1 - \theta_Q)\mf{h}^*_{S, \mb{R}, +}. \]
\item \label{itm:hc deformations 3} If $Q_S$ is $K'$-stable, then
\[ \Gamma_\mb{R}^{K'}(\tilde{Q}_S)^\mon = (\mf{h}^*_{S, \mb{R}})^{-\theta_Q} \cap (\mf{h}^*_\mb{R})^\delta.\]
\end{enumerate}
\end{prop}

In fact, one can easily check that the containment in \eqref{itm:hc deformations 2} is an equality, but we will not use this fact.

\begin{proof}
We first prove \eqref{itm:hc deformations 1} and \eqref{itm:hc deformations 3} when $S = \emptyset$; in this case, $Q_S = Q$ is a single $K$-orbit. Fix $x \in Q$. Then we may identify the fiber of $\tilde{\mc{B}} \to \mc{B}$ over $x$ with $H$, and
\[ \Gamma(\tilde{Q}, \mc{O}_{\tilde{Q}}^\times)^K \cong \Gamma(H, \mc{O}_H^\times)^{\mrm{Stab}_K(x)}. \]
By construction, $\mrm{Stab}_K(x)$ acts on $H$ via a homomorphism
\[ \tau_x \colon \mrm{Stab}_K(x) \to H^{\theta_Q},\]
which is surjective over the identity component. Let us write $H' \subset H$ for the image. Then
\[ \Gamma(\tilde{Q}, \mc{O}_{\tilde{Q}}^\times)^K \cong \Gamma(H, \mc{O}_H^\times)^{H'} = \Gamma(H/H', \mc{O}^\times).\]
Now for $\mu \in \mb{X}^*(H)$, $\Gamma(Q, \mc{O}^\times)_\mu^K$ is identified with the subset
\[ \Gamma(Q, \mc{O}^\times)_\mu^K = \begin{cases} \mb{C}^\times \cdot \mu, &\mbox{if $\mu \in \mb{X}^*(H/H')$,} \\ 0, &\mbox{otherwise.} \end{cases} \]
So we deduce that $\varphi$ defines an isomorphism
\[ \Gamma_\mb{R}^K(\tilde Q)^\mon \overset{\sim}\to \mb{X}^*(H/H') \otimes \mb{R} = (\mf{h}^*_\mb{R})^{-\theta_Q} \subset \mf{h}^*_\mb{R},\]
proving \eqref{itm:hc deformations 1} for $S = \emptyset$. We also deduce \eqref{itm:hc deformations 3} by observing that
\[ \Gamma(Q, \mc{O}^\times)_\mu^{K'} = \Gamma(Q, \mc{O}^\times)_\mu^{K} \]
as long as $\mu \in 2\mb{X}^*(H/H')^\delta$.

Now consider \eqref{itm:hc deformations 1} and \eqref{itm:hc deformations 3} when $S \neq \emptyset$. We may assume without loss of generality that $Q \subset Q_S$ is open; the claims now follow from the fact that
\[ f \in \Gamma(Q, \mc{O})_\mu = \mrm{H}^0(Q, \mc{O}(\mu))\]
extends to a non-vanishing section of $\mc{O}(\mu)$ on $Q_S$ if and only if $\mc{O}(\mu)$ descends to $\mc{P}_S$.

We now prove that $Q_S$ is affinely embedded, and \eqref{itm:hc deformations 2}. Following \cite[Lemma 3.5.3]{BB2}, we will show that $Q_S$ has a boundary equation in $\mrm{H}^0(\bar{Q}_S, \mc{O}(\mu - \theta_Q\mu))^K$ for any $\mu \in \mb{X}^*(H) \cap \mf{h}^*_{S, \mb{R}, +}$. This implies the result by Proposition \ref{prop:positivity}.

To construct the boundary equation, consider the map
\[ p \colon \mc{B} \xrightarrow{(1, \theta)} \mc{B} \times \mc{B} \xrightarrow{(\pi_S, \pi_{\delta(S)})} \mc{P}_S \times \mc{P}_{\delta(S)}.\]
Fix $x \in Q \subset \mc{B}$ corresponding to a Borel $B = B_x$, let $P = P_{\pi_S(x)}$ be the corresponding parabolic and let
\[ X = G \cdot (P, \theta(P)) \subset \mc{P}_S \times \mc{P}_{\delta(S)}\]
be the corresponding $G$-orbit. Then, by the argument of \cite[Lemma 3.5.3]{BB2},
\[ Q_S \subset p^{-1}(X)\]
is a connected component. Let us write $\mc{P}_S = G/P$ and $\mc{P}_{\delta(S)} = G/\theta(P)$. Then, since $\theta_Q$ preserves $\Phi_S$, $P \cap \theta(P)$ contains a Levi subgroup of both $P$ and $\theta(P)$. Now, $\mu \in \mb{X}^*(H) \cap \mf{h}^*_{S, \mb{R}, +}$ defines a character of $P$, and hence also of $\theta(P)$ via this common Levi; note that $\mu$ is dominant for $P$ but not necessarily for $\theta(P)$. We write $\mc{O}(\mu, -\mu)$ for the corresponding line bundle on $G/P \times G/\theta(P)$. We claim that $X$ has a $G$-invariant boundary equation in $\mrm{H}^0(\bar{X}, \mc{O}(\mu, -\mu))$; pulling back along $p$, this gives the desired boundary equation for $\bar{Q}_S$.

Now, the claim is equivalent to the assertion that the $\theta(P)$-orbit $X_0$ of the base point in $G/P$ has a boundary equation in $\mrm{H}^0(\bar{X}_0, \mc{O}(\mu))$ on which $\theta(P)$ acts with character $\mu$. Since $P$ and $\theta(P)$ contain a common Levi, this orbit is also a $\theta(N)$-orbit (where $N \subset B$ is the unipotent radical), so we may apply the argument of \cite[Lemma 3.5.1]{BB2} to conclude that there is such a boundary equation, given by the highest weight vector in $\mrm{H}^0(G/P, \mc{O}(\mu))$ restricted to the orbit.
\end{proof}

\subsection{Proof in the $K$-equivariant case} \label{subsec:pf of signature K}

We now give the proof of Theorem \ref{thm:hodge and signature K}, i.e., that
\begin{equation} \label{eq:hodge and signature pf 1}
\chi^{\mathit{sig}}(\mc{M}, \zeta) = \zeta^{c(\mc{M})} \chi^{H}(\mc{M}, \zeta) \mod \zeta^2 - 1,
\end{equation}
for $\lambda \in \mf{h}^*_\mb{R}$ dominant, $\mc{M} \in \mhm(\mc{D}_\lambda, K)$ a polarized Hodge module of weight $w$, and
\[ c(\mc{M}) = \dim \mc{B} - w.\]
Clearly it is enough to prove the theorem in the case where $\mc{M}$ is irreducible. So we may as well assume that $\mc{M} = j_{!*}\gamma$ with $j \colon Q \hookrightarrow \mc{B}$ a $K$-orbit and $\gamma$ a $K$-equivariant $(\lambda - \rho)$-twisted local system on $Q$. We work by induction on $\dim Q$.

We first reduce to the case of relevant local systems.

\begin{lem} \label{lem:irrelevant parameters}
Assume the local system $\gamma$ on $Q$ is not relevant. Then
\[ \Gamma(\mc{B}, \Gr^Fj_{!*}\gamma) = 0.\]
Hence, \eqref{eq:hodge and signature pf 1} holds trivially for $\mc{M} = j_{!*}\gamma$.
\end{lem}
\begin{proof}
By definition, we have $\Gamma(\mc{B}, j_{!*}\gamma) = 0$ since $\gamma$ is not relevant. But by Theorem \ref{thm:filtered exactness}, the spectral sequence associated with the Hodge filtration degenerates at the first page, so this implies that $\Gamma(\mc{B}, \Gr^Fj_{!*}\gamma) = 0$ as claimed.
\end{proof}

We next note the following base case.

\begin{lem} \label{lem:continuous param zero}
If $(1 - \theta_Q)\lambda = 0$ and $\gamma$ is relevant, then \eqref{eq:hodge and signature pf 1} holds for $\mc{M} = j_{!*}\gamma$.
\end{lem}
\begin{proof}
The condition says that $j_{!*}\gamma$ is tempered. Hence, \eqref{eq:hodge and signature pf 1} holds for $\mc{M} = j_{!*}\gamma$ by Theorem \ref{thm:tempered hodge}.
\end{proof}

Suppose now that $\gamma$ is relevant and $(1 - \theta_Q)\lambda \neq 0$. In this setting, we use the techniques of \S\ref{sec:deformations} to construct an appropriate deformation of $\gamma$ to one with $(1 - \theta_Q)\lambda = 0$.

First, let $S$ denote the set of simple roots $\alpha$ such that $\langle \lambda, \check\alpha \rangle = 0$ and $\theta_Q\alpha \in \Phi_-$. By Proposition \ref{prop:relevant+tempered}, the assumption that $\gamma$ is relevant implies that these are all real with respect to the orbit $Q$ and that $\gamma$ extends cleanly to a pure Hodge module $\mc{N}$ on $Q_S:= \pi_S^{-1}\pi_S(Q)$. Since the extension to $Q_S$ is clean, writing $j_S \colon Q_S \to \mc{B}$ for the inclusion, we have
\[ j_! f\gamma = j_{S!}f\mc{N}, \quad j_*f\gamma = j_{S*}f\mc{N} \quad \mbox{and} \quad j_{!*}f\gamma = j_{S!*}f\mc{N}\]
for all $f \in \Gamma_\mb{R}^K(\tilde{Q}_S)^{\mon}$.

By construction, $\lambda \in \mf{h}^*_{S, \mb{R}, +}$, so by Proposition \ref{prop:hc deformations}, we have a unique $f \in \Gamma_\mb{R}^K(\tilde{Q}_S)^\mon_+$ such that $\varphi(f) = (1 - \theta_Q)\lambda$. We consider the associated family
\[ j_{S!*}f^s\mc{N} = j_{!*}f^s\gamma \in \mhm(\mc{D}_{\lambda + s (1 - \theta_Q)\lambda}, K). \]

\begin{lem}
Let
\[ s_0 = \min\left\{\left. s \in \left[-\frac{1}{2}, 0\right]\, \right| \,\mbox{$\lambda + s(1 - \theta_Q)\lambda$ is dominant}\right\}.\]
If $s_0 > -\frac{1}{2}$ then $f^{s_0}\gamma$ is not relevant.
\end{lem}
\begin{proof}
Write
\[ \lambda' = \lambda + s_0(1 - \theta_Q)\lambda.\]
If $s_0 > -\frac{1}{2}$ then there exists a simple root $\alpha$ such that $\langle \lambda', \check\alpha\rangle = 0$ and $\langle (1 - \theta_Q)\lambda, \check\alpha \rangle > 0$. But now
\[ \langle \lambda', \theta_Q\check\alpha \rangle = -\langle(1 - \theta_Q) \lambda', \check\alpha\rangle = -(1 + 2s_0)\langle (1 - \theta_Q)\lambda, \check\alpha \rangle < 0.\]
Hence, we must have that $\theta_Q\alpha \in \Phi_-$ and that $\alpha$ is not real. So by Proposition \ref{prop:relevant+tempered}, $f^{s_0}\gamma$ is not relevant as claimed.
\end{proof}

By Lemmas \ref{lem:irrelevant parameters} and \ref{lem:continuous param zero}, we deduce that \eqref{eq:hodge and signature pf 1} holds for $\mc{M} = j_{!*}f^{s_0}\gamma$.

We now apply Proposition \ref{prop:hyperplanes}: since $f$ is positive for $Q_S$, there exist $s_0 < s_1 < s_2 < \cdots < s_n < s_{n + 1} = 0$ such that
\[ \mbox{$j_!f^s\gamma \to j_*f^s\gamma$ is an isomorphism for $s \in (s_0, 0) - \{s_1, \ldots, s_n\}$}.\]
By continuity of eigenvalues, the signature characters $\chi^{\mathit{sig}}(j_{!*}f^s\gamma, \zeta)$ are constant on the open intervals $(s_i, s_{i + 1})$. By Theorem \ref{thm:semi-continuity} and Proposition \ref{prop:monodromic semi-continuity}, the same is true for $\Gr^F j_{!*}f^s\gamma$, and hence $\chi^H(j_{!*}f^s\gamma, u)$. So, by induction, it suffices to prove the following lemma.

\begin{lem}
Assume that \eqref{eq:hodge and signature pf 1} holds for dominantly twisted Harish-Chandra sheaves and supports of lower dimension than $Q$. Then the following are equivalent:
\begin{enumerate}
\item \eqref{eq:hodge and signature pf 1} holds for $\mc{M} = j_{!*}f^s\gamma$ with $s_i < s < s_{i + 1}$.
\item \eqref{eq:hodge and signature pf 1} holds for $\mc{M} = j_{!*}f^{s_i}\gamma$.
\item \eqref{eq:hodge and signature pf 1} holds for $\mc{M} = j_{!*}f^s\gamma$ with $s_{i - 1} < s < s_i$.
\end{enumerate}
\end{lem}
\begin{proof}
We first relate the Hodge characters. By Theorem \ref{thm:semi-continuity}, we have
\[ \chi^H(j_{!*}f^s\gamma, u) = \begin{cases}\chi^H(j_*f^{s_i}\gamma, u), & \mbox{if $s_i < s < s_{i + 1}$}, \\ \chi^H(j_!f^{s_i}\gamma, u), & \mbox{if $s_{i - 1} < s < s_{i}$}.\end{cases} \]
We also have the Jantzen filtrations $J_\bullet$ on $j_!f^{s_i}\gamma$ and $j_*f^{s_i}\gamma$ and isomorphisms
\[ (s - s_i)^n \colon \Gr^J_n j_*f^{s_i}\gamma \overset{\sim}\to \Gr^J_{-n}j_!f^{s_i}\gamma(-n) \]
for $n \geq 0$. We have $\Gr^J_0 j_!f^{s_i}\gamma = j_{!*}f^{s_i}\gamma$, and for $n > 0$, $\Gr^J_{-n} j_!f^{s_i}\gamma$ is supported on the boundary of $Q_S$. By Theorem \ref{thm:filtered exactness}, the Hodge character $\chi^H$ is additive in exact sequences. We deduce that
\begin{equation} \label{eq:hs wall crossing 1}
\chi^H(j_{!*}f^{s}\gamma, u) = \begin{cases} \chi^H(j_{!*}f^{s_i}\gamma, u) + \sum_{n > 0} u^{-n}\chi^H(\Gr^J_{-n}j_!f^{s_i}\gamma, u), &\mbox{for $s_i < s < s_{i + 1}$}, \\ \chi^H(j_{!*}f^{s_i}\gamma, u) + \sum_{n > 0} \chi^H(\Gr^J_{-n}j_!f^{s_i}\gamma, u), &\mbox{for $s_{i - 1} < s < s_i$}.\end{cases}
\end{equation}

We next relate the signature characters. By Theorem \ref{thm:jantzen}, the Hodge modules $\Gr^J_{-n}j_!f^{s_i}\gamma$ are pure of weight $w - n$ (where $w = \dim \tilde{Q}$ is the weight of $\gamma$) and the Jantzen forms are polarizations. By construction, the signature for $s_i < s < s_{i + 1}$ is that of the direct sum of the Jantzen forms, and the signature for $s_{i - 1} < s < s_i$ is that of the direct sum of the Jantzen forms with alternating sign. So we deduce that
\begin{equation} \label{eq:hs wall crossing 2}
\chi^{\mathit{sig}}(j_{!*}f^s\gamma, \zeta) = \begin{cases} \chi^{\mathit{sig}}(j_{!*}f^{s_i}\gamma, \zeta) + \sum_{n > 0} \chi^{\mathit{sig}}(\Gr^J_{-n}j_!f^{s_i}\gamma, \zeta), &\mbox{if $s_i < s < s_{i + 1}$}, \\ \chi^{\mathit{sig}}(j_{!*}f^{s_i}\gamma, \zeta) + \sum_{n > 0} \zeta^n \chi^{\mathit{sig}}(\Gr^J_{-n}j_!f^{s_i}\gamma, \zeta), &\mbox{if $s_{i - 1} < s < s_i$}.\end{cases}
\end{equation}

Applying the induction hypothesis on dimension (and recalling that the constant $c$ depends on the weight), we have
\[ \chi^{\mathit{sig}}(\Gr^J_{-n}j_!f^{s_i}\gamma, \zeta) = \zeta^{c(\Gr^J_{-n}j_!f^{s_i}\gamma)} \chi^H(\Gr^J_{-n}j_!f^{s_i}\gamma, \zeta) = \zeta^{c + n} \chi^H(\Gr^J_{-n}j_!f^{s_i}\gamma, \zeta),\]
where $c = c(j_{!*}\gamma)$. We deduce from \eqref{eq:hs wall crossing 1} and \eqref{eq:hs wall crossing 2} that
\[ \chi^{\mathit{sig}}(j_{!*}f^s\gamma, \zeta) - \zeta^c\chi^H(j_{!*}f^s\gamma, \zeta) = \chi^{\mathit{sig}}(j_{!*}f^{s_i}\gamma, \zeta) - \zeta^c\chi^H(j_{!*}f^{s_i}\gamma, \zeta) \]
for $s_{i - 1} < s < s_{i + 1}$. The statement of the lemma now follows.
\end{proof}

This concludes the proof of Theorem \ref{thm:hodge and signature K}.

\subsection{Proof in the $K'$-equivariant case} \label{subsec:pf of signature K'}

In this subsection, we explain how to Prove Theorem \ref{thm:hodge and signature K'} by extending the arguments of \S\ref{subsec:pf of signature K} to the $K'$-equivariant setting.

Let $\mc{M} \in \mhm(\tilde{\mc{D}}, K')$ be irreducible and dominantly twisted. Then $\mc{M}$ takes one of two forms: either $\mc{M} = \mc{N} \oplus \theta^*\mc{N}$, where $\mc{N} \in \mhm(\tilde{\mc{D}}, K)$ satisfies $\theta^*\mc{N} \not\cong \mc{N}$, or $\mc{M}$ is already irreducible as a $K$-equivariant module.

In the first case, the desired result follows from the $K$-equivariant case and the following obvious lemma. Here we note that we have a well-defined map
\[ (1 + \xi) \colon \mrm{Rep}(K) \to \mrm{Rep}(K'),\]
where $\xi$ is the non-trivial character of $\{1, \theta\} = K'/K$.

\begin{lem}
We have
\[ \chi^H(\mc{N} \oplus \theta^*\mc{N}, u) = (1 + \xi)\chi^H(\mc{N}, u) \]
and
\[ \chi^{\mathit{sig}}(\mc{N} \oplus \theta^*\mc{N}, \zeta) = (1 + \xi)\chi^{\mathit{sig}}(\mc{N}, \zeta).\]
\end{lem}

So we may as well assume that $\mc{M} = j_{!*}\gamma$, where $\gamma$ is a $K'$-equivariant $\lambda$-twisted local system on a $\theta$-stable $K$-orbit $Q$. In this case, we run exactly the same argument as in \S\ref{subsec:pf of signature K}. The only extra thing we need to check is that the deformations we construct are $K'$-equivariant. By Proposition \ref{prop:hc deformations} \eqref{itm:hc deformations 3}, it suffices to show that
\begin{equation} \label{eq:K'-equivariant 1}
(1 - \theta_Q)\lambda \in (\mf{h}^*_{S, \mb{R}})^{-\theta_Q} \cap (\mf{h}^*_\mb{R})^\delta.
\end{equation}
To show \eqref{eq:K'-equivariant 1}, we observe that since $\theta^*\gamma \cong \gamma$, we must have $\delta(\lambda) = \lambda$, so it is enough to show that $\theta_Q$ and $\delta$ commute. If we choose $x \in Q$, with corresponding Borel $B_x$ and a $\theta$-stable maximal torus $T \subset B_x$, then we have a commutative diagram
\[
\begin{tikzcd}
T \ar[r, "\theta"] \ar[d, "\tau_{\theta(x)}"] & T \ar[r, "\theta"] \ar[d, "\tau_x"] & T \ar[r, "\theta"] \ar[d, "\tau_x"] & T \ar[d, "\tau_{\theta(x)}"] \\
H \ar[r, "\delta"] & H \ar[r, "\theta_Q"] & H \ar[r, "\delta"] & H,
\end{tikzcd}
\]
where $\tau_x, \tau_{\theta(x)} \colon T \to H$ are the isomorphisms determined by $x$ and $\theta(x)$. Here commutativity of the left and right squares follows from the fact that the action of $\theta$ on $G$ and $\tilde{\mc{B}}$ intertwines the action of $\delta$ on $H$, and commutativity of the middle square is the definition of $\theta_Q$. Since $\theta(Q) = Q$, we have $\theta(x) \in Q$ also, so commutativity of the outer rectangle gives
\[ \theta_Q = \delta \theta_Q \delta.\]
So $\theta_Q$ and $\delta$ commute as claimed. This completes the proof of Theorem \ref{thm:hodge and signature K'}.

\end{document}